\documentclass[12pt,leqno]{amsart}
\usepackage[latin9]{inputenc}
\usepackage{verbatim}
\usepackage{amsthm}
\usepackage{amstext}
\usepackage{amssymb}
\usepackage[dvipdfmx]{graphicx}
\usepackage{xcolor}
\makeatletter
\numberwithin{equation}{section}
\numberwithin{figure}{section}

\usepackage{amsthm}\usepackage{amscd}\usepackage{bm}
\theoremstyle{plain}
\newtheorem{main theorem}{Main Theorem}
\newtheorem{theorem}{Theorem}[section]
\newtheorem{lemma}[theorem]{Lemma}
\newtheorem{conjecture}[theorem]{Conjecture}
\newtheorem{corollary}[theorem]{Corollary}
\newtheorem{proposition}[theorem]{Proposition}
\newtheorem{claim}[theorem]{Claim}

\newtheorem{property}[theorem]{Property}
\theoremstyle{definition}
\newtheorem{definition}[theorem]{Definition}
\newtheorem{remark}[theorem]{Remark}

\newtheorem{notation}[theorem]{Notation}

\newtheorem{condition}[theorem]{Condition}
\newcommand{\mdim}{\mathrm{mdim}}
\newcommand{\diam}{\mathrm{diam}}
\newcommand{\widim}{\mathrm{Widim}}

\newcommand{\supp}{\mathrm{supp}}
\newcommand{\norm}[1]{\left|\!\left|#1\right|\!\right|}

\newcommand{\sinc}{\mathrm{sinc}}
\newcommand{\id}{\mathrm{id}}
\makeatother

\begin{document}

\title[Embedding problem of $\mathbb{Z}^k$-actions]{Application of signal analysis to the embedding problem of $\mathbb{Z}^k$-actions}

\author{Yonatan Gutman, Yixiao Qiao,  Masaki Tsukamoto}



\date{\today}

\thanks{Y.G. was partially supported by the Marie Curie grant PCIG12-GA-2012-334564 and
the National Science Center (Poland) grant 2016/22/E/ST1/00448.
Y.Q. was partially supported by NNSF of China (11371339 and 11571335).
M.T. was supported by John Mung Program of Kyoto University. }

\maketitle

\begin{abstract}
We study the problem of embedding arbitrary $\mathbb{Z}^k$-actions into the shift action on 
the infinite dimensional cube $\left([0,1]^D\right)^{\mathbb{Z}^k}$.
We prove that if a $\mathbb{Z}^k$-action $X$ satisfies the marker property (in particular if $X$ is 
a minimal system without periodic points) and if its mean dimension is smaller than $D/2$ then we can 
embed it in the shift on $\left([0,1]^D\right)^{\mathbb{Z}^k}$.
The value $D/2$ here is optimal.
The proof goes through signal analysis.
We develop the theory of encoding $\mathbb{Z}^k$-actions into band-limited signals and apply it to proving the above statement.
Main technical difficulties come from higher dimensional phenomena in signal analysis.
We overcome them by exploring analytic techniques tailored to our dynamical settings.
The most important new idea is to encode the information of a tiling of 
$\mathbb{R}^k$ into a band-limited function which is constructed from 
another tiling.
\end{abstract}

\section{Introduction}

\subsection{Background}  \label{subsection: background}

The main purpose of this paper is to deepen interactions between signal analysis and 
dynamical systems. (This subsection is just a motivation. So readers can skip unfamiliar/undefined terminologies.)
If we think of signal analysis in a broad sense (like signal analysis $\approx$ Fourier analysis), then 
applications of signal analysis are ubiquitous in ergodic theory.
For example, Fourier analysis proof of the unique ergodicity of the irrational rotation (due to Weyl) is a milestone of 
such applications.
But here we would like to think of signal analysis in a narrower sense, say signal analysis is a part of communication theory.
From this viewpoint, a masterpiece is Shannon's work on \textit{communications over band-limited channels with Gaussian noise}
(\cite{Shannon}, \cite[Chapter 9]{Cover--Thomas}):
Suppose that we try to send information by using signals (say, of telephone line) whose frequencies are limited in $[-W/2,W/2]$.
If the averaged power of signal is $P$ and if
the noise is additive white Gaussian noise of spectral density $N$, then the capacity of the channel (i.e. the amount of information 
we can send by using this channel) is given by the formula
\[ \mathrm{Capacity} = \frac{W}{2} \log_2\left(1+ \frac{P}{NW}\right) \quad \text{bits per second}. \]
The details of this formula are not important here.
We just would like to emphasize that the \textbf{band-width} $W$ is a crucial parameter in the communications over band-limited channels.

Recently the first and third named authors \cite{Gutman--Tsukamoto minimal} found an analogous theory in the context of 
topological dynamics.
They started the theory of \textit{encoding arbitrary dynamical systems into band-limited signals}
and applied it to solving an open problem posed by Lindenstrauss \cite{Lindenstrauss} in 1999.
The purpose of the present paper is to expand this theory to \textit{multi-dimensional case}.

\subsection{Encoding into discrete signals} \label{subsection: embedding into discrete signals}

We start from a problem seemingly unrelated to signal analysis.
Let $k$ be a natural number.
A triple $(X,\mathbb{Z}^k,T)$ (often abbreviated to $X$) is called a dynamical system if 
$X$ is a compact metric space and 
\[ T:\mathbb{Z}^k\times X\to X, \quad (n, x) \mapsto T^n x \]
is a continuous action.
A fundamental example of dynamical systems is the \textbf{shift action on the infinite dimensional cube}.
Let $D$ be a natural number.
Consider the infinite product $\left([0,1]^D\right)^{\mathbb{Z}^k}$.
Let $\sigma: \mathbb{Z}^k\times \left([0,1]^D\right)^{\mathbb{Z}^k}\to \left([0,1]^D\right)^{\mathbb{Z}^k}$ be the shift:
\[ \sigma^m \left((x_n)_{n\in \mathbb{Z}^k}\right)  = (x_{n+m})_{n\in \mathbb{Z}^k}. \]
The triple $\left(\left([0,1]^D \right)^{\mathbb{Z}^k}, \mathbb{Z}^k, \sigma\right)$ is a dynamical system and called 
the shift on $\left([0,1]^D \right)^{\mathbb{Z}^k}$.

We study the old problem of \textit{embedding arbitrary dynamical systems in the shift on 
$\left([0,1]^D \right)^{\mathbb{Z}^k}$}.
Let $(X,\mathbb{Z}^k,T)$ be a dynamical system.
A map $f:X\to \left([0,1]^D \right)^{\mathbb{Z}^k}$ is called an embedding of a dynamical system if 
$f$ is a $\mathbb{Z}^k$-equivariant continuous injection.
We would like to understand when $X$ can be embedded in the shift on $\left([0,1]^D \right)^{\mathbb{Z}^k}$.
Notice that every compact metric space can be \textit{topologically} embedded in $\left([0,1]^D \right)^{\mathbb{Z}^k}$.
So the problem is a genuinely \textit{dynamical} question.

It will be convenient later to see the problem from a slightly different viewpoint:
A point $x= (x_n)_{n\in \mathbb{Z}^k}$ in $\left([0,1]^D \right)^{\mathbb{Z}^k}$ can be seen as a 
\textit{discrete signal}\footnote{When $k=2$, it might be better to call $x$ a \textit{discrete image}.}
valued in $[0,1]^D$.
Then the embedding problem asks \textit{when we can encode a given dynamical system $X$ into discrete signals}.

We review the history of the problem
before explaining our main result (Main Theorem \ref{main theorem: discrete signal} below).
The embedding problem was first studied by Jaworski \cite{Jaworski} in 1974. 
For understanding his result, notice that periodic points form an obvious obstruction to the embedding.
For example, the shift on $\left([0,1]^2\right)^{\mathbb{Z}^k}$ cannot be embedded in the shift on 
$[0,1]^{\mathbb{Z}^k}$ because the fixed points sets of $\left([0,1]^2\right)^{\mathbb{Z}^k}$ and $[0,1]^{\mathbb{Z}^k}$
are homeomorphic to the square $[0,1]^2$ and the line segment $[0,1]$ respectively and 
the square cannot be topologically embedded in the line segment.

A dynamical system $(X,\mathbb{Z}^k,T)$ is said to be \textbf{aperiodic} if 
$T^n x\neq x$ for all $x\in X$ and nonzero $n\in \mathbb{Z}^k$.
Jaworski \cite{Jaworski} proved that periodic points are the only obstruction if $X$ is a finite dimensional system:

\begin{theorem}[Jaworski, 1974]  \label{theorem: Jaworski}
Let $(X,\mathbb{Z}, T)$ be an aperiodic finite dimensional dynamical system.
Then we can embed it in the shift on $[0,1]^{\mathbb{Z}}$.
\end{theorem}

Although this theorem is stated only for $\mathbb{Z}$-actions, it can be easily generalized to $\mathbb{Z}^k$-actions.
(Indeed the $\mathbb{Z}^k$-case can be deduced from the $\mathbb{Z}$-case.)

After Jaworski, people were interested in whether the assumption of finite dimensionality is essential or not.
Auslander \cite[p. 193]{Auslander} asked whether we can embed every minimal system 
$(X,\mathbb{Z},T)$ in the shift on $[0,1]^{\mathbb{Z}}$.
A system $(X,\mathbb{Z}^k,T)$ is said to be \textbf{minimal} if the orbit $\{T^n x\}_{n\in \mathbb{Z}^k}$
is dense in $X$ for every $x\in X$.
When $k=1$, minimal systems $X$ have no periodic points unless $X$ is a finite set.
(If $X$ is finite, then it can be obviously embedded in the shift on $[0,1]^\mathbb{Z}$.)
Therefore the question essentially asks whether there is another obstruction different from periodic points.

Lindenstrauss--Weiss \cite{Lindenstrauss--Weiss} found that mean dimension provides a new obstruction to the embedding.
Mean dimension (first introduced by Gromov \cite{Gromov}) is a topological invariant of dynamical systems which 
counts the average number of parameters of a given system.
The mean dimension of $(X,\mathbb{Z}^k,T)$ is denoted by $\mdim(X)$.
We review its definition in Subsection \ref{subsection: review of mean dimension}.
Finite dimensional systems and finite topological entropy systems are known to have zero mean dimension.

The mean dimension of the shift on $\left([0,1]^D \right)^{\mathbb{Z}^k}$ is equal to $D$, which means that 
$\left([0,1]^D \right)^{\mathbb{Z}^k}$ has $D$ parameters in average.
If $(X,\mathbb{Z}^k,T)$ can be embedded in the shift on $\left([0,1]^D \right)^{\mathbb{Z}^k}$ then $\mdim(X)\leq D$.
Lindenstrauss--Weiss \cite[Proposition 3.5]{Lindenstrauss--Weiss} constructed a minimal system 
$(X,\mathbb{Z},T)$ of mean dimension strictly greater than one.
This system cannot be embedded in the shift on $[0,1]^\mathbb{Z}$ although it is minimal.
Thus it solved Auslander's question.

Lindenstrauss went further; he proved a partial converse \cite[Theorem 5.1]{Lindenstrauss}.

\begin{theorem}[Lindenstrauss, 1999]
If a minimal system $(X,\mathbb{Z},T)$ satisfies $\mdim(X)<D/36$ then we can embed it in the shift on 
$\left([0,1]^D\right)^{\mathbb{Z}}$.
\end{theorem}

Lindenstrauss \cite[p. 229]{Lindenstrauss} asked the problem of improving the condition $\mdim(X)<D/36$.
This problem was solved by the first and third named authors \cite[Theorem 1.4]{Gutman--Tsukamoto minimal}:

\begin{theorem}[Gutman--Tsukamoto]  \label{theorem: Gutman--Tsukamoto}
If a minimal system $(X,\mathbb{Z},T)$ satisfies $\mdim(X)<D/2$ then we can embed it in the shift on 
$\left([0,1]^D \right)^{\mathbb{Z}}$.
\end{theorem}

The condition $\mdim(X)<D/2$ is optimal because there exists a minimal system $(X,\mathbb{Z},T)$ of mean dimension 
$D/2$ which cannot be embedded in the shift on $\left([0,1]^D \right)^{\mathbb{Z}}$ (\cite{Lindenstrauss--Tsukamoto}).
Theorem \ref{theorem: Gutman--Tsukamoto} can be seen as a dynamical analogue of the classical theorem in dimension theory
\cite[Thm. V2]{Hurewicz--Wallman}:
A compact metric space $X$ can be topologically embedded in $[0,1]^D$ if $\dim X < D/2$.
We review the proof of this classical result in Subsection \ref{subsection: technical results on simplicial complexes}.

A motivation of the present paper is to generalize Theorem \ref{theorem: Gutman--Tsukamoto} to 
$\mathbb{Z}^k$-actions.
It is convenient to introduce the following notion.
A dynamical system $(X,\mathbb{Z}^k,T)$ is said to satisfy the \textbf{marker property} if for every natural number $N$ there exists 
an open set $U\subset X$ satisfying 
\[   U\cap T^{-n}U = \emptyset \quad (0<|n|<N), \quad 
     X = \bigcup_{n\in \mathbb{Z}^k} T^{-n} U. \]
This property obviously implies the aperiodicity.     
The following dynamical systems are known to satisfy the marker property:
\begin{itemize}
     \item  Aperiodic minimal systems. (This is an immediate fact from the definition.)
     \item  Aperiodic finite dimensional systems 
              (\cite[Theorem 6.1]{Gutman 2})\footnote{More generally, we can prove that if an aperiodic dynamical system satisfies the 
              \textit{small boundary property} then it satisfies the marker property. The small boundary property is a notion introduced by 
              Lindenstrauss--Weiss \cite[Definition 5.2]{Lindenstrauss--Weiss}, which is satisfied by every aperiodic finite dimensional system. 
              We do not use these facts. So we omit the detailed explanation.}.
     \item  If $X$ satisfies the marker property and if $\pi:Y\to X$ is an extension (i.e. $\mathbb{Z}^k$-equivariant continuous 
               surjection) then $Y$ also satisfies the marker property.
\end{itemize}
The authors do not know an example of aperiodic actions which do not satisfy the marker property.

Lindenstrauss and the first and third named authors \cite[Theorem 1.5]{Gutman--Lindenstrauss--Tsukamoto} proved:

\begin{theorem}[Gutman--Lindenstrauss--Tsukamoto, 2016]  \label{theorem: Gutman--Lindenstrauss--Tsukamoto}
If a dynamical system $(X,\mathbb{Z}^k,T)$ satisfies the marker property and 
\[ \mdim(X) < \frac{D}{2^{k+1}} \]
then we can embed it in the shift on $\left([0,1]^{2D}\right)^{\mathbb{Z}^k}$.
\end{theorem}

After stating this theorem, the paper \cite{Gutman--Lindenstrauss--Tsukamoto} 
posed the problem of improving the condition $\mdim(X)< D/2^{k+1}$.
Quoting \cite[p. 782]{Gutman--Lindenstrauss--Tsukamoto}:
\begin{quote}
Note also that in our condition $\mdim(X)<D/2^{k+1}$, the constant involved is likely far from optimal. Presumably for an aperiodic 
$\mathbb{Z}^k$-system if $\mdim(X) < \frac{D}{2}$ then $X$ can be embedded in  $([0,1]^D)^{\mathbb{Z}^k}$.
\end{quote}

The first main theorem of the present paper confirms this conjecture under the assumption of the marker property:

\begin{main theorem} \label{main theorem: discrete signal}
If a dynamical system $(X,\mathbb{Z}^k,T)$ satisfies the marker property and 
\[  \mdim(X) < \frac{D}{2},   \]
then we can embed it in the shift on $\left([0,1]^D\right)^{\mathbb{Z}^k}$.
\end{main theorem}

As in the case of Theorem \ref{theorem: Gutman--Tsukamoto}, the condition $\mdim(X)< D/2$ is optimal because 
the paper \cite{Lindenstrauss--Tsukamoto}\footnote{Strictly speaking, the paper \cite{Lindenstrauss--Tsukamoto}
considered only the case of $k=1$. But their construction can be generalized to $\mathbb{Z}^k$-actions without any changes.} 
provides an example of aperiodic minimal system $(X,\mathbb{Z}^k,T)$ 
of mean dimension $D/2$ which cannot be embedded in the shift on $\left([0,1]^D\right)^{\mathbb{Z}^k}$.
Main Theorem \ref{main theorem: discrete signal} has some novelty even in the case of $k=1$:
Since both aperiodic finite dimensional systems and aperiodic minimal systems satisfy the marker property,
Main Theorem \ref{main theorem: discrete signal} unifies Jaworski's theorem (Theorem \ref{theorem: Jaworski})
and Theorem \ref{theorem: Gutman--Tsukamoto} into a single statement.

\subsection{Encoding into band-limited signals}  \label{subsection: embedding into band-limited signals}

A discovery of the paper \cite{Gutman--Tsukamoto minimal}
is that we can approach to the embedding problem in Subsection \ref{subsection: embedding into discrete signals}
via signal analysis.
We need to prepare some terminologies.
For a rapidly decreasing function $\varphi:\mathbb{R}^k\to \mathbb{C}$ we define its Fourier transforms by 
\begin{equation*}
  \begin{split}
    \mathcal{F}(\varphi)(\xi) = \hat{\varphi}(\xi) = \int_{\mathbb{R}^k} \varphi(t) \, e^{-2\pi \sqrt{-1} t\cdot \xi} \, dt_1\dots dt_k, \\
    \overline{\mathcal{F}}(\varphi)(t) = \check{\varphi}(t) = \int_{\mathbb{R}^k} \varphi(\xi)\,  e^{2\pi \sqrt{-1} t\cdot \xi} d\xi_1\dots d\xi_k.   
   \end{split}
\end{equation*}    
It follows that $\mathcal{F}\left(\overline{\mathcal{F}}(\varphi)\right) = \overline{\mathcal{F}}\left(\mathcal{F}(\varphi)\right) = \varphi$.
We extend $\mathcal{F}$ and $\overline{\mathcal{F}}$ to tempered distributions $\psi$ by the dualities
$\langle \mathcal{F}(\psi), \varphi \rangle = \langle \psi, \overline{\mathcal{F}}(\varphi)\rangle$
and $\langle \overline{\mathcal{F}}(\psi), \varphi \rangle = \langle \psi, \mathcal{F}(\varphi)\rangle$ 
where $\varphi$ are rapidly decreasing functions (Schwartz \cite[Chapter 7]{Schwartz}).
For example, if $\psi(t) = e^{2\pi\sqrt{-1}a\cdot t}$ $(a\in \mathbb{R}^k)$ then $\mathcal{F}(\psi)=\delta_a$ 
is the delta probability measure 
at $a$.

Let $a_1,\dots,a_k$ be positive numbers.
A bounded continuous function $\varphi:\mathbb{R}^k \to \mathbb{R}$ is said to be band-limited in 
$[-a_1/2,a_1/2]\times\dots \times [-a_k/2,a_k/2]$ if the Fourier transform $\hat{\varphi}$ satisfies 
\[  \supp\, \hat{\varphi} \subset  \left[-\frac{a_1}{2}, \frac{a_1}{2}\right]\times \dots \times \left[-\frac{a_k}{2}, \frac{a_k}{2}\right],   \]
which means that $\langle \hat{\varphi}, \phi \rangle = \langle \varphi, \check{\phi}\rangle = 0$ for all rapidly decreasing functions $\phi$
satisfying 
\[  \supp \, \phi \cap  \left[-\frac{a_1}{2}, \frac{a_1}{2}\right]\times \dots \times \left[-\frac{a_k}{2}, \frac{a_k}{2}\right] = \emptyset. \]
For example, when $k=1$, the functions 
\[ \sin(\pi t_1), \quad \cos(\pi t_1), \quad \frac{\sin(\pi t_1)}{\pi t_1} \]
are band-limited in $[-1/2,1/2]$.

We define $\mathcal{B}(a_1,\dots,a_k)$ as the Banach space of all bounded continuous functions $\varphi:\mathbb{R}^k\to \mathbb{R}$
band-limited in $[-a_1/2,a_1/2]\times \dots \times [-a_k/2,a_k/2]$.
Its norm is the $L^\infty$-norm.
We define $\mathcal{B}_1(a_1,\dots, a_k)$ as the set of $\varphi\in \mathcal{B}(a_1,\dots,a_k)$ satisfying 
$\norm{\varphi}_{L^\infty(\mathbb{R}^k)}\leq 1$.
In this paper we promise that the space $\mathcal{B}_1(a_1,\dots,a_k)$ is always endowed with the
\textbf{topology of uniform convergence over compact subsets}.
Namely $\mathcal{B}_1(a_1,\dots,a_k)$ is endowed with the topology given by the distance 
\[ d(\varphi,\psi) = \sum_{n=1}^\infty 2^{-n} \norm{\varphi-\psi}_{L^\infty(B_n)},  \quad 
   \left(B_n = \{t\in \mathbb{R}^k|\, |t|\leq n\}\right). \]
$\mathcal{B}_1(a_1,\dots,a_k)$ is compact (see Corollary \ref{cor: compactness of B_1} below) and endowed with a 
natural continuous action of $\mathbb{Z}^k$ defined by 
\[ \sigma:\mathbb{Z}^k\times \mathcal{B}_1(a_1,\dots,a_k)\to \mathcal{B}_1(a_1,\dots,a_k), \quad 
   \sigma^n(\varphi)(t) = \varphi(t+n). \] 
We call the dynamical system $(\mathcal{B}_1(a_1,\dots,a_k), \mathbb{Z}^k,\sigma)$ the \textbf{shift on $\mathcal{B}_1(a_1,\dots,a_k)$}.
Its mean dimension is equal to the product\footnote{This fact is not used in the paper. So we omit the proof.} 
\[  a_1\dots a_k, \]
which is a multi-dimensional version of band-width $W$
introduced in Subsection \ref{subsection: background}.
As band-width $W$ plays a crucial role in Shannon's theory, the quantity $a_1\dots a_k$ becomes a key parameter below.

We consider the problem of embedding arbitrary dynamical systems in the shift on $\mathcal{B}_1(a_1,\dots,a_k)$.
In other words we ask \textit{when we can encode a given dynamical system into band-limited signals}.
When $k=1$, the first and third named authors \cite[Theorem 1.7]{Gutman--Tsukamoto minimal} proved:

\begin{theorem}[Gutman--Tsukamoto] \label{theorem: Gutman--Tsukamoto continuous signal}
If an aperiodic minimal system $(X,\mathbb{Z},T)$ satisfies $\mdim(X)< a_1/2$ then we can embed it in the shift on 
$\mathcal{B}_1(a_1)$.
\end{theorem}

The second main theorem of the present paper generalizes this theorem to the higher dimensional case:

\begin{main theorem}  \label{main theorem: continuous signal}
If a dynamical system $(X,\mathbb{Z}^k,T)$ satisfies the marker property and  
\[  \mdim(X) < \frac{a_1 \dots a_k}{2},  \]
then we can embed it in the shift on $\mathcal{B}_1(a_1,\dots,a_k)$.
\end{main theorem}

As in the case of Main Theorem \ref{main theorem: discrete signal}, 
the condition $\mdim(X)<a_1\dots a_k/2$ is optimal because 
a modification of \cite{Lindenstrauss--Tsukamoto} provides an 
aperiodic minimal system $(X,\mathbb{Z}^k,T)$ of mean dimension $a_1\dots a_k/2$ which cannot be embedded in the shift on 
$\mathcal{B}_1(a_1,\dots,a_k)$.

We can prove Main Theorem \ref{main theorem: discrete signal} by using Main Theorem \ref{main theorem: continuous signal}.

\begin{proof}[Proof: Main Theorem \ref{main theorem: continuous signal} implies Main Theorem \ref{main theorem: discrete signal}]
From $\mdim(X)<D/2$ we can choose positive numbers $a_1,\dots, a_k$ satisfying 
\[ a_1<D,  \, a_2<1,\,  \dots, a_k<1, \quad \mdim(X) < \frac{a_1\dots a_k}{2}. \]
Main Theorem \ref{main theorem: continuous signal} implies that we can embed $X$ in the shift on $\mathcal{B}_1(a_1,\dots,a_k)$.
We define a lattice $\Lambda\subset \mathbb{R}^k$ by 
\[  \Lambda = \left\{\left(\frac{n_1}{D}, n_2, \dots, n_k\right) \middle| \, n_1,n_2\dots, n_k\in \mathbb{Z}\right\}. \]
It follows from a sampling theorem (see Lemma \ref{lemma: sampling theorem} below) that the map 
\[   \mathcal{B}_1(a_1,\dots,a_k) \to [-1, 1]^\Lambda, \quad \varphi \mapsto \varphi|_\Lambda \]
is injective.   
Set $e=(1/D, 0, \dots,0)\in \mathbb{R}^k$.
The above injectivity means that the $\mathbb{Z}^k$-equivariant map 
\[  \mathcal{B}_1(a_1,\dots,a_k) \to \left([-1,1]^D\right)^{\mathbb{Z}^k}, \quad 
    \varphi \mapsto \left(\varphi(n), \varphi(n+e), \dots, \varphi \left(n+(D-1)e\right) \right)_{n\in \mathbb{Z}^k} \]
is an embedding.    
Since the system $X$ can be embedded in the shift on $\mathcal{B}_1(a_1,\dots,a_k)$, it can be also embedded    
in the shift on $\left([-1,1]^D \right)^{\mathbb{Z}^k}$, which is
obviously isomorphic to the shift on $\left([0,1]^D \right)^{\mathbb{Z}^k}$.
\end{proof}

\subsection{One dimension versus multi-dimension; what are the main difficulties?}
\label{subsection: one dimension versus multi-dimension}

This subsection explains what are the main difficulties in the proof of Main Theorem \ref{main theorem: continuous signal}.
When the authors started the research of this paper, 
they optimistically thought that the proof of Main Theorem \ref{main theorem: continuous signal}
is probably more or less a direct generalization of the proof of Theorem \ref{theorem: Gutman--Tsukamoto continuous signal} in 
\cite{Gutman--Tsukamoto minimal}.
This expectation turned out to be wrong.
Completely new difficulties arose in the context of signal analysis.
The paper \cite{Gutman--Tsukamoto minimal} used the technique of \textit{one dimensional} signal analysis.
The proof of Main Theorem \ref{main theorem: continuous signal} uses \textit{mutli-dimensional} signal analysis.
The difference of one dimension/multi-dimension is fundamental and causes big changes of the approaches significantly.

We review some basic signal analysis in Subsection \ref{subsection: review of band-limited functions} below.
The main results there are Paley--Wiener's theorem (Lemma \ref{lemma: Paley--Wiener})
and a sampling theorem (Lemma \ref{lemma: sampling theorem}) for functions in $\mathcal{B}(a_1,\dots,a_k)$, whose
proofs are a simple generalization of (or a reduction to) the corresponding theorems in one-variable case.
They might give readers an impression that multi-dimensional signal analysis is just a simple generalization of 
one-dimensional signal analysis.
This is far from being true.
It is well-known in classical analysis community (cf. \cite{Olevskii--Ulanovskii}) that 
\textit{multi-dimensional signal analysis is inherently more difficult than one-dimensional signal analysis}.
Some important theorems in one-dimensional case \textit{cannot} be generalized to higher dimension.

For explaining further, we need to recall the statement of Paley--Wiener's theorem (Lemma \ref{lemma: Paley--Wiener}):
A bounded continuous function $f:\mathbb{R}^k \to \mathbb{R}$ belongs to $\mathcal{B}(a_1,\dots,a_k)$ if and only if 
it can be extended to a holomorphic function in $\mathbb{C}^k$ satisfying 
\[ \left|f\left(x_1+y_1\sqrt{-1},\dots,x_k+y_k\sqrt{-1}\right)\right| \leq \mathrm{const} \cdot 
    e^{\pi\left(a_1|y_1|+\dots+a_k|y_k|\right)}.  \]
Roughly speaking, band-limited functions are the same as \textit{holomorphic functions of exponential type}.
Therefore we can say that signal analysis is a part of complex analysis\footnote{Of course this is over-simplified.
Signal analysis is a mutli-faceted discipline and cannot be classified as a sub-area of complex analysis.
For example, we emphasized its communication theory aspect in Subsection \ref{subsection: background}.
But the complex analysis viewpoint is convenient here.}.
Then it is easy to see why multi-dimensional case is more difficult than one-dimensional case:
Zero points of holomorphic functions in $\mathbb{C}$ are isolated, whereas zero points of holomorphic functions in 
$\mathbb{C}^k$ $(k\geq 2)$ form positive dimensional complex varieties whose geometry are highly nontrivial in general.
This causes a fundamental difference in the nature of one-dimensional/multi-dimensional signal analysis.

More concretely speaking, we face difficulties of higher dimension in the following two issues.
(The second issue is more serious than the first one.)

\begin{description}
    \item[(1) Interpolation]
     \textit{Sampling and interpolation} are the most basic themes of signal analysis.
    We recall a sampling theorem stated in Lemma \ref{lemma: sampling theorem} with a bit simplification\footnote{Here we state 
    a (seemingly) simpler version of the theorem for clarifying the argument, but indeed this statement is equivalent to 
    Lemma \ref{lemma: sampling theorem}.}: Let $f\in \mathcal{B}(1,\dots,1)$ and $0<c<1$. If $f$ vanishes over
    $\Lambda \overset{\mathrm{def}}{=} c\mathbb{Z}^k$ then $f$ is identically zero.
    This theorem is valid for all dimensions.
    The crucial point of the statement is that we consider only \textbf{regular sampling}, which means that 
    the fundamental domain $[0,c]^k$ of the sampling set $\Lambda = c\mathbb{Z}^k$ is similar to the 
    \textit{frequency domain} $[-1/2,1/2]^k$. 
    It is also easy to prove the following \textit{regular} interpolation theorem:
    If $c>1$ then the map 
    \[  \mathcal{B}(1,\dots,1)\to \ell^\infty(c\mathbb{Z}^k), \quad f\mapsto \left(f(\lambda)\right)_{\lambda\in c\mathbb{Z}^k} \]
    is surjective.

    When $k=1$, Beurling \cite{Beurling} proved that essentially the same results also hold for 
    \textit{irregular} sampling/interpolation:
    Let $\Lambda\subset \mathbb{R}$ be a uniformly 
    discrete set, i.e. the infimum of $|\lambda_1-\lambda_2|$ over distinct $\lambda_1,\lambda_2\in \Lambda$ is positive.
    Suppose the \textbf{lower density}
    \[   \liminf_{R\to \infty} \inf_{t\in \mathbb{R}}\frac{\#(\Lambda\cap [t,t+R])}{R} \]
    is larger than $1$. 
    Under this setting, Beurling proved that if $f\in \mathcal{B}(1)$ vanishes over $\Lambda$ then $f$ is identically zero.
    Similarly, he also proved that if $\Lambda\subset \mathbb{R}$ is a uniformly discrete set and its \textbf{upper density} 
    \[  \limsup_{R\to \infty} \sup_{t\in \mathbb{R}} \frac{\#(\Lambda\cap [t,t+R])}{R}  \]
    is smaller than $1$, then the map 
    \[  \mathcal{B}(1) \to \ell^\infty(\Lambda), \quad f\mapsto \left(f(\lambda)\right)_{\lambda\in \Lambda} \]
    is surjective.
    Therefore only the lower/upper density of $\Lambda$ determines its sampling/interpolation property.
    
    When $k\geq 2$, the situation is much more messy.
    We have no clear theorems valid for irregular sampling and interpolation 
    (see \cite{Olevskii--Ulanovskii} for the details).
    For example, consider 
    \[  \Lambda_1 = \left\{\left(\varepsilon^2\,  t_1, \frac{t_2}{\varepsilon}\right)\middle|\, t_1,t_2\in \mathbb{Z}\right\}, \quad 
        \Lambda_2 =  \left\{\left(\varepsilon\,  t_1, \frac{t_2}{\varepsilon^2}\right)\middle|\, t_1,t_2\in \mathbb{Z}\right\} \quad 
         (\varepsilon >0). \]
    Suppose $\varepsilon$ is very small. The density of $\Lambda_1$ is equal to $1/\varepsilon$ (very large) and 
    the density of $\Lambda_2$ is equal to $\varepsilon$ (very small).
    The function $f(z_1,z_2) = \sin (\pi \varepsilon z_2) \in \mathcal{B}(1,1)$ vanishes over $\Lambda_1$ but is not identically zero.
    Thus $\Lambda_1$ is not a \textit{sampling set} for functions in $\mathcal{B}(1,1)$ although its density is 
    very large. 
    Similarly, it is also easy to prove that $\Lambda_2$ is not an \textit{interpolation set} for $\mathcal{B}(1,1)$, namely the map
    \[  \mathcal{B}(1,1) \to \ell^\infty(\Lambda_2), \quad f\mapsto  \left(f(\lambda)\right)_{\lambda\in \Lambda_2} \]
    is not surjective, although its density is very small.
    
    In the proof of Main Theorem \ref{main theorem: continuous signal} we need to construct interpolating functions for 
    some \textit{irregular} sets $\Lambda\subset \mathbb{R}^k$.
    The paper \cite{Gutman--Tsukamoto minimal} addressed the same problem (in dimension one) 
    by employing Beurling's interpolation theorem 
    mentioned above. But this theorem is not valid in higher dimension.

    \item[(2) Zero points of band-limited maps $f:\mathbb{C}^k \to \mathbb{C}^k$]
     One crucial ingredient of the proof of Main Theorem \ref{main theorem: continuous signal} is the study of (isolated) 
     zero points\footnote{A point $p\in \mathbb{C}^k$ is a zero point of $f:\mathbb{C}^k\to \mathbb{C}^k$ if $f(p)=0$.
     Notice that the domain and target of $f$ have the same dimension. Therefore \textit{generically} zero points of $f$ 
     are expected to be isolated.} 
     of some holomorphic maps $f = (f_1,\dots, f_k):\mathbb{C}^k\to \mathbb{C}^k$ all of whose entries $f_i$ are band-limited functions.
     We call such a map $f$ a \textbf{band-limited map}.
     It becomes important to know the growth of 
     \begin{equation}   \label{eq: growth of zero points}
        \{z\in \mathbb{C}^k|\, \text{$z$ is an isolated zero point of $f$ with $|z|<R$}\}  
     \end{equation}   
     as $R$ goes to infinity.
     If $k=1$, then this is a very easy problem. 
     It follows from \textit{Jensen's formula} or \textit{Nevanlinna's first main theorem} (\cite[Section 1.3]{Hayman},
     \cite[Section 1.1]{Noguchi--Winkelmann}) that if $f\in B(a)$ then 
    \[  \int_1^R \#\{z\in \mathbb{C}|\, f(x)=0, |z|<r\} \frac{dr}{r} \leq 2aR + \mathrm{const}_f. \]
     In particular (\ref{eq: growth of zero points}) grows at most linearly in $R$.
     Moreover if $f\in \mathbb{B}(a)$ has no zero points, then it must be of the form 
     \[ f(z) = e^{\alpha z + \beta} \]
     for some constants $\alpha$ and $\beta$.
     
     When $k\geq 2$, the situation is radically different.
     First, we have no simple description of $f:\mathbb{C}^k\to \mathbb{C}^k$ having no zero points. 
     For example, the map of the form 
     \[  f(z_1, z_2) = (e^{\alpha z_1+\beta}, f_2(z_1, z_2)), \quad (\text{$f_2(z_1,z_2)$: arbitrary band-limited function}),  \]
     have no zero points.
     Second (and more importantly), the set (\ref{eq: growth of zero points}) may have \textit{arbitrarily fast growth}.
     This inconvenient phenomena was first observed by Cornalba--Shiffman \cite{Cornalba--Shiffman} and known as 
     \textit{the failure of the transcendental Bezout theorem}.
     Here we briefly review their construction (with minor modification).
     Let $\{\alpha_n\}_{n=1}^\infty$ be an arbitrary sequence of positive integers (say, $\alpha_n = 2^{2^n}$).
     We choose a polynomial $p_n(w)$ (of one-variable) having $\alpha_n$ zeros in $|w|<1$.
     Let $\varphi:\mathbb{R}\to \mathbb{R}$ be a  (nonzero) rapidly decreasing function satisfying 
     $\supp \, \hat{\varphi} \subset [-1/2, 1/2]$.
     We choose $\beta_n>0$ so that the function 
     \[ g_n(w) \overset{\mathrm{def}}{=} \beta_n\,  \varphi(w) p_n(w) \]
     is bounded by $2^{-n}$ over $\mathbb{R}$.
     We define $f = (f_1,f_2):\mathbb{C}^2\to \mathbb{C}^2$ as follows:
     \begin{equation*}
          f_1(z_1,z_2) = \sin (\pi z_1), \quad
          f_2(z_1,z_2) = \sum_{n=1}^\infty g_n(z_2) \frac{\sin \pi(z_1-n)}{\pi(z_1-n)}.
     \end{equation*}
     Both $f_1$ and $f_2$ belong to $\mathcal{B}(1,1)$.
     The map $f$ has $\alpha_n$ zero points of the form $(n,w)$ $(|w|<1)$.
     Therefore 
     \[  \#\{z\in \mathbb{C}^2|\, \text{$z$ is an isolated zero point of $f$ and $|z|<n+1$}\} \geq \alpha_n. \]
     
     As a conclusion, it is hard (and sometimes impossible) to control the set (\ref{eq: growth of zero points})
     in higher dimension.
\end{description}

The above two issues show that general theory of multi-dimensional signal analysis is not strong enough for the proof of
Main Theorem \ref{main theorem: continuous signal}.
Thus we have to develop \textit{tailored methods} specific for the situation of the theorem.
This becomes the main technical task of the paper.
We address the issues (1) and (2) in Sections \ref{section: Interpolating functions} and 
\ref{section: dynamical construction of tiling-like band-limited functions} respectively.
The techniques developed there seem to have some independent interests, and possibly further applications 
will be found in a future.

The key ideas are as follows. (Here we ignore many details for simplicity. The real arguments are different.)

\begin{description}
      \item[(1)' Almost regular interpolation]
       As we mentioned above, we need to construct interpolation for some irregular sets $\Lambda\subset \mathbb{R}^k$.
       But our sets $\Lambda$ are not completely general uniformly discrete sets.
       They are ``almost'' regular sets in the sense that we have a decomposition 
       \[ \Lambda = \bigcup_{n=1}^\infty \Lambda_n \]
       such that each $\Lambda_n$ is a part of a regular interpolation set and that the distances between different pieces $\Lambda_m$
       and $\Lambda_n$ are sufficiently large.
       See Figure \ref{fig: interpolation}.
       We develop a technique of constructing interpolation for such almost regular sets.
       As a conclusion, we can prove that they have (essentially) all the nice properties which regular interpolation sets possess.

     \item[(2)' Good/bad decomposition of the space]
      It is very difficult in general to control the zero points of holomorphic maps 
      $f:\mathbb{C}^k\to \mathbb{C}^k$.
      The maps $f$ used in the proof of Main Theorem \ref{main theorem: continuous signal}
      are very special, but still we cannot control all their (isolated) zero points.
      Therefore we simply give up to control \textit{all} the zero points of $f$.
      Instead we decompose the space $\mathbb{C}^k$ into two disjoint regions:
      \[  \mathbb{C}^k = \mathbb{G} \cup \mathbb{B}. \]
      In the ``good region''$\mathbb{G}$, the zero points of $f$ are very sparsely distributed
      and we can control them as we like.
      In the ``bad region'' $\mathbb{B}$, the zero points of $f$ may have extremely high density,
      but the region $\mathbb{B}$ itself is very tiny.
     See Figure \ref{fig: good bad decomposition}.
     We consider a \textit{Voronoi diagram} with respect to the zero points of $f$.
     Then most of the space $\mathbb{C}^k$ is cover by very large (and hence well organized) tiles whose 
     Voronoi centers are zero points in $\mathbb{G}$.
     We don't have any control of tiles whose centers are located in $\mathbb{B}$.
     They may have very complicated structure.
     But every messy phenomena is confined in the tiny region $\mathbb{B}$ and does not affect much the whole picture.

\end{description}

\begin{figure}
    \centering
    \includegraphics[,bb= 0 0 500 160]{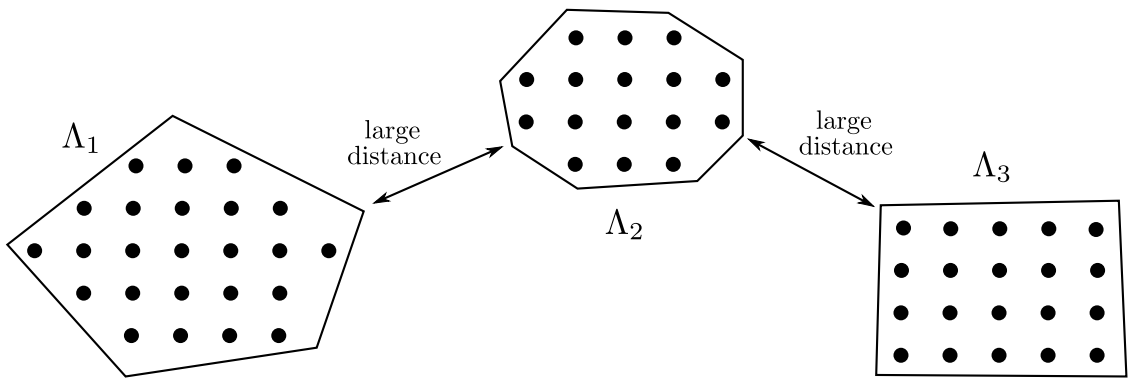}
    \caption{Almost regular interpolation set $\Lambda$. Different pieces $\Lambda_m$ and $\Lambda_n$ are sufficiently far from each other.} 
    \label{fig: interpolation}
\end{figure}

\begin{figure}
    \centering
    \includegraphics[,bb= 0 0 400 180]{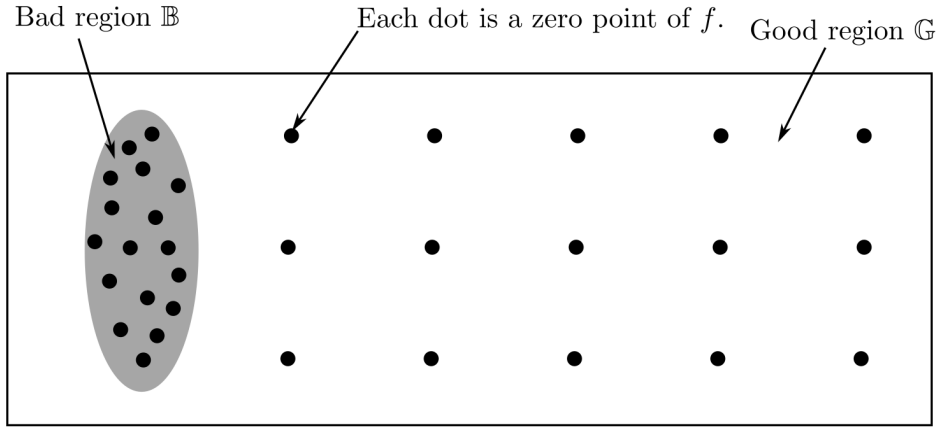}
    \caption{Schematic picture of the good/bad decomposition. 
    Zero points of $f$ are sparsely and almost regularly distributed inside the good region $\mathbb{G}$. The bad region $\mathbb{B}$
    may contain plenty of zero points but it is tiny.} 
    \label{fig: good bad decomposition}
\end{figure}

In Subsection \ref{subsection: main proposition} we state a key proposition (Proposition \ref{prop: main proposition}) 
and prove Main Theorem 
\ref{main theorem: continuous signal} assuming it.
The rest of the paper is devoted to the proof of Proposition \ref{prop: main proposition}.
The strategy of the proof is explained in Subsection \ref{subsection: strategy of the proof}.
It might be better for some readers to go to Subsections \ref{subsection: main proposition} and \ref{subsection: strategy of the proof}
before reading Section \ref{section: preliminaries}.

\subsection{Organization of the paper}  \label{subsection: organization of the paper}

In Section \ref{section: preliminaries} we review mean dimension, band-limited functions, 
simplicial complexes and convex sets.
In Section \ref{section: main proposition}
we state the main proposition and prove Main Theorem \ref{main theorem: continuous signal}
assuming it. In particular, we give an overview of the proof of the main theorem.
In Section \ref{section: Interpolating functions} we construct interpolating functions.
In Section \ref{section: dynamical construction of tiling-like band-limited functions} we construct 
certain ``tiling-like maps'' and study their properties.
In Section \ref{section: proof of Main Proposition} we prove the main proposition by using the results in 
Sections \ref{section: preliminaries}, \ref{section: Interpolating functions} and \ref{section: dynamical construction of 
tiling-like band-limited functions}.
In Section \ref{section: open problems} we explain some open problems.

\vspace{0.2cm}

\textbf{Acknowledgement.}
This paper was written when the third named author stayed in the Einstein Institute of Mathematics in the Hebrew University of Jerusalem.
He would like to thank the institute for its hospitality.

\section{Preliminaries}   \label{section: preliminaries}

\subsection{Review of mean dimension}  \label{subsection: review of mean dimension}

Here we review the definition of mean dimension \cite{Gromov, Lindenstrauss--Weiss}.
Throughout this paper we promise that every simplicial complex is finite (i.e. it has only finitely many simplicies) and that 
for a natural number $N$
\[ [N] = \{0,1,2,\dots, N-1\}^k.  \]

Let $(X,d)$ be a compact metric space and $Y$ a topological space.
Let $\varepsilon$ be a positive number and $f:X\to Y$ a continuous map.
$f$ is said to be an \textbf{$\varepsilon$-embedding} if $\diam f^{-1}(y)<\varepsilon$ for all $y\in Y$.
We define the \textbf{$\varepsilon$-width dimension} $\widim_\varepsilon(X,d)$ 
as the minimum integer $n$ such that there exist an $n$-dimensional simplicial complex $P$ and 
an $\varepsilon$-embedding $f:X\to P$.
This is a \textit{macroscopic dimension} of $X$ in the scale of $\varepsilon$.
The topological dimension $\dim X$ can be obtained by\footnote{If readers do not know the definition of topological dimension, then 
they may think that this is the \textit{definition} of topological dimension.
Only one nontrivial point is that $\dim X$ is a \textit{topological invariant} whereas $\widim_\varepsilon(X,d)$ depends on the distance $d$.} 
\[  \dim X = \lim_{\varepsilon\to 0} \widim_\varepsilon (X,d). \]

Let $(X,\mathbb{Z}^k, T)$ be a dynamical system, namely $X$ is a compact metric space (with a distance $d$) and 
$T:\mathbb{Z}^k\times X\to X$ is a continuous action.
For a finite subset $\Omega\subset \mathbb{Z}^k$ we define a distance $d_\Omega$ on $X$ by 
\[    d_\Omega(x,y) = \max_{n\in \Omega} d(T^n x, T^n y).  \]
It is easy to check the following \textit{subadditivity} and \textit{invariance}: 
\begin{equation*}  
   \begin{split}
     &\widim_\varepsilon(X,d_{\Omega_1\cup \Omega_2}) \leq \widim_\varepsilon(X,d_{\Omega_1}) + \widim_\varepsilon(X,d_{\Omega_2}),  \\
     &\widim_\varepsilon(X, d_{a+\Omega}) = \widim_\varepsilon(X,d_\Omega) \quad 
     (a\in \mathbb{Z}^k, a+\Omega = \{a+x|\, x\in \Omega\}). 
   \end{split}
\end{equation*}   
Then by the standard \textit{division argument} we can show the existence of the limit 
\[   \lim_{N\to \infty} \frac{\widim_\varepsilon \left(X,d_{[N]}\right)}{N^k}. \]
We define the \textbf{mean dimension} of $(X,\mathbb{Z}^k,T)$ by 
\[ \mdim(X,\mathbb{Z}^k,T) = \lim_{\varepsilon\to 0}  
    \left(\lim_{N\to \infty} \frac{\widim_\varepsilon \left(X,d_{[N]}\right)}{N^k}  \right). \]
This is a topological invariant, namely it is independent of the choice of the distance $d$.
We usually abbreviate this to $\mdim(X)$.

\subsection{Review of band-limited functions}  \label{subsection: review of band-limited functions}

Here we review some basic theorems on band-limited functions.
Throughout this subsection we assume that $a_1,\dots,a_k$ are positive numbers.
We denote by $z=x+y\sqrt{-1}$ the standard coordinate of $\mathbb{C}^k$ with 
$x=(x_1,\dots,x_k), y=(y_1,\dots,y_k)\in \mathbb{R}^k$.

\begin{lemma}  \label{lemma: exponential type}
Let $f:\mathbb{C}^k\to \mathbb{C}$ be a holomorphic function such that there exists $C>0$ satisfying 
$\left|f\left(x+y\sqrt{-1}\right)\right|  \leq C e^{\pi\left(a_1|y_1|+\dots+a_k|y_k|\right)}$.
Then it follows that 
\[  \left|f\left(x+y\sqrt{-1}\right)\right|   \leq \norm{f}_{L^\infty(\mathbb{R}^k)} e^{\pi(a_1|y_1|+\dots+a_k|y_k|)}. \]
\end{lemma}

\begin{proof}
The paper \cite[Lemma 2.1]{Gutman--Tsukamoto minimal} proves this for $f:\mathbb{C}\to \mathbb{C}$.
So we consider the case of $k>1$.
Fix $x+y\sqrt{-1}\in \mathbb{C}^k$. For $s+t\sqrt{-1}\in \mathbb{C}$ we set 
$g\left(s+t\sqrt{-1}\right) = f\left(x+y\left(s+t\sqrt{-1}\right)\right)$.
$g$ is a holomorphic function satisfying 
\[  \left|g\left(s+t\sqrt{-1}\right)\right|  \leq C e^{\pi |t|(a_1|y_1|+\dots+a_k|y_k|)}.   \]
By the statement of the one variable case, $\left|g\left(s+t\sqrt{-1}\right)\right|$ is bounded by 
\[ \norm{g}_{L^\infty (\mathbb{R})} e^{\pi |t|(a_1|y_1|+\dots+a_k|y_k|)}  \leq 
   \norm{f}_{L^\infty(\mathbb{R}^k)} e^{\pi |t|(a_1|y_1|+\dots+a_k|y_k|)}.   \]
Letting $s=0$ and $t=1$, we get the statement.    
\end{proof}

\begin{lemma}[Paley--Wiener's theorem]  \label{lemma: Paley--Wiener}
Let $f:\mathbb{R}^k\to \mathbb{C}$ be a bounded continuous function. Then the following two conditions are equivalent.
\begin{enumerate}
   \item The Fourier transform $\hat{f}$ is supported in $\prod_{i=1}^k [-a_i/2,a_i/2]$.
   \item $f$ can be extended to a holomorphic function in $\mathbb{C}^k$ such that there exists $C>0$ satisfying 
           $\left|f\left(x+y\sqrt{-1}\right)\right|  \leq C e^{\pi(a_1|y_1|+\dots+a_k|y_k|)}$.
\end{enumerate}
\end{lemma}

\begin{proof}
This is a special case of a distribution version of Paley--Wiener's theorem (\cite[Chapter 7, Section 8]{Schwartz}).
Here we prove it directly only by using the standard version of Paley--Wiener's theorem (\cite[Section 3.3]{Dym--McKean}):
For a (one-variable) $L^2$-function $f:\mathbb{R}\to \mathbb{C}$, the above two Conditions (1) and (2) are equivalent.
Set $\Omega = \prod_{i=1}^k [-a_i/2,a_i/2]$.

First we assume that $f:\mathbb{R}^k \to \mathbb{C}$ is a $L^2$-function and prove (1) $\Leftrightarrow$ (2).
Suppose $f$ satisfies (1). Then 
\[ f(x) = \int_{\mathbb{R}^k} \hat{f}(\xi) e^{2\pi\sqrt{-1}x\cdot \xi} \, d\xi_1\dots d\xi_k
         = \int_{\Omega} \hat{f}(\xi) e^{2\pi\sqrt{-1}x\cdot \xi} \, d\xi_1\dots d\xi_k. \]
We extend $f$ to a holomorphic function in $\mathbb{C}^k$ by 
\[  f\left(x+y\sqrt{-1}\right) = \int_\Omega \hat{f}(\xi) e^{2\pi\sqrt{-1}(x+y\sqrt{-1})\cdot \xi} \, d\xi_1\dots d\xi_k. \]
Then we can check Condition (2):
\[  \left|f\left(x+y\sqrt{-1}\right)\right| \leq  \int_\Omega |\hat{f}(\xi)| e^{-2\pi y\cdot \xi} \, d\xi_1\dots d\xi_k  \leq 
    \mathrm{const}\cdot e^{\pi\left(a_1|y_1|+\dots+a_k |y_k|\right)}. \]

Next suppose $f$ satisfies (2).
Fix $\xi \in \mathbb{R}^k$ and suppose $|\xi_1|>a_1/2$.
We will show $\hat{f}(\xi)=0$.
We fix $x_2,\dots,x_k$ and consider $f(x_1,\dots,x_k)$ as a one-variable function of variable $x_1$.
This is a $L^2$-function (of variable $x_1$) for almost every choice of $(x_2,\dots,x_k)$.
It follows from $\left|f\left(x_1+ y_1\sqrt{-1},x_2,\dots,x_k\right)\right|  \leq \mathrm{const} \cdot e^{\pi a_1 |y_1|}$
and the standard Paley--Wiener's theorem that
\[  \int_{-\infty}^\infty f(x_1,x_2,\dots,x_k) e^{-2\pi\sqrt{-1} x_1 \xi_1} \, d x_1 = 0
     \quad \text{for a.e. $(x_2,\dots,x_k)$}.  \]
Therefore 
\[ \hat{f}(\xi) = \int_{\mathbb{R}^{k-1}} \left(\int_{-\infty}^\infty f(x_1,x_2,\dots,x_k) e^{-2\pi \sqrt{-1}x_1 \xi_1} \, d x_1 \right)
                     e^{-2\pi\sqrt{-1}(x_2 \xi_2 + \dots + x_k \xi_k)} \, d x_2 \dots d x_k = 0. \]
Thus the proof has been completed under the assumption $f\in L^2$.

Next we only assume that $f:\mathbb{R}^k\to \mathbb{C}$ is a bounded continuous function.
We choose a rapidly decreasing function $\varphi:\mathbb{R}^k\to \mathbb{R}$ such that 
$\max_{x\in \mathbb{R}^k} |\varphi(x)| = \varphi(0) = 1$ and $\supp\, \hat{\varphi} \subset (-1/2,1/2)^k$.
We set $\varphi_\varepsilon(x) = \varphi(\varepsilon x)$ for $\varepsilon>0$.
It satisfies $\hat{\varphi_\varepsilon}(\xi) = \varepsilon^{-k}\hat{\varphi}(\xi/\varepsilon)$ and hence 
$\supp \, \hat{\varphi_\varepsilon} \subset (-\varepsilon/2,\varepsilon/2)^k$.
$\varphi_\varepsilon$ converges to $1$ uniformly over every compact subset of $\mathbb{R}^k$ as $\varepsilon\to 0$.
Since $\varphi_\varepsilon$ is rapidly decreasing (and hence in $L^2$), we can extend it to a holomorphic function satisfying 
\begin{equation} \label{eq: bound on varphi, Paley--Wiener}
  \left|\varphi_\varepsilon\left(x+y\sqrt{-1}\right)\right| \leq e^{\pi \varepsilon (|y_1|+\dots+|y_k|)}
  \quad (\text{here we used Lemma \ref{lemma: exponential type}}).
\end{equation}
We set $f_\varepsilon(x) = \varphi_\varepsilon(x) f(x)$.
Then $f_\varepsilon$ is a $L^2$-function and satisfies $|f_\varepsilon(x)| \leq |f(x)|$.

Suppose $f$ satisfies (1). Then $\hat{f_\varepsilon}$ is supported in $\prod_{i=1}^k [-(a_i+\varepsilon)/2, (a_i+\varepsilon)/2]$.
Thus by the previous argument for $L^2$-functions and Lemma \ref{lemma: exponential type}, 
we can extend $f_\varepsilon$ to a holomorphic function satisfying 
\begin{equation}  \label{eq: bound on f_varepsilon, Paley--Wiener}
   \begin{split}
     \left|f_\varepsilon\left(x+y\sqrt{-1}\right)\right|  &\leq \norm{f_\varepsilon}_{L^\infty(\mathbb{R}^k)} 
          e^{\pi \left\{(a_1+\varepsilon)|y_1| +  \dots  + (a_k+\varepsilon)|y_k|\right\}} \\
         &\leq \norm{f}_{L^\infty(\mathbb{R}^k)}  e^{\pi \left\{(a_1+\varepsilon)|y_1| +  \dots  + (a_k+\varepsilon)|y_k|\right\}}.
   \end{split}
\end{equation}   
Then we can  extend $f$ to a \textit{meromorphic function} in $\mathbb{C}^k$ by 
\[  f\left(x+y\sqrt{-1}\right) = \frac{f_\varepsilon\left(x+y\sqrt{-1}\right)}{\varphi_\varepsilon\left(x+y\sqrt{-1}\right)}. \]
This is independent of $\varepsilon>0$ because of the unique continuation.
Since $\varphi_\varepsilon$ converges to $1$ uniformly over every compact subset, $f\left(x+y\sqrt{-1}\right)$ is actually 
a holomorphic function (namely it does not have neither poles nor indeterminacy points).
By taking the limit in (\ref{eq: bound on f_varepsilon, Paley--Wiener}) we get 
\[ \left| f\left(x+y\sqrt{-1}\right)\right| \leq \norm{f}_{L^\infty(\mathbb{R}^k)} e^{\pi(a_1 |y_1| +\dots+ a_k |y_k|)}. \]
 
 Next suppose $f$ satisfies (2).
 By (\ref{eq: bound on varphi, Paley--Wiener}) and Lemma \ref{lemma: exponential type}, 
 we get (\ref{eq:  bound on f_varepsilon, Paley--Wiener}) again.
Since $f_\varepsilon \in L^2$, the Fourier transform $\hat{f_\varepsilon}$ is supported in 
$\prod_{i=1}^k [-(a_i+\varepsilon)/2, (a_i+\varepsilon)/2]$. 
The functions $\hat{f_\varepsilon}$ converge to $\hat{f}$ in the topology of distribution.
Thus $\hat{f}$ is supported in $\prod_{i=1}^k [-a_i/2, a_i/2]$.
\end{proof}

\begin{corollary}  \label{cor: compactness of B_1}
The space $\mathcal{B}_1(a_1,\dots,a_k)$ of continuous functions $f:\mathbb{R}^k\to [-1,1]$ 
band-limited in $\prod_{i=1}^k [-a_i/2,a_i/2]$ is compact with respect to the topology of uniform convergence over compact subsets 
of $\mathbb{R}^k$.
\end{corollary}

\begin{proof}
Let $\{f_n\}_{n=1}^\infty$ be a sequence in $\mathcal{B}_1(a_1,\dots,a_k)$.
By Lemmas \ref{lemma: exponential type} and \ref{lemma: Paley--Wiener} the functions $f_n$ can be holomorphically extended over 
$\mathbb{C}^k$ satisfying $|f_n(z)|\leq e^{\pi(a_1|y_1|+\dots+a_k|y_k|)}$.
In particular $\{f_n\}$ is bounded over every compact subset of $\mathbb{C}^k$.
By the Cauchy integration formula, it becomes an equicontinuous family over every compact subset
(i.e. normal family).
So we can choose a converging subsequence by the Arzela--Ascoli theorem.
\end{proof}

\begin{lemma}[Sampling theorem]  \label{lemma: sampling theorem}
Let $f:\mathbb{R}^k\to \mathbb{R}$ be a bounded continuous function band-limited in $\prod_{i=1}^k [-a_i/2,a_i/2]$.
Suppose that there exist positive numbers $b_1,\dots,b_k$ such that $a_i b_i < 1$ for all $1\leq i\leq k$ and that $f$ vanishes over 
\[ \{(b_1 n_1,\dots,b_k n_k)|\,n_1,\dots,n_k\in \mathbb{Z}\}. \] 
Then $f$ is identically zero.
\end{lemma}

\begin{proof}
As in the proof of Lemma \ref{lemma: Paley--Wiener},
it is enough to prove the statement under the assumption that $f$ is a $L^2$-function.
Then the proof is standard as follows.
Set $g(x_1,\dots,x_k) = f(b_1 x_1,\dots, b_k x_k)$.
This vanishes on $\mathbb{Z}^k$.
\[ \hat{g}(\xi) = \frac{1}{b_1\dots b_k} \hat{f}\left(\frac{\xi_1}{b_1},\dots,\frac{\xi_k}{b_k}\right). \]
Then $\hat{g}$ is supported in 
\[  \prod_{i=1}^k \left[-\frac{a_i b_i}{2}, \frac{a_i b_i}{2}\right] \subset \left[-\frac{1}{2},\frac{1}{2}\right]^k. \]
Consider the Fourier series of $\hat{g}(\xi)$ for $\xi \in (-1/2,1/2)^k$:
\[ \hat{g}(\xi) = \sum_{n\in \mathbb{Z}^k} a_n e^{-2\pi\sqrt{-1} n\cdot \xi}. \]
The coefficients $a_n$ are given by 
\[ a_n = \int_{[-1/2,1/2]^k} e^{2\pi \sqrt{-1} n\cdot \xi} \hat{g}(\xi) d\xi_1\dots d\xi_k 
   = \overline{\mathcal{F}}\left(\mathcal{F}(g)\right)(n) = g(n)=0. \]
Thus $\hat{g}=0$. This implies $g=0$ and $f=0$.
\end{proof}

\subsection{Technical results on simplicial complexes}  \label{subsection: technical results on simplicial complexes}

The results of this subsection are used only in Section \ref{section: proof of Main Proposition}.
So readers can postpone reading this subsection until they come to Section \ref{section: proof of Main Proposition}.
Recall that we promised that every simplicial complex is finite.
(This is mainly for the simplicity of the exposition.)
Let $P$ be a simplicial complex.
A map $f:P\to \mathbb{R}^n$ is said to be \textbf{simplicial} if it has the form 
\[  f\left(\sum_{k=0}^m \lambda_k v_k \right) = \sum_{k=0}^m \lambda_k f(v_k),  \quad
    \left(\lambda_k\geq 0, \> \sum_{k=0}^m\lambda_k =1\right),   \]
on every simplex $\Delta\subset P$, where $v_0, \dots, v_m$ are the vertices of $\Delta$.
We define $\norm{x}_\infty = \max_{1\leq i\leq n} |x_i|$ for 
$x=(x_1,\dots,x_n)\in \mathbb{R}^n$.

\begin{lemma}[Approximation lemma]  \label{lemma: approximation lemma}
Let $\varepsilon$ and $\delta$ be positive numbers.
Let $(X,d)$ be a compact metric space and $P$ a simplicial complex.
Let $\pi:X\to P$ be an $\varepsilon$-embedding and
$f:X\to \mathbb{R}^n$ a continuous map satisfying 
\[  d(x,y) < \varepsilon \Longrightarrow \norm{f(x)-f(y)}_\infty < \delta. \]
Then, after taking a sufficiently fine subdivision of $P$, we can find a simplicial map 
$g:P\to \mathbb{R}^n$ satisfying 
\[  \norm{f(x)-g(\pi(x))}_\infty < \delta \quad (\forall x\in X). \]
\end{lemma}

\begin{proof}
We reproduce the proofs in \cite{Gutman--Lindenstrauss--Tsukamoto, Gutman--Tsukamoto minimal} for the completeness.
By subdividing $P$ we can assume that for every vertex $v\in P$ the diameter of $\pi^{-1}(O(v))$ is smaller than $\varepsilon$,
where $O(v)$ is the open star of $v$, i.e. the union of the relative interiors of all simplices containing $v$.
For each vertex $v\in P$ we choose $g(v)\in f(\pi^{-1}(O(v)))$. (If this is empty, then we choose arbitrary $g(v)$.)
We linearly extend $g$ all over $P$.
Take $x\in X$ and let $\Delta\subset P$ be the simplex containing $\pi(x)$ in its relative interior.
Every vertex $v\in \Delta$ satisfies $\pi(x) \in O(v)$ and hence $\norm{f(x)-g(v)}_\infty <\delta$
because the diameter of $f(\pi^{-1}(O(v)))$ (which contains both $f(x)$ and $g(v)$) is smaller than $\delta$.
$g(\pi(x))$ is a convex combination of such $g(v)$. Hence $\norm{f(x)-g(\pi(x))}_\infty <\delta$.
\end{proof}

Let $P$ be a simplicial complex.
We denote by $V$ the set of all vertices of $P$.
The set $\mathrm{Hom}(P,\mathbb{R}^n)$ of simplicial maps $f:P\to \mathbb{R}^n$ can be identified with $\left(\mathbb{R}^n\right)^V$, 
which is endowed with the natural topology.

\begin{lemma}[Embedding lemma]  \label{lemma: embedding lemma}
  \begin{enumerate}
  \item If $\dim P<n/2$ then the set of simplicial embeddings (i.e. injective simplicial maps) $f:P\to \mathbb{R}^n$ is an open dense subset of 
  $\mathrm{Hom}(P,\mathbb{R}^n)$.
  \item Let $U$ be the set of simplicial maps $f:P\to \mathbb{R}^n$ such that, for every subcomplex $Q\subset P$ and every subset 
  $A\subset \{1,2,\dots,n\}$ with $\#A > 2\dim Q$, the map 
  \[  Q\to \mathbb{R}^A, \quad x\mapsto f(x)|_A \]
  is an embedding. Then $U$ is open and dense in $\mathrm{Hom}(P,\mathbb{R}^n)$.
  \end{enumerate}
\end{lemma}

\begin{proof}
The statement (2) follows from (1) because the natural map (restriction)
\[ \mathrm{Hom}(P,\mathbb{R}^n) \to \mathrm{Hom}(Q,\mathbb{R}^A) \]
is an open map and thus the preimage of an open and dense set under this map is also open and dense.
(Notice that it corresponds to the natural projection $(\mathbb{R}^n)^{V(P)}\to (\mathbb{R}^A)^{V(Q)}$
where $V(P)$ and $V(Q)$ are the sets of vertices of $P$ and $Q$.)
Thus it is enough to prove (1).
First we show that simplicial embeddings $f:P\to \mathbb{R}^n$ form an open set.
Let $f:P\to \mathbb{R}^n$ be a simplicial embedding.
We would like to show that ``injectiveness'' is preserved under a small perturbation of $f$.
\begin{itemize}
   \item Let $\Delta \subset P$ be a simplex with the vertices $v_0,\dots, v_m$.
           Then $f|_\Delta: \Delta \to \mathbb{R}^n$ is an injection, which means that $f(v_0),\dots, f(v_m)$ are affinely independent.
           If $g:P\to \mathbb{R}^n$ is sufficiently close to $f$ then $g(v_0), \dots, g(v_m)$ are also affinely independent.
   \item Let $\Delta_1,\Delta_2 \subset P$ be two disjoint simplices.   Then the distance between $f(\Delta_1)$ and $f(\Delta_2)$ is 
           positive. This condition is certainly preserved under a small perturbation.
   \item Let $\Delta_1,\Delta_2\subset P$ be two simplices which share a nonempty simplex $\Delta_1\cap \Delta_2$.
           We assume $\Delta_1\neq \Delta_1\cap \Delta_2$ and $\Delta_2\neq \Delta_1\cap \Delta_2$.
           Consider $V(\Delta_1)\setminus V(\Delta_1\cap \Delta_2)$ and $V(\Delta_2)\setminus V(\Delta_1\cap \Delta_2)$, 
           and denote their convex hulls by $\Delta'_1$ and $\Delta'_2$ respectively. Fix $v\in V(\Delta_1\cap \Delta_2)$.
           The map $f|_{\Delta_1\cup \Delta_2}:\Delta_1\cup \Delta_2\to \mathbb{R}^n$ is injective, which means the positivity of 
           \[  \min\left\{\text{Angle between the vectors $f(\overrightarrow{v w_1})$ and $f(\overrightarrow{v w_2})$}|\, w_1\in \Delta_1', 
                w_2\in \Delta'_2\right\}.   \]
           This condition is preserved under a small perturbation. 
\end{itemize}
The above proves that sufficiently small perturbations of $f$ are also injective.

Next we prove that simplicial embeddings $f:P\to \mathbb{R}^n$ form a dense set.
Let $f:P\to \mathbb{R}^n$ be an arbitrary simplicial map and $U\subset \mathrm{Hom}(P,\mathbb{R}^n)$ an open neighborhood of $f$.
By simple linear algebra we can choose $g\in U$ such that for any subset $\{v_1,\dots, v_m\}\subset V(P)$ with $m  \leq n+1$
the points $g(v_1),\dots, g(v_m)$ are affinely independent.
As $2\dim P +1 \leq n$, $g|_{\Delta_1\cup \Delta_2}$ is an embedding for any two simplices $\Delta_1,\Delta_2 \subset P$
and thus $g$ is an embedding.
\end{proof}

\underline{How to use Lemmas \ref{lemma: approximation lemma} and \ref{lemma: embedding lemma}}:  
We would like to illustrate the above two lemmas by a simple application.
(This is a prototype of the later argument but logically independent of the proof of the main theorems.
So experienced readers can skip it.)
Let $(X,d)$ be a compact metric space.
We prove a classical result in the dimension theory:
If $\dim X<n/2$ then $X$ can be topologically embedded into $\mathbb{R}^n$. 
The set of all embeddings $f:X\to \mathbb{R}^n$ is equal to 
\begin{equation} \label{eq: embedding set, warming up}
   \bigcap_{m=1}^\infty \{f:X\to \mathbb{R}^n|\, \text{$(1/m)$-embedding}\}. 
\end{equation}   
So it is a $G_\delta$-subset of the space of all continuous maps $f:X\to \mathbb{R}^n$ since
``$(1/m)$-embedding'' is an open condition.
Fix a natural number $m$ and a positive number $\delta$.
Let $f: X\to \mathbb{R}^n$ be an arbitrary continuous map.
Choose $0<\varepsilon<1/m$ so that 
\[  d(x,y)<\varepsilon \Longrightarrow  \norm{f(x)-f(y)}_\infty < \delta. \]
By $\dim X = \lim_{\varepsilon\to 0}\widim_\varepsilon(X,d)$, we can find a simplicial complex $P$ of $\dim P\leq \dim X <n/2$ 
and an $\varepsilon$-embedding $\pi:X\to P$.
By Lemma \ref{lemma: approximation lemma} there exists a simplicial map $g:P\to \mathbb{R}^n$
satisfying $\norm{f(x)-g(\pi(x))}_\infty < \delta$ for all $x\in X$.
By $\dim P<n/2$ and Lemma \ref{lemma: embedding lemma} (1) we can find a simplicial \textit{embedding} $h:P\to \mathbb{R}^n$
satisfying $\norm{g(p)-h(p)}_\infty <\delta$ for all $p\in P$.
Then the map $h\circ \pi :X\to \mathbb{R}^n$ is an $\varepsilon$-embedding (and hence $(1/m)$-embedding by $\varepsilon <1/m$)
and satisfies $\norm{f(x)-h\circ\pi(x)}_\infty <2\delta$ for all $x\in X$.
Since $\delta$ is arbitrary, this shows that 
\[ \{f:X\to \mathbb{R}^n|\,  \text{$(1/m)$-embedding}\} \]
is dense in the space of all continuous maps from $X$ to $\mathbb{R}^n$.
Thus, by the Baire Category Theorem, the set (\ref{eq: embedding set, warming up}) is dense $G_\delta$.
In particular there exists a topological embedding of $X$ into $\mathbb{R}^n$.

\subsection{A lemma on convex sets}   \label{subsection: a lemma on convex sets}

The result of this subsection is used only in Section \ref{section: proof of Main Proposition}.
 Let $r>0$ and $W\subset \mathbb{R}^k$.   We define $\partial_r W$ as the set of all $t\in \mathbb{R}^k$ such that 
 the closed $r$-ball $B_r(t)$ around $t$ have non-empty intersections both with $W$ and $\mathbb{R}^k\setminus W$. 
 We set $\mathrm{Int}_r W= W\setminus \partial _r W$.

\begin{lemma}  \label{lemma: comparing boundary and interior of convex set}
 For any $c>1$ and $r>0$ there exists $R>0$ such that if a bounded closed convex subset 
$W\subset \mathbb{R}^k$ satisfies $\mathrm{Int}_R W \neq \emptyset$ then 
\[   \left| W\cup \partial_r W \right| < c \left|\mathrm{Int}_r W \right|.  \] 
Here $|\cdot |$ denotes the $k$-dimensional volume (Lebesgue measure).
\end{lemma}

\begin{proof}
We can assume $0\in \mathrm{Int}_R W$.
For $a>0$ we set $a W= \{at|\, t\in W\}$.

\begin{claim}
  $\partial_r W\subset \left(1+ \frac{r}{R}\right)W$.
\end{claim}

\begin{proof}
Take $u\in \mathbb{R}^k$ outside of $\left(1+\frac{r}{R}\right)W$.
There exists a hyperplane $H = \{a_1x_1+\dots+a_k x_k=1\} \subset \mathbb{R}^k$ \textit{separating} $u$ and $\left(1+\frac{r}{R}\right)W$.
Namely 
\[ u\in \{a_1x_1+\dots+a_k x_k>1\}, \quad \left(1+\frac{r}{R}\right)W \subset \{a_1 x_1+\dots+a_k x_k<1\}. \]
Then 
\[  W \subset \left\{a_1 x_1 + \dots + a_k x_k < \left(1+\frac{r}{R}\right)^{-1}\right\}. \]
It follows that the distance between $u$ and $W$ is greater than the distance between $H$ and $\left(1+\frac{r}{R}\right)^{-1} H$.
Let $h$ be the distance between $\left(1+\frac{r}{R}\right)^{-1} H$ and the origin. 
$h>R$ since $0\in \mathrm{Int}_R W$.
Then the distance between $H$ and $\left(1+\frac{r}{R}\right)^{-1} H$ is 
\[  \left(1+\frac{r}{R}\right) h - h = \frac{rh}{R} > r. \]
This implies that $u$ does not belong to $\partial_r W$.
\end{proof}

\begin{claim}
   $\partial_r W \cap \left(1-\frac{r}{R}\right) W = \emptyset$ 
   and hence $\mathrm{Int}_r W \supset \left(1-\frac{r}{R}\right) W$.
\end{claim}

\begin{proof}
Almost the same argument shows that if we take a point $u$ outside of $W$ then the distance between $u$ and $\left(1-\frac{r}{R}\right)W$
is greater than $r$. This implies that every point $v\in \left(1-\frac{r}{R}\right)W$ satisfies $B_r(v)\subset W$ and hence does not 
belong to $\partial_r W$.
\end{proof}

It follows from these two claims that 
\begin{equation*}
    \begin{split}
     \left|W\cup \partial_r W\right|  &\leq \left| \left(1+ \frac{r}{R}\right)W\right| 
       = \left(1+\frac{r}{R}\right)^k |W|, \\
     \left|\mathrm{Int}_r W\right|  &\geq \left|\left(1-\frac{r}{R}\right) W\right|
      = \left(1-\frac{r}{R}\right)^k |W|.
    \end{split}
\end{equation*}
Thus 
\[ \left| W\cup \partial_r W\right| \leq \left(\frac{1+\frac{r}{R}}{1-\frac{r}{R}}\right)^k \left|\mathrm{Int}_r W \right|. \]
Since $c>1$, if $r/R$ is sufficiently small then the right-hand side is smaller than $c \left|\mathrm{Int}_r W\right|$.
\end{proof}

\section{Main proposition} \label{section: main proposition}

\subsection{Statement of the main proposition}  \label{subsection: main proposition}

Main Theorem \ref{main theorem: continuous signal} in Subsection \ref{subsection: embedding into band-limited signals} 
follows from the next proposition whose proof occupies 
the rest of the paper.

\begin{proposition}[Main Proposition] \label{prop: main proposition}
Let $a_1,\dots,a_k$ and $\delta$ be positive numbers.
Let $(X,\mathbb{Z}^k,T)$ be a dynamical system and $d$ a distance on $X$.
Let $f:X\to \mathcal{B}_1(a_1,\dots,a_k)$ be a $\mathbb{Z}^k$-equivariant continuous map.
If $X$ has the marker property and satisfies 
\[ \mdim(X) < \frac{a_1\dots a_k}{2} \]
then there exists a $\mathbb{Z}^k$-equivariant continuous map $g:X\to \mathcal{B}_1(a_1+\delta,\dots,a_k+\delta)$ such that 
\begin{itemize}
  \item $\norm{f(x)-g(x)}_{L^\infty(\mathbb{R}^k)} < \delta$ for all $x\in X$.
  \item $g$ is a $\delta$-embedding with respect to the distance $d$.
\end{itemize}
\end{proposition}

It is important to note that the map $g$ takes values in $\mathcal{B}_1(a_1+\delta,\dots,a_k+\delta)$
which is slightly larger than the target $\mathcal{B}_1(a_1,\dots,a_k)$ of the original map $f$.
Here we prove that Proposition \ref{prop: main proposition} implies Main Theorem \ref{main theorem: continuous signal}.

\begin{proof}[Proof of Main Theorem \ref{main theorem: continuous signal}, assuming Proposition \ref{prop: main proposition}]
The proof is very close to the standard proof of the Baire Category Theorem.
We assume $\diam(X,d) <1$ by rescaling the distance (just for simplicity of the notation).
For $1\leq i\leq k$ we choose sequences $\{a_{in}\}_{n=1}^\infty$ such that 
\[ 0<a_{i1}<a_{i2}<a_{i3}<\dots<a_i, \quad \mdim(X) < \frac{a_{1n}a_{2n}\dots a_{kn}}{2} \quad  (\forall n\geq 1). \]
For $n\geq 1$ we inductively define a positive number $\delta_n$ and a $\mathbb{Z}^k$-equivariant 
$(1/n)$-embedding (with respect to $d$)
$f_n:X\to \mathcal{B}_1(a_{1n},\dots,a_{kn})$. 
First we set 
\[  \delta_1 = 1, \quad f_1(x)=0 \quad (\forall x\in X). \]
Notice that $f_1$ is a $1$-embedding because $\diam(X,d)<1$.
Suppose we have defined $\delta_n$ and $f_n$. 
Since $f_n$ is a $(1/n)$-embedding, we can choose $0<\delta_{n+1}<\delta_n/2$ such that 
if a $\mathbb{Z}^k$-equivariant continuous map $g:X\to \mathcal{B}_1(a_1,\dots,a_k)$ satisfies 
\[  \sup_{x\in X} \norm{f_n(x)-g(x)}_{L^\infty(\mathbb{R}^k)} \leq \delta_{n+1} \]
then $g$ is also a $(1/n)$-embedding.
By applying Proposition \ref{prop: main proposition} to $f_n$, we can find a 
$\mathbb{Z}^k$-equivariant continuous map $f_{n+1}:X\to \mathcal{B}_1(a_{1,n+1},\dots, a_{k,n+1})$ such that 
\begin{itemize}
   \item $\sup_{x\in X} \norm{f_n(x)-f_{n+1}(x)}_{L^\infty(\mathbb{R}^k)} < \delta_{n+1}/2$.
   \item $f_{n+1}$ is a $\frac{1}{n+1}$-embedding with respect to $d$.
\end{itemize}

For $n>m\geq 1$ 
\begin{equation*}
   \begin{split}
    \sup_{x\in X} \norm{f_n(x)-f_m(x)}_{L^\infty(\mathbb{R}^k)}  &\leq 
   \sum_{l=m}^{n-1}\sup_{x\in X}\norm{f_l(x)-f_{l+1}(x)}_{L^\infty(\mathbb{R}^k)}   \\
   &< \sum_{l=m}^{n-1}\frac{\delta_{l+1}}{2}      \\
   &< \sum_{l=1}^\infty  2^{-l} \delta_{m+1} = \delta_{m+1} \to 0 \quad (m\to \infty).
   \end{split}
\end{equation*}   
Here we used $\delta_{l+1} < \delta_{l}/2$.
By taking a limit, we find a $\mathbb{Z}^k$-equivariant continuous map $f:X\to \mathcal{B}_1(a_1,\dots,a_k)$ satisfying 
\[  \sup_{x\in X} \norm{f(x)-f_n(x)}_{L^\infty(\mathbb{R}^k)}   \leq \delta_{n+1}  \]
for all $n\geq 1$.
It follows from the definition of $\delta_{n+1}$  that $f$ is a $(1/n)$-embedding.
Since $n$ is arbitrary, $f$ is an embedding.
\end{proof}

The above proof used a flexibility exhibited by band-limited signals and not by discrete signals.
At each step of the induction, the target $\mathcal{B}_1(a_{1,n+1},\dots, a_{k,n+1})$ of the map $f_{n+1}$ is slightly bigger than 
the previous target $\mathcal{B}_1(a_{1n},\dots,a_{kn})$.
The differences $a_{i, n+1}-a_{i n}$ converge to zero as $n$ goes to infinity.
Such an argument is possible because $a_1, \dots, a_k$ are \textit{continuous} parameters.
We cannot apply the same argument to $\left([0,1]^D\right)^{\mathbb{Z}^k}$ because it depends on a \textit{discrete}
parameter $D$.

\subsection{Strategy of the proof}  \label{subsection: strategy of the proof}

Most known theorems about embedding (e.g. the Whitney embedding theorem for manifolds) are proved by 
\textit{perturbation}.
The proof of Proposition \ref{prop: main proposition} also uses this idea  \textit{but with a twist}.
The map $g$ in the statement of Proposition \ref{prop: main proposition} will have the form 
\[  g(x) = g_1(x) + g_2(x), \quad  (x\in X). \]
The first term $g_1(x)$ is a band-limited function in $\mathcal{B}_1(a_1,\dots,a_k)$, which is constructed by 
perturbing the function $f(x)\in \mathcal{B}_1(a_1,\dots,a_k)$.
The second term $g_2(x)$ is also a band-limited function whose frequencies are restricted in a compact subset of 
\[  \prod_{i=1}^k \left[-\frac{a_i+\delta}{2},\frac{a_i+\delta}{2}\right]  \setminus \prod_{i=1}^k \left[-\frac{a_i}{2},\frac{a_i}{2}\right]. \]
The construction of $g_2(x)$ is (essentially) independent of $f(x)$.
The function $g_2(x)$ \textit{encodes how to perturb $f(x)$}.
In other words, the function $g_1(x)$ is constructed by perturbing the function $f(x)$, and the 
perturbation \textit{method} is determined by the function $g_2(x)$.
Probably this explanation is not very clear. So let us try a more naive explanation.
Consider \textit{cooking}.
Each point $x\in X$ is a cook.
The function $f(x)$ is an ingredient of cooking (e.g. a raw salmon) that the cook $x$ chooses.
The function $g_2(x)$ is a cookware (e.g. kitchen knife and oven) that the cook $x$ possesses.
The function $g_1(x)$ is the result of cooking (e.g. grilled salmon).
The grilled salmon $g_1(x)$ is made from the raw salmon $f(x)$ by using the knife and oven $g_2(x)$.
The knife and oven $g_2(x)$ are independent of the raw salmon $f(x)$.

The Fourier transforms $\mathcal{F}\left(g_1(x)\right)$ and $\mathcal{F}\left(g_2(x)\right)$ have disjoint supports.
Hence if we have the equation $g(x)=g(y)$ for some $x$ and $y$ in $X$, then it follows that 
\[   g_1(x) = g_1(y), \quad g_2(x)=g_2(y). \]
We would like to deduce $d(x,y)<\delta$ from these equations.
In other words we try to determine who is the cook (up to some small error) by 
knowing the result of cooking ($g_1(x)=g_1(y)$) and what cookware was used ($g_2(x)=g_2(y)$).

More precisely, the proof goes as follows.
Take $x\in X$.
The function $f(x)$ is defined over $\mathbb{R}^k$.
It is difficult to control functions over unbounded domains.
So we introduce a \textit{tiling} (indexed by $\mathbb{Z}^k$)
\[  \mathbb{R}^k = \bigcup_{n\in \mathbb{Z}^k} W(x,n),   \]
such that 
\begin{enumerate}
  \item Each tile $W(x,n)$ is a bounded convex set, and different tiles $W(x,n)$ and $W(x,m)$ are disjoint except for their boundaries.
  \item $W(x,n)$ depends continuously on $x\in X$ in the Hausdorff topology. 
  \item The tiles are equivariant in the sense that 
  \[  W(T^n x, m) = -n + W(x, n+m)   = \{-n+t|\, t\in W(x,n+m)\}.  \]
  \item ``Most'' part of the space $\mathbb{R}^k$ is covered by ``sufficiently large'' tiles.
\end{enumerate}  

We try to perturb the function $f(x)$ over each tile $W(x,n)$.
We can construct a good perturbation over sufficiently large tiles.
So if every tile is sufficiently large, then we can construct a good perturbation of $f(x)$ over the whole space $\mathbb{R}^k$.
Unfortunately some tiles may be small in general.
We cannot construct a good perturbation over such tiny tiles.
Condition (4) helps us here.
We will construct a ``social welfare system'' of the tiles $W(x,n)$.
(This idea was first introduced in \cite{Gutman--Tsukamoto minimal}.)
Large tiles help small tiles and 
bear ``additional'' perturbations which are originally ``duties'' of small tiles.
By using this social welfare system we construct the perturbation $g_1(x)$ of the function $f(x)$.

The map $g_1$ will have the following property:
If $x$ and $y$ in $X$ satisfy $g_1(x)=g_1(y)$ and 
\begin{equation}  \label{eq: tiling equation}
   \forall n\in \mathbb{Z}^k: \quad W(x,n) = W(y,n)
\end{equation}
then it follows that $d(x,y)<\delta$.
Namely if we know $g_1(x)$ and the tiling $\{W(x,n)\}$ then we can recover the point $x$ up to error $\delta$.
So the next problem is how to encode the information of the tiles.
The function $g_2(x)$ is introduced for solving this problem.
We encode all the information of the tiles $W(x,n)$ to a band-limited function $g_2(x)$.
(Indeed the real argument below goes in a reverse way. First we construct the function $g_2(x)$, and then 
the tiles $W(x,n)$ are constructed from $g_2(x)$. So $g_2(x)$ knows everything about the tiles $W(x,n)$.)
In particular the equation $g_2(x)=g_2(y)$ implies (\ref{eq: tiling equation}).
As a conclusion, the function $g(x)=g_1(x)+g_2(x)$ satisfies the requirements of Proposition \ref{prop: main proposition}.
This is the outline of the proof.

We would like to emphasize that the most important
new idea introduced in this paper is the technique to encode the information of the tiling $\{W(x,n)\}$ into the 
band-limited function $g_2(x)$.
This idea is new even in the one dimensional case.
This is the main reason why Main Theorems \ref{main theorem: discrete signal} and \ref{main theorem: continuous signal} have
 some novelty even in the case of $k=1$.

Section \ref{section: Interpolating functions} introduces interpolating functions which will be used in the perturbation process.
In Section \ref{section: dynamical construction of tiling-like band-limited functions} we construct the tiling $W(x,n)$ and their 
social welfare system and explain how to encode them into a band-limited function.
The proof of Proposition \ref{prop: main proposition} is finished in Section \ref{section: proof of Main Proposition}.

\subsection{Fixing some notations} \label{subsection: fixing some notations}

The proof of Proposition \ref{prop: main proposition} is notationally messy.
So here we gather some of the notations for the convenience of readers.
Throughout the rest of the paper (except for Section \ref{section: open problems} where we discuss open problems)
we fix the following.

\begin{itemize}
  \item $a_1,\dots, a_k$ and $\delta$ are positive numbers. We additionally assume 
          \begin{equation}  \label{eq: choice of a_i and delta}
             \delta < \min(1, a_1,\dots, a_k).
          \end{equation}
  \item $(X,\mathbb{Z}^k,T)$ is a dynamical system satisfying the marker property and  
  \[ \mdim(X) < \frac{a_1\dots a_k}{2}. \]
  \item We fix positive \textit{rational numbers} $\rho_1,\dots,\rho_k$ satisfying 
        \begin{equation} \label{eq: choice of rho_i}
             \rho_i < a_i \quad (\forall 1\leq i\leq k), \quad \mdim(X) < \frac{\rho_1\dots \rho_k}{2}. 
        \end{equation}
  \item We define a lattice $\Gamma\subset \mathbb{R}^k$ by 
        \begin{equation}  \label{eq: definition of Gamma}
               \Gamma = \left\{\left(\frac{t_1}{\rho_1},\dots,\frac{t_k}{\rho_k}\right) \middle|\, t_1,\dots, t_k\in \mathbb{Z}\right\}.
        \end{equation}
       Moreover we define $\Gamma_1\subset \mathbb{R}^k$ by 
       \begin{equation} \label{eq: definition of Gamma hat}
              \Gamma_1 = \bigcup_{n\in \mathbb{Z}^k} (n+\Gamma), \quad 
              (n+\Gamma = \{n+\gamma|\, \gamma\in \Gamma\}).   
       \end{equation}
       Since $\rho_i$ are rational numbers, $\Gamma_1$ is also a lattice.
\end{itemize}

\section{Interpolating functions}  \label{section: Interpolating functions}

Recall that we fixed positive numbers $a_1,\dots, a_k$, positive rational numbers $\rho_1,\dots,\rho_k$ with 
$\rho_i<a_i$ and 
lattices $\Gamma \subset \Gamma_1\subset \mathbb{R}^k$
in Subsection \ref{subsection: fixing some notations}.
We fix a positive number $\tau$ satisfying 
\begin{equation}  \label{eq: fixing tau}
    \forall 1\leq i\leq k: \quad \rho_i+ \tau < a_i. 
\end{equation}    

For $r>0$ and $a\in \mathbb{R}^k$ we denote by $B_r(a)$ the closed $r$-ball around $a$, and we set 
$B_r = B_r(0)$.
We choose a rapidly decreasing function $\chi_0: \mathbb{R}^k\to \mathbb{R}$ satisfying 
$\chi_0(0)=1$ and $\supp \, \mathcal{F}(\chi_0) \subset B_{\tau/2}$.

\begin{notation} \label{notation: choice of r_0}
  \begin{enumerate}
     \item We set 
       \begin{equation}  \label{eq: definition of K_0}
          K_0 = \sup_{t\in \mathbb{R}^k} \sum_{\lambda\in \Gamma_1} |\chi_0(t-\lambda)|. 
       \end{equation}
              This is finite because $\chi_0$ is rapidly decreasing and $\Gamma_1$ is a lattice.
     \item We fix $r_0>0$ satisfying 
       \begin{equation}  \label{eq: definition of r_0}
           \sum_{\lambda\in \Gamma_1 \setminus B_{r_0}} |\chi_0(\lambda)| < \frac{1}{2}. 
       \end{equation}    
   \end{enumerate}    
\end{notation}

\begin{definition}  \label{def: admissible set and function}
     \begin{enumerate}
         \item  A subset $\Lambda \subset \Gamma_1$ is called an \textbf{admissible set} if any two points 
         $\lambda_1, \lambda_2\in \Lambda$ satisfy at least one of the following two conditions:
         \begin{itemize}
           \item $\lambda_1-\lambda_2\in \Gamma$.
           \item $|\lambda_1-\lambda_2| > r_0$.
         \end{itemize}
         This notion is an elaborated formulation of ``almost regular interpolation set'' introduced in Subsection
          \ref{subsection: one dimension versus multi-dimension}.
           We denote by $\mathrm{Ads}$ the set of all admissible sets $\Lambda\subset \Gamma_1$.
        \item  A function $p:\Gamma_1 \to [0,1]$ is called an \textbf{admissible function} if the set 
        \[  \{\lambda \in \Gamma_1|\, p(\lambda)>0\}  \]
        is an admissible set.
        We denote by $\mathrm{Adf}$ the set of all admissible functions.
     \end{enumerate}
\end{definition}

We define the function $\sinc: \mathbb{R}\to \mathbb{R}$ by 
\[  \sinc(t) = \frac{\sin (\pi t)}{\pi t}, \quad \sinc(0) = 1. \]
This is one of the most famous functions in signal analysis.
The function $\sinc(t)$ vanishes on $\mathbb{Z}\setminus \{0\}$ and 
its Fourier transform is supported in $[-1/2,1/2]$.
It satisfies $|\sinc (t)| \leq 1$ for all $t\in \mathbb{R}$.

Let $\Lambda\subset \Gamma_1$ be an admissible set.
We denote by $\ell^\infty(\Lambda)$ the Banach space of all bounded functions 
$u:\Lambda\to \mathbb{R}$ endowed with the supremum norm $\norm{\cdot}_\infty$.
For $u\in \ell^\infty(\Lambda)$ we define a band-limited function $\varphi_\Lambda(u)\in \mathcal{B}(\rho_1+\tau, \dots,\rho_k+\tau)$
by 
\[ \varphi_\Lambda(u)(t) = \sum_{\lambda = (\lambda_1,\dots,\lambda_k)\in \Lambda} 
   u(\lambda) \chi_0(t-\lambda) \prod_{i=1}^k \sinc\left(\rho_i(t_i-\lambda_i)\right), 
    \quad   (t=(t_1,\dots, t_k)  \in \mathbb{R}^k). \]
Here $\rho_i$ are rational numbers introduced in Subsection \ref{subsection: fixing some notations}.
It follows from the definition of $K_0$ that 
\[  \norm{\varphi_\Lambda(u)}_{L^\infty(\mathbb{R}^k)} \leq K_0 \norm{u}_\infty. \]
Since we promised $\rho_i+\tau<a_i$ in (\ref{eq: fixing tau}), the function $\varphi_\Lambda(u)$ belongs to 
$\mathcal{B}(a_1,\dots,a_k)$.
We define a bounded linear operator $S_\Lambda:\ell^\infty(\Lambda)\to \ell^\infty(\Lambda)$ by 
\[ S_\Lambda(u) = \varphi_\Lambda(u)|_{\Lambda} \, \left(= \left(\varphi_\Lambda(u)(\lambda)\right)_{\lambda\in \Lambda}\right). \]

\begin{remark}  \label{remark: empty case}
When $\Lambda$ is the empty set, we define $\ell^\infty(\Lambda)$ to be the trivial vector space consisting 
only of zero and $\varphi_\Lambda(0)=0$,
the operator $S_\Lambda$ the identity operator of the trivial vector space.
\end{remark}

\begin{lemma} \label{lemma: operator norm of S-1}
  $\norm{S_\Lambda-\id} \leq 1/2$, where the left-hand side is the operator norm of $S_\Lambda-\id$.
\end{lemma}

\begin{proof}
For $\lambda = (\lambda_1,\dots,\lambda_k)\in \Lambda$ the function 
\[   \chi_0(t-\lambda) \prod_{i=1}^k \sinc\left(\rho_i(t_i-\lambda_i)\right) \]
vanishes on $\left(\lambda+\Gamma\right) \setminus \{\lambda\}$ and its value at $t=\lambda$ is $1$.
Then by the admissibility of $\Lambda$ 
\[ \varphi_\Lambda(u)(\lambda)-u(\lambda) = \sum_{\lambda'\in \Lambda\setminus B_{r_0}(\lambda)} 
   u(\lambda') \chi_0(\lambda'-\lambda) \prod_{i=1}^k \sinc\left(\rho_i(\lambda'_i-\lambda_i)\right). \]
By the choice of $r_0$ in (\ref{eq: definition of r_0}),
\[ |\varphi_\Lambda(u)(\lambda)-u(\lambda)| \leq 
   \left(\sum_{\lambda' \in \Lambda\setminus B_{r_0}(\lambda)} |\chi_0(\lambda'-\lambda)|\right) \norm{u}_\infty 
    \leq \frac{1}{2} \norm{u}_\infty. \]
\end{proof}

Therefore the operator $S_\Lambda$ has the inverse 
\[  S_\Lambda^{-1} = \sum_{n=0}^\infty (\id-S_\Lambda)^n  \text{ with $\norm{S_\Lambda^{-1}}\leq 2$}. \]
For $u\in \ell^\infty(\Lambda)$ we define a band-limited function $\psi_\Lambda(u)\in \mathcal{B}(a_1,\dots,a_k)$ by
\[ \psi_\Lambda(u) =  \varphi_\Lambda(S_\Lambda^{-1}(u)). \]
This function satisfies 
\[ \psi_\Lambda(u)(\lambda)=u(\lambda) \quad (\forall \lambda\in \Lambda), \quad 
    \norm{\psi_\Lambda(u)}_{L^\infty(\mathbb{R}^k)} \leq 2K_0 \norm{u}_\infty. \]
    The former property means that $\psi_\Lambda(u)$ interpolates the sequence $u$.
The construction of $\psi_\Lambda(u)$ is ``$\mathbb{Z}^k$-equivariant'': For $n\in \mathbb{Z}^k$ and $u\in \ell^\infty(\Lambda)$ we 
define $v\in \ell^\infty(n+\Lambda)$ by $v(n+\lambda) = u(\lambda)$. 
Then
\[ \psi_{n+\Lambda}(v)(t+n) = \psi_\Lambda(u)(t). \]
If $\Lambda$ is the empty set, then $\psi_\Lambda(0)=0$.

We denote by $\mathbb{B}_1\left(\ell^\infty(\Lambda)\right)$ the closed unit ball of $\ell^\infty(\Lambda)$ around the origin with 
respect to the supremum norm.
The function $\psi_\Lambda(u)$ depends continuously on $\Lambda$ and $u$ as follows:

\begin{lemma}[Continuity of the construction of $\psi_\Lambda(u)$]  \label{lemma: continuity of psi}
For any $r>0$ and $\varepsilon>0$ there exist $r'>0$ and $\varepsilon'>0$ so that if $\Lambda_1,\Lambda_2\in \mathrm{Ads}$,   
$u_1\in \mathbb{B}_1\left(\ell^\infty(\Lambda_1)\right)$ and $u_2\in \mathbb{B}_1\left(\ell^\infty(\Lambda_2)\right)$ satisfy 
\begin{equation}  \label{eq: topology of Lambda and u}
    \Lambda_1\cap B_{r'} = \Lambda_2\cap B_{r'}, \quad 
    |u_1(\lambda)-u_2(\lambda)| < \varepsilon' \quad (\forall \lambda\in  \Lambda_1\cap B_{r'}) 
\end{equation}    
then for all $t\in B_r$ 
\[ |\psi_{\Lambda_1}(u_1)(t)-\psi_{\Lambda_2}(u_2)(t)| < \varepsilon. \]
\end{lemma}

\begin{proof}
By the linearity of $\varphi_{\Lambda}$,
\[ \psi_{\Lambda}(u) = \varphi_\Lambda(S_\Lambda^{-1}(u)) = \sum_{n=0}^\infty \varphi_\Lambda\left((\id-S_\Lambda)^n u\right). \]
By Lemma \ref{lemma: operator norm of S-1}
\[  \norm{\varphi_\Lambda\left((\id-S_\Lambda)^n u\right)}_{L^\infty(\mathbb{R}^k)} \leq K_0 \norm{(\id-S_\Lambda)^n u}_\infty
    \leq \frac{K_0}{2^n} \norm{u}_\infty. \]
Therefore the statement follows from the next claim.

\begin{claim}
 \begin{enumerate}
     \item  For any $r>0$ and $\varepsilon>0$ there exist $r'>0$ and $\varepsilon'>0$ so that if $\Lambda_1,\Lambda_2\in \mathrm{Ads}$,   
              $u_1\in \mathbb{B}_1\left(\ell^\infty(\Lambda_1)\right)$ and $u_2\in \mathbb{B}_1\left(\ell^\infty(\Lambda_2)\right)$ satisfy 
              (\ref{eq: topology of Lambda and u}) then for all $t \in B_r$ 
             \[  |\varphi_{\Lambda_1}(u_1)(t)-\varphi_{\Lambda_2}(u_2)(t)| < \varepsilon. \]
    \item  For any $r>0$, $\varepsilon>0$ and a natural number $n$ there exist $r'>r$ and $\varepsilon'>0$ so that 
              if $\Lambda_1,\Lambda_2\in \mathrm{Ads}$,   
              $u_1\in \mathbb{B}_1\left(\ell^\infty(\Lambda_1)\right)$ and $u_2\in \mathbb{B}_1\left(\ell^\infty(\Lambda_2)\right)$ satisfy 
              (\ref{eq: topology of Lambda and u}) then for all $\lambda \in \Lambda_1 \cap B_r$
             \[  \left|(\id-S_{\Lambda_1})^n(u_1)(\lambda)-(\id-S_{\Lambda_2})^n(u_2)(\lambda)\right| < \varepsilon. \]  
  \end{enumerate}
\end{claim}

\begin{proof}
(1)  We choose $r'>r$ so large that for all $t \in B_r$ 
\[ \sum_{\lambda\in \Gamma_1 \setminus B_{r'}} |\chi_0(t-\lambda)| < \frac{\varepsilon}{4}. \] 
Then Condition (\ref{eq: topology of Lambda and u}) implies that for $t \in B_r$
\begin{equation*}
   \begin{split}
    |\varphi_{\Lambda_1}(u_1)(t)-\varphi_{\Lambda_2}(u_2)(t)|  &< \frac{\varepsilon}{2} + 
      \sum_{\lambda\in \Lambda_1\cap B_{r'}} |u_1(\lambda)-u_2(\lambda)| \cdot |\chi_0(t-\lambda)| \\
    &\leq \frac{\varepsilon}{2} + \varepsilon' \sum_{\lambda\in \Gamma_1 \cap B_{r'}} |\chi_0(t-\lambda)|. 
   \end{split}
\end{equation*}    
We can choose $\varepsilon' >0$ so small that the right-hand side is smaller than $\varepsilon$.

(2) The statement can be reduced to the case of $n=1$ by induction.
The case of $n=1$ immediately follows from the statement (1).
\end{proof} 
\end{proof}

$\psi_\Lambda(u)$ is a nice interpolating function.
But we need one more twist.
We naturally identify the set $\{0,1\}^{\Gamma_1}$ with the set of all subsets of $\Gamma_1$.
Namely, $\Lambda\subset \Gamma_1$ is identified with the characteristic function $1_\Lambda \in \{0,1\}^{\Gamma_1}$.
Then in particular $\mathrm{Ads}$ (the set of admissible sets) is a Borel subset of $\{0,1\}^{\Gamma_1}$.
Let $p:\Gamma_1 \to [0,1]$ be an admissible \textit{function}, namely $\{\lambda\in \Gamma_1|\, p(\lambda)>0\}$
belongs to $\mathrm{Ads}$.
We define a probability measure $\mu_p$ on $\{0,1\}^{\Gamma_1}$ by 
\[  \mu_p = 
     \prod_{\lambda\in \Gamma_1} \left(\left(1-p(\lambda)\right)\boldsymbol{\delta}_0 + p(\lambda) \boldsymbol{\delta}_1\right), \]
where $\boldsymbol{\delta}_0$ and $\boldsymbol{\delta}_1$ are the delta probability measures at $0$ and $1$ on $\{0,1\}$ respectively.
Since $p$ is admissible, the set $\mathrm{Ads}$ has full measure: $\mu_p(\mathrm{Ads}) =1$.

Let $\ell^\infty(\Gamma_1)$ be the Banach space of bound functions $u:\Gamma_1 \to \mathbb{R}$.
For $u\in \ell^\infty(\Gamma_1)$ we define a band-limited function $\Psi(p,u) \in \mathcal{B}(a_1,\dots,a_k)$ by
\begin{equation} \label{eq: interpolation function Psi}
     \Psi(p,u) = \int_{\Lambda\in \mathrm{Ads}} \psi_\Lambda(u|_\Lambda)\, d\mu_p(\Lambda). 
\end{equation}     
This function is the main product of this section.

\begin{lemma}[Basic properties of $\Psi$] \label{lemma: basic properties of Psi}
    $\Psi$ satisfies:
    \begin{description}
       \item[(1) Interpolation]   If $\lambda\in \Gamma_1$ satisfies $p(\lambda)=1$ then $\Psi(p,u)(\lambda)=u(\lambda)$. 
       \item[(2) Boundedness]  $\norm{\Psi(p,u)}_{L^\infty(\mathbb{R}^k)} \leq 2K_0 \norm{u}_\infty$. 
       \item[(3) Equivariance]   For $n\in \mathbb{Z}^k$ we define $q\in \mathrm{Adf}$ and $v\in \ell^\infty(\Gamma_1)$ by 
                              $q(\lambda+n) = p(\lambda)$ and $v(\lambda+n) = u(\lambda)$ for $\lambda\in \Gamma_1$.
                              Then $\Psi(q,v)(t+n) = \Psi(p,u)(t)$. 
    \end{description}
\end{lemma}    

\begin{proof}
These three properties immediately follow from the corresponding properties of $\psi_\Lambda(u)$.
\end{proof}

We would like to show that the map $\Psi$ depends continuously on $p$ and $u$ (Proposition \ref{prop: continuity of Psi} below).
But we need a preparation before that.
For $r>0$ and $p\in \mathrm{Adf}$ we define probability measures $\mu_{p,r}$ and $\nu_{p,r}$ on $\{0,1\}^{\Gamma_1 \cap B_r}$
and $\{0,1\}^{\Gamma_1\setminus B_r}$ respectively by 
\[ \mu_{p,r} = \prod_{\lambda\in \Gamma_1 \cap B_r}
   \left(\left(1-p(\lambda)\right)\boldsymbol{\delta}_0 + p(\lambda) \boldsymbol{\delta}_1\right), \quad 
   \nu_{p,r} =  \prod_{\lambda\in \Gamma_1\setminus B_r} 
    \left(\left(1-p(\lambda)\right)\boldsymbol{\delta}_0 + p(\lambda) \boldsymbol{\delta}_1\right).  \]
It follows $\mu_p = \mu_{p,r}\otimes \nu_{p,r}$. 
 
\begin{lemma}[Truncation of the integral] \label{lemma: truncation of the integral}
For any $r>0$ and $\varepsilon>0$ if we choose $r'>0$ sufficiently large then for any $p\in \mathrm{Adf}$,
$u\in \mathbb{B}_1\left(\ell^\infty(\Gamma_1)\right)$ and $t\in B_r$ 
\begin{equation} \label{eq: truncation of the integral}
   \left|\Psi(p,u)(t) - \int_{\Lambda\in \mathrm{Ads}\cap \{0,1\}^{\Gamma_1\cap B_{r'}}} \psi_{\Lambda}(u|_{\Lambda})(t) \, 
   d\mu_{p,r'}(\Lambda)\right| < \varepsilon. 
\end{equation}   
\end{lemma}

\begin{proof}
For an admissible set $\Lambda_1\subset \Gamma_1\cap B_{r'}$ we define $\mathrm{Ads}(\Lambda_1,r')\subset \mathrm{Ads}$ as the set of 
$\Lambda_2\subset \Gamma_1\setminus B_{r'}$ satisfying $\Lambda_1\cup \Lambda_2\in \mathrm{Ads}$.
If $\mu_{p,r'}\left(\{\Lambda_1\}\right) >0$ then $\mathrm{Ads}(\Lambda_1, r')$ has full measure with respect to $\nu_{p,r'}$
because $\mu_p\left(\mathrm{Ads}\right) = 1$.

By $\mu_p=\mu_{p,r'}\otimes \nu_{p,r'}$
\[ \Psi(p,u) = 
    \int_{\Lambda_1\in \mathrm{Ads}\cap \{0,1\}^{\Gamma_1\cap B_{r'}}} 
   \left(\int_{\Lambda_2\in \mathrm{Ads}(\Lambda_1,r')} \psi_{\Lambda_1\cup \Lambda_2}(u|_{\Lambda_1\cup \Lambda_2})\, 
   d\nu_{p,r'}(\Lambda_2)\right) \, d\mu_{p,r'}(\Lambda_1). \]
If $r'$ is sufficiently large, then it follows from Lemma \ref{lemma: continuity of psi} that 
\[ \left|\psi_{\Lambda_1\cup \Lambda_2}(u|_{\Lambda_1\cup \Lambda_2}) - \psi_{\Lambda_1}(u|_{\Lambda_1})\right| < \varepsilon \,
   \text{ on $B_r$}. \]
Then (\ref{eq: truncation of the integral}) follows.
\end{proof}

\begin{proposition}[Continuity of $\Psi$]  \label{prop: continuity of Psi}
For any $r>0$ and $\varepsilon>0$ there exist $r'>0$ and $\varepsilon'>0$ so that if 
$p, q\in \mathrm{Adf}$ and $u, v\in \mathbb{B}_1\left(\ell^\infty(\Gamma_1)\right)$
satisfy 
\begin{equation}  \label{eq: topology of p and u}
   \forall \lambda\in \Gamma_1  \cap B_{r'}:  \quad  |p(\lambda)-q(\lambda)|<\varepsilon', \quad 
    |u(\lambda)-v(\lambda)|<\varepsilon' 
\end{equation}    
then for all $t \in B_r$ 
\[  |\Psi(p,u)(t) - \Psi(q,v)(t)| < \varepsilon. \]
\end{proposition}

\begin{proof}
From Lemma \ref{lemma: truncation of the integral} it is enough to show that if Condition (\ref{eq: topology of p and u})
holds for sufficiently large $r'>0$ and sufficiently small $\varepsilon'>0$ then for all $t \in B_r$
\[ \left| \int_{\Lambda\in \mathrm{Ads}\cap \{0,1\}^{\Gamma_1\cap B_{r'}}} \psi_{\Lambda}(u|_{\Lambda})(t) \, 
   d\mu_{p,r'}(\Lambda)  -
    \int_{\Lambda\in \mathrm{Ads}\cap \{0,1\}^{\Gamma_1\cap B_{r'}}} \psi_{\Lambda}(v|_{\Lambda})(t) \, 
   d\mu_{q,r'}(\Lambda)\right| < \varepsilon. \]
The left-hand side is bounded by 
\begin{equation*}
    \begin{split}
     \int_{\Lambda\in \mathrm{Ads}\cap \{0,1\}^{\Gamma_1\cap B_{r'}}} 
    \left|\psi_{\Lambda}(u|_{\Lambda})(t)-\psi_{\Lambda}(v|_{\Lambda})(t)\right| \, d\mu_{p,r'}(\Lambda) \\
    + \int_{\Lambda\in \mathrm{Ads}\cap \{0,1\}^{\Gamma_1\cap B_{r'}}} |\psi_{\Lambda}(v|_{\Lambda})(t)|\, 
    d|\mu_{p,r'}-\mu_{q,r'}|(\Lambda). 
   \end{split}
\end{equation*}   
The first term can be made arbitrarily small by Lemma \ref{lemma: continuity of psi}.
The second term is bounded by 
\begin{equation*}
    2K_0 \int_{\{0,1\}^{\Gamma_1\cap B_{r'}}} d|\mu_{p,r'}-\mu_{q,r'}|.
\end{equation*}     
The integral here is the total variation of the signed measure $\mu_{p,r'}-\mu_{q,r'}$, which is equal to the sum 
\[  \sum_{\Lambda\in \{0,1\}^{\Gamma_1\cap B_{r'}}} \left|\mu_{p,r'}(\{\Lambda\}) - \mu_{q,r'}(\{\Lambda\})\right|. \]
This can be made arbitrarily small by choosing $\varepsilon' >0$ in (\ref{eq: topology of p and u}) sufficiently small.
\end{proof}

\section{Dynamical construction of tiling-like band-limited maps}  
\label{section: dynamical construction of tiling-like band-limited functions}

This section is the longest section of the paper.
As we described in Subsection \ref{subsection: strategy of the proof}, the proof of Proposition \ref{prop: main proposition}
uses a \textit{tiling} $\mathbb{R}^k = \bigcup_{n\in \mathbb{Z}^k} W(x,n)$ and a \textit{social welfare system among tiles},
which should be \textit{encoded into a certain band-limited function $g_2(x)$}.
The purpose of this section is to construct these ingredients.
The following diagram outlines the construction.

\begin{equation*}
  \begin{CD}
   \text{A point $x$ in a dynamical system $(X,\mathbb{Z}^k,T)$ satisfying the marker property.}\\
   @VVV\\
   \text{A tiling $\mathbb{R}^k = \bigcup_{n\in \mathbb{Z}^k} W_0(x,n)$.}\\
   @VVV\\
   \text{A tiling-like band-limited map $\Phi(x): \mathbb{C}^k\to \mathbb{C}^k$,} 
   \\ \text{which is equivalent to a band-limited function $g_2(x)$.} \\
   @VVV \\
   \text{A tiling $\mathbb{R}^k = \bigcup_{n\in \mathbb{Z}^k} W(x,n)$.} \\
   @VVV \\
   \text{Weight functions $w(x,n)\in [0,1]^{\mathbb{Z}^k}$ for $n\in \mathbb{Z}^k$,}
   \\ \text{which form a social welfare system among tiles $W(x,n)$}.
  \end{CD}
\end{equation*}

Take $x\in X$. First we construct a tiling $\mathbb{R}^k = \bigcup_{n\in \mathbb{Z}^k} W_0(x,n)$ from the point $x$.
We use the marker property assumption only here.
This tiling is \textit{not} used directly in the proof of Proposition \ref{prop: main proposition}.
Instead we construct a certain holomorphic map $\Phi(x): \mathbb{C}^k\to \mathbb{C}^k$ called a 
``tiling-like band-limited map'' from the tiles $W_0(x,n)$.
Write $\Phi(x)= (\Phi(x)_1,\dots,\Phi(x)_k)$. 
The Fourier transforms of $\Phi(x)_1|_{\mathbb{R}^k} ,\dots, \Phi(x)_k|_{\mathbb{R}^k}$
and their complex conjugates are adjusted to have disjoint supports.
Then the band-limited function 
\[  g_2(x) \overset{\mathrm{def}}{=} \sum_{i=1}^k \mathrm{Re}\left(\Phi(x)_i\right) \]
becomes ``equivalent'' to $\Phi(x)$.
Namely if $g_2(x)=g_2(y)$ for some $x,y\in X$ then $\Phi(x)=\Phi(y)$.
(Indeed we will define $g_2(x)$ in Section \ref{section: proof of Main Proposition} by multiplying a small constant to the above function
so that the norm of $g_2(x)$ becomes sufficiently small.)

Let us give some more heuristics concerning the construction of $\Phi(x)$. Consider the function 
\[\tilde{\Theta}_{L}=(\sin\frac{\pi z_{1}}{L},\cdots, \sin\frac{\pi z_{k}}{L}): \mathbb{R}^{k}\to\mathbb{R}^{k}\]
for some positive integer $L$ thought to be very big. 
Every entry $\sin (\pi z_i/L)$ of the map $\tilde{\Theta}_L$ 
vanishes on $L\mathbb{Z}^{k}\subset \mathbb{R}^{k}$ and belongs to $B(\frac{1}{L},\cdots,\frac{1}{L})$.  Suppose we are given a family of tilings 
$W_{0}(x)=\{W_{0}(x,n)\}_{n\in\mathbb{Z}^{k}}$ of $\mathbb{R}^{k}$ depending continuously\footnote{More precisely, for $\varepsilon>0$ and for each $x\in X$ and $n\in \mathbb{Z}^k$ with $\mathrm{Int}\, W_{0}(x,n)\neq \emptyset$, if $y\in X$ is sufficiently close to $x$, then the Hausdorff distance between $W_{0}(x,n)$ and $W_{0}(y,n)$ is smaller than $\varepsilon$.} and $\mathbb{Z}^{k}$-equivariantly (i.e. $W_{0}(T^{n}x,m)=-n+W_{0}(x,n+m)$) on $x\in X$.
We would like to replace this family of tilings by a family of Voronoi tilings\footnote{A Voronoi tiling generated by a discrete set $A\subset\mathbb{R}^{k}$ is a partitioning of $\mathbb{R}^{k}$ into regions based on the closest distance to points in $A$. For more details, see Subsection \ref{subsection: dynamical Voronoi diagram}.} generated by the zeros of a continuous function
$F: X\times \mathbb{R}^{k}\to\mathbb{R}^k$, where $F$ has the following properties:
\begin{description}
  \item[(1) Band-limited] Every entry of $F(x,\cdot)$ belongs to $B(\frac{1}{L},\cdots,\frac{1}{L})$.
  \item[(2) Equivariance] $F(T^{n}x,z)=F(x,z+n)$ for any $x\in X$, $z\in\mathbb{R}^{k}$ and $n\in\mathbb{Z}^{k}$.
\end{description}
A candidate for $F$ is the following function
\[X\times \mathbb{R}^{k}\to\mathbb{R}^k, \quad (x,t)\mapsto\sum_{n\in\mathbb{Z}^{k}} \widetilde{\Theta}_{L}(t-n)\>1_{W_{0}(x,n)}(t).
\]
This function is equivariant but unfortunately not continuous. However a small change (for details see Subsection \ref{subsection: tiling-like band-limited functions}) makes it band-limited (and continuous).

\begin{figure}[h]
    \centering
    \includegraphics[width=4.0in]{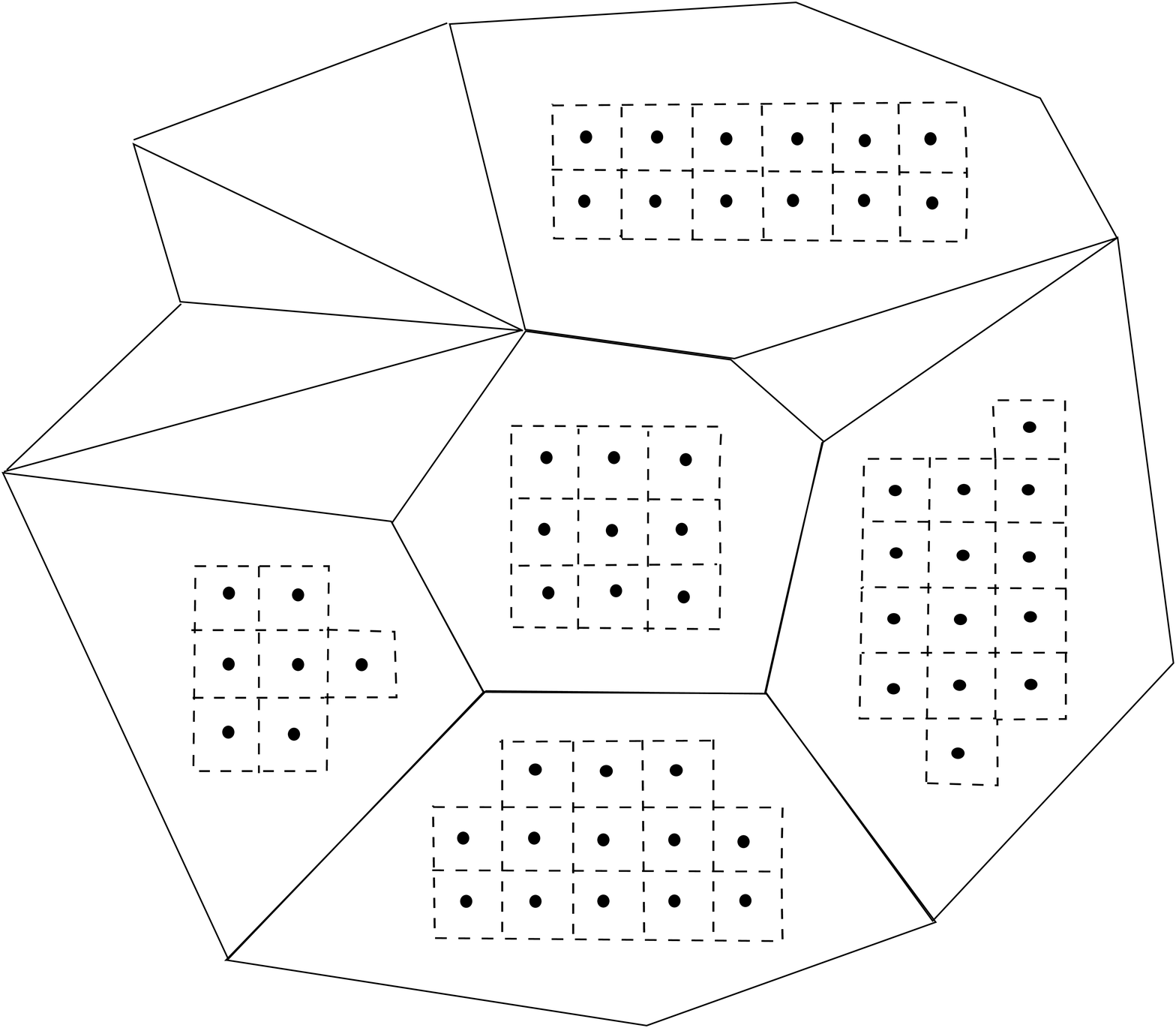}
    \caption{The original tiling $W_{0}(x,n)$: ------;
    The zeros of $F$ in the ``deep" interior of ``big" tiles: $\bullet$;
    ``Regular" part of the new tiling $W(x,n)$: - - - - -.} 
    \label{fig: newtiling}
\end{figure}

The most simple way of generating Voronoi tilings would be to use the zeros of $F(x,\cdot)$ as generating sets for the Voronoi tiling, i.e. $W(x)=\text{Voronoi}\big(\text{zeros of } F(x,\cdot)\big)$.
However, the zeros of band-limited maps behave rather wildly in general, e.g. they are not isolated and form a 
positive dimensional variety in general.
We will thus instead adopt a more elaborate scheme. 
We give a weight $v(x,n)\in[0,1]$ to every $x\in X$ and $n\in\mathbb{Z}^{k}$, which is constructed from
(some appropriate) zeros of $F(x,\cdot)$ in a neighborhood of $n$.
We use these weights to generate a family of indexed Voronoi tilings $\{W(x,n)\}_{n\in\mathbb{Z}^{k}}$. The exact procedure is described in Subsection \ref{subsection: weight functions from tiling-like band-limited functions}. For big tiles and big $L>0$ the picture schematically looks like Figure \ref{fig: newtiling}.

The new tiling is used in the proof of Proposition \ref{prop: main proposition} in Section \ref{section: proof of Main Proposition}.
We construct a certain ``weight function'' (not to be confused with $v(x,n)$ of the previous paragraph)
\[ w(x) = \left(w(x, n)\right)_{n\in \mathbb{Z}^k} \in \left([0,1]^{\mathbb{Z}^k}\right)^{\mathbb{Z}^k} \]
from the tiles $W(x,n)$.
This will play a role of a social welfare system among the tiles $W(x,n)$.
The tiles $W(x,n)$ and functions $w(x,n)$ are constructed from $\Phi(x)$.
So all their informations are encoded into $\Phi(x)$, which is equivalent to 
the band-limited function $g_2(x)$.

Subsection \ref{subsection: quantitative persistence of zero points}
is a function-theoretic preparation for the construction of tiling-like band-limited maps.
Subsection \ref{subsection: tiling-like band-limited functions} constructs a tiling-like band-limited map from a tiling of $\mathbb{R}^k$
(not necessarily coming from a dynamical system).
Subsection \ref{subsection: dynamical Voronoi diagram} describes a construction of the first tiling 
$\mathbb{R}^k = \bigcup_{n\in \mathbb{Z}^k} W_0(x,n)$.
Subsection \ref{subsection: weight functions from tiling-like band-limited functions}
uses the preparations in Subsections \ref{subsection: quantitative persistence of zero points}, 
\ref{subsection: tiling-like band-limited functions} and \ref{subsection: dynamical Voronoi diagram} and 
constructs $\Phi(x)$, $W(x,n)$ and $w(x,n)$.

\subsection{Quantitative persistence of zero points}  \label{subsection: quantitative persistence of zero points}

We start from the following elementary observation:
Let $f:\mathbb{C}^k\to \mathbb{C}^k$ be a holomorphic map such that $f(0)=0$ and the derivative $df_0$ at the origin is
an invertible matrix.
For any neighborhood $U$ of the origin in $\mathbb{C}^k$,
if $g:\mathbb{C}^k\to \mathbb{C}^k$ is a holomorphic map such that $|f(z)-g(z)|$ is sufficiently small over the unit ball $B_1$,
then there exists $z\in U$ such that $g(z)=0$ and the derivative $dg_z$ at $z$ is also invertible.
The purpose of this subsection is to develop a \textit{quantitative} version of this observation
for some very special $f$ (i.e. the function $\Theta_L$ introduced below).
Here ``quantitative'' means that we would like to specify \textit{how small $|g(x)-f(x)|$ should be}
and \textit{how much non-degenerate $dg_z$ is}.

Let $r>0$. For $u\in \mathbb{C}$ we define $D_r(u)$ as the closed disk of radius $r$ around $u$ in $\mathbb{C}$.
We set $D_r = D_r(0)$.
For $u=(u_1,\dots, u_k)\in \mathbb{C}^k$ we define the polydisc $D_r^k(u)$ in $\mathbb{C}^k$ by
$D_r^k(u) = D_r(u_1)\times \dots \times D_r(u_k)$. 
We set $D_r^k = D_r^k(0)$.

First we study the sine function $\sin z$ in the complex domain.
Consider the image $\sin(\partial D_{\pi/2})$ of the circle $\partial D_{\pi/2}$ of radius $\pi/2$ under the map 
$\sin:\mathbb{C}\to \mathbb{C}$.
This curve does not contain the origin and its rotation number around the origin is one:
\[   \frac{1}{2\pi\sqrt{-1}} \int_{\sin(\partial D_{\pi/2})} \frac{dz}{z} 
    = \frac{1}{2\pi \sqrt{-1}} \int_{\partial D_{\pi/2}} \frac{\cos z}{\sin z} \, dz = 1. \]
This implies that the map 
\[ \sin_* : H_1\left(D_{\pi/2}\setminus \{0\}\right) \to H_1\left(\mathbb{C}\setminus \{0\}\right) \]
is an isomorphism. Here $H_1(\cdot)$ is the standard homology group of $\mathbb{Z}$-coefficients.
(Both sides of the above are isomorphic to $\mathbb{Z}$.)
The map $\sin_*$ is the homomorphism induced by the map $\sin: D_{\pi/2}\setminus\{0\}\to \mathbb{C}\setminus \{0\}$.
The natural maps 
\[ H_2\left(D_{\pi/2}, D_{\pi/2} \setminus\{0\}\right) \to H_1\left(D_{\pi/2} \setminus\{0\}\right), \quad 
     H_2\left(\mathbb{C}, \mathbb{C}\setminus \{0\}\right) \to H_1\left(\mathbb{C}\setminus \{0\}\right) \]
are isomorphic by the canonical long exact sequence (\cite[Theorem 2.16]{Hatcher}). It follows by composition that 
\[ \sin_*: H_2\left(D_{\pi/2}, D_{\pi/2} \setminus \{0\}\right) \to H_2\left(\mathbb{C},\mathbb{C}\setminus \{0\}\right)  \]
is also an isomorphism.
In the same way, for any $n\in \mathbb{Z}$, the map 
\[  \sin_*: H_2\left(D_{\pi/2}(\pi n), D_{\pi/2}(\pi n) \setminus \{\pi n\}\right) \to H_2\left(\mathbb{C},\mathbb{C}\setminus \{0\}\right)  \]
is an isomorphism.

\begin{definition}  \label{definition: building block of tiling function}
 \begin{enumerate}
  \item Recall that we introduced positive numbers $a_1,\dots, a_k$ and $\delta$ in Subsection \ref{subsection: fixing some notations}.
          We set $b_i = a_i + \delta/2$.
          For $L>1$ we define $\Theta_L:\mathbb{C}^k\to \mathbb{C}^k$ by 
          \[ \Theta_L(z_1,\dots,z_k) = \left(e^{\pi \sqrt{-1}\, b_1 z_1} \sin (\pi z_1/L), \dots, e^{\pi\sqrt{-1}\, b_k z_k} \sin (\pi z_k/L)\right). \]
         This map vanishes on $L\mathbb{Z}^k$.
         The Fourier transform of the $i$-th entry 
         \[   e^{\pi\sqrt{-1}\, b_i t_i} \sin (\pi t_i/L), \quad (t = (t_1,\dots, t_k) \in \mathbb{R}^k)    \]
         of $\Theta_L|_{\mathbb{R}^k}$ (the restriction of $\Theta_L$ to $\mathbb{R}^k$)
         is supported in 
         \[  \{0\}^{i-1}\times  \left[ \frac{b_i}{2}-\frac{1}{2L}, \frac{b_i}{2} + \frac{1}{2L}\right] \times \{0\}^{k-i},  \quad (1\leq i\leq k). \]
         We will choose a very large $L$ later. So this line segment will be very small.
  \item  Let $A$ be a $k\times k$ matrix.      
           We set 
         \[  \nu(A) = \min_{z\in \mathbb{C}^k, |z|=1} |A z|. \]
         The matrix $A$ is invertible if and only if $\nu(A)>0$.
         So $\nu(A)$ quantifies the amount of non-degeneracy of $A$.         
         For another $k\times k$ matrix $B$ we have 
         \begin{equation}  \label{eq: continuity of nu(A)}
             |\nu(A)-\nu(B)| \leq |A-B| \quad (\text{the operator norm of $A-B$}). 
         \end{equation}    
         In particular $\nu(A)$ depends continuously on $A$.
 \end{enumerate} 
\end{definition}

\begin{lemma}  \label{lemma: homological property of Theta_L}
For every $n = (n_1,\dots,n_k) \in \mathbb{Z}^k$ the homomorphism 
\[ \left(\Theta_L\right)_*: H_{2k}\left(D_{L/2}^k(Ln), D_{L/2}^k(Ln)\setminus \{Ln\}\right) 
   \to H_{2k}\left(\mathbb{C}^k,\mathbb{C}^k\setminus\{0\}\right) \]
 is isomorphic. (Notice that the both sides are isomorphic to $\mathbb{Z}$.)
\end{lemma}

\begin{proof}
We define a homotopy $f_t: \left(D_{L/2}^k(Ln), D_{L/2}^k(Ln)\setminus \{Ln\}\right) \to \left(\mathbb{C}^k,\mathbb{C}^k\setminus\{0\}\right)$
for $0\leq t\leq 1$ by 
\[ f_t(z_1,\dots,z_k) =  \left(e^{\pi \sqrt{-1}\, t\, b_1 z_1} \sin (\pi z_1/L), \dots, e^{\pi\sqrt{-1}\, t\, b_k z_k} \sin (\pi z_k/L)\right). \]
$f_1 = \Theta_L$. So it is enough to show that $f_0$ induces an isomorphism on the $2k$-th homology groups (see \cite[Theorem 2.10]{Hatcher}).
The map $f_0$ is given by 
\[ f_0(z_1,\dots,z_k) = \left(\sin (\pi z_1/L), \dots, \sin (\pi z_k/L)\right). \]
Set $\varphi(z) = \sin(\pi z/L)$ for $z\in \mathbb{C}$.
We have natural isomorphisms 
by the K\"{u}nneth theorem (Spanier \cite[p. 235, 10 Theorem]{Spanier}):
\begin{equation*}
   \begin{split}
     H_{2k}\left(D_{L/2}^k(Ln), D_{L/2}^k(Ln)\setminus \{Ln\}\right)  
     &\cong  \bigotimes_{i=1}^k H_2\left(D_{L/2}(Ln_i), D_{L/2}(Ln_i)\setminus \{L n_i\}  \right)  \\
      H_{2k}\left(\mathbb{C}^k,\mathbb{C}^k\setminus\{0\}\right) 
      &\cong  \bigotimes_{i=1}^k H_2\left(\mathbb{C},\mathbb{C}\setminus \{0\}\right).
    \end{split}
\end{equation*}      
Under these isomorphisms, the homomorphism $(f_0)_*$ is identified with 
\[ \otimes_{i=1}^k \varphi_* :     \bigotimes_{i=1}^k H_2\left(D_{L/2}(Ln_i), D_{L/2}(Ln_i)\setminus \{L n_i\}  \right) \to 
    \bigotimes_{i=1}^k H_2\left(\mathbb{C},\mathbb{C}\setminus \{0\}\right). \]
By the argument in the beginning of this subsection, each factor 
\[ \varphi_*: H_2\left(D_{L/2}(Ln_i), D_{L/2}(Ln_i)\setminus \{L n_i\}  \right) \to  H_2\left(\mathbb{C},\mathbb{C}\setminus \{0\}\right)  \]
is isomorphic. Thus $(f_0)_*$ is an isomorphism. 
\end{proof}

The derivative $(d\Theta_L)_z$ at $z=(z_1,\dots,z_k)$ is a diagonal matrix whose $(i,i)$-th entry is given by 
\[ \pi e^{\pi\sqrt{-1}\, b_i z_i}\left(\sqrt{-1}\, b_i \sin(\pi z_i/L) + \frac{\cos(\pi z_i/L)}{L}\right), \quad (1\leq i\leq k).  \]

\begin{notation}  \label{notation: choice of r_1}
We choose $0<r_1< 1/4$ so small that for all  $1\leq i\leq k$ and $z\in D_{r_1}$ 
\[  \pi \left| e^{\pi\sqrt{-1}\, b_i z}\left(\sqrt{-1}\, b_i \sin(\pi z/L) + \frac{\cos(\pi z/L)}{L}\right)\right| > \frac{3}{L}. \]
Since for small $z$
\[  \sin(\pi z/L) \approx  \frac{\pi z}{L},  \quad    \frac{\cos(\pi z/L)}{L} \approx  \frac{1}{L},   \]
we can choose $r_1$ independent of $L>1$.
(This independence is a plausible fact but it is not really needed in the argument below.)
\end{notation}

It immediately follows from the definition that for all $n\in \mathbb{Z}^k$ and $z\in D_{r_1}^k(L n)$ 
\begin{equation}  \label{eq: how non-degenerate dTheta is}
    \nu\left((d\Theta_L)_z\right)  > \frac{3}{L}. 
\end{equation}    

\begin{notation}[The threshold $\theta_L$] \label{notation: theta_L}
We define $\Omega\subset \mathbb{C}^k$ as the set of $z=(z_1,\dots,z_k)$ satisfying $|\mathrm{Im}(z_i)|\leq 1$ for all 
$1\leq i\leq k$.
We define $\theta_L>0$ as the minimum of the following two quantities:
\[        k^{-1/2} (3/4)^{k+1} \frac{1}{L},   \quad         \inf \left\{ |\Theta_L(z)| \, \middle|
          z \in \Omega \setminus \bigcup_{n\in \mathbb{Z}^k} D_{r_1}^k(L n)\right\}.  \]
\end{notation}

The next lemma is the main result of this subsection.

\begin{lemma}[Persistence of zero points of $\Theta_L$]  \label{lemma: persistence of zero points of Theta_L}
Let $f:\mathbb{C}^k\to \mathbb{C}^k$ be a holomorphic map.
Let $W \subset \mathbb{C}^k$ be a subset and suppose $|f(z)-\Theta_L(z)| < \theta_L$ for all $z\in W$.
Then the following (1) and (2) hold.

  \begin{enumerate}
    \item \[ \{z  \in W\cap \Omega |\, f(z)=0 \} 
             \subset \bigcup_{n\in \mathbb{Z}^k} D_{r_1}^k(L n). \]
    \item For every $n\in \mathbb{Z}^k$ with $D_{1}^k(L n) \subset W$ there exists $z\in D_{r_1}^k(L n)$ satisfying 
            \[  f(z)=0, \quad \nu(df_z) > \frac{2}{L}. \]
  \end{enumerate}
\end{lemma}

\begin{proof}
(1) is a direct consequence of the definition of $\theta_L$.
So we consider (2).
For the notational simplicity, assume $n=0$.
It follows from the definition of $\theta_L$ that 
\[ \max_{z\in \partial D_{r_1}^k} |f(z)-\Theta_L(z)| < \min_{z\in \partial D_{r_1}^k} |\Theta_L(z)|, \]
where $\partial D_{r_1}^k$ is the boundary of $D_{r_1}^k$.
(So in particular its $(2k-1)$-th homology group is isomorphic to $\mathbb{Z}$.)
Then we can define a homotopy $f_t:\partial D_{r_1}^k \to \mathbb{C}^k\setminus \{0\}$ between $\Theta_L$ and $f$ by 
\[  f_t(z) = (1-t) \Theta_L(z)  + t f(z), \quad (0\leq t\leq 1). \]
Here the point is that $f_t$ has no zero on $\partial D_{r_1}^k$.
It follows from Lemma \ref{lemma: homological property of Theta_L} that 
$f_* = (\Theta_L)_*: H_{2k-1}(\partial D_{r_1}^k) \to H_{2k-1}\left(\mathbb{C}^k\setminus \{0\}\right)$ is isomorphic.
Then $f$ must attain zero at some (interior) point of $D_{r_1}^k$.
Otherwise the map $f_*: H_{2k-1}\left(\partial D_{r_1}^k\right) \to H_{2k-1}\left(\mathbb{C}^k\setminus \{0\}\right)$
is equal to the composition of the maps
\[  H_{2k-1} \left(\partial D_{r_1}^k\right)  \to H_{2k-1}\left(D_{r_1}^k\right) \overset{f_*}{\to} H_{2k-1}\left(\mathbb{C}^k\setminus\{0\}\right).  \] 
But $H_{2k-1}\left(D_{r_1}^k\right) =0$ implies that the map 
$f_*: H_{2k-1}\left(\partial D_{r_1}^k\right) \to H_{2k-1}\left(\mathbb{C}^k\setminus \{0\}\right)$ is zero.
This is a contradiction.

The rest of the work is to prove the statement about the derivative of $f$ at zero points.
By (\ref{eq: continuity of nu(A)}) and (\ref{eq: how non-degenerate dTheta is}) 
it is enough to prove that for all $z\in D_{r_1}^k$ the operator norm 
of $df_z-(d\Theta_L)_z$ is smaller than $1/L$.
Let $z=(z_1,\dots,z_k)\in D_{r_1}^k$.
By the Cauchy integration formula,
\[ f(z) -  \Theta_L(z) = \frac{1}{(2\pi\sqrt{-1})^k} \int_{\partial D_1}\dots \int_{\partial D_1}
   \frac{f(w)-\Theta_L(w)}{(w_1-z_1)\dots (w_k-z_k)} \, dw_1\dots dw_k. \]
For $u=(u_1,\dots,u_k)\in \mathbb{C}^k$ with $|u|=1$, the vector $df_z(u)-(d\Theta_L)_z(u)$ is given by 
\[  \frac{1}{(2\pi\sqrt{-1})^k} \sum_{i=1}^k u_i \int_{\partial D_1}\dots \int_{\partial D_1}
   \frac{f(w)-\Theta_L(w)}{(w_1-z_1)\dots (w_i-z_i)^2\ \dots (w_k-z_k)} \, dw_1\dots dw_k. \]
For $z_i\in D_{r_1}$ and $w_i\in \partial D_1$ we have $|w_i-z_i| > 3/4$ because $r_1<1/4$.
Hence 
\begin{equation*}
   \begin{split}
    \left| df_z(u)- (d\Theta_L)_z(u)\right|  &<  \sum_{i=1}^k |u_i| (4/3)^{k+1} \theta_L \\
    &\leq \sqrt{k} (4/3)^{k+1} \theta_L \leq \frac{1}{L}, \quad \left(\text{by $\theta_L\leq k^{-1/2}(3/4)^{k+1} \frac{1}{L}$}\right).
   \end{split}  
\end{equation*}
This proves $\left|df_z-(d\Theta_L)\right| < 1/L$.
\end{proof}

\subsection{Tiling-like band-limited maps}  \label{subsection: tiling-like band-limited functions}

We choose a rapidly decreasing function $\chi_1:\mathbb{R}^k\to \mathbb{R}$ satisfying
$\supp \, \mathcal{F}(\chi_1) \subset B_{\delta/8}$ and 
\begin{equation}  \label{eq: integral of chi_1}
   \mathcal{F}(\chi_1)(0) = \int_{\mathbb{R}^k} \chi_1(t) \, dt_1\dots dt_k = 1.
\end{equation}

\begin{notation} \label{notation: K_1}
  We set 
  \[ K_1 = \int_{\mathbb{R}^k} |\chi_1(t)|\, dt_1\dots dt_k. \]
\end{notation}

The function $\chi_1$ can be extended holomorphically over $\mathbb{C}^k$ (see Lemma \ref{lemma: Paley--Wiener}) by 
\[  \chi_1(z) = \int_{B_{\delta/8}} \mathcal{F}(\chi_1)(\xi)\,  e^{2\pi\sqrt{-1} \xi\cdot z} \, d\xi_1\dots d\xi_k, \quad 
     (z=(z_1,\dots,z_k)\in \mathbb{C}^k). \]
For any polynomial $P(z_1,\dots,z_k)$, by the integration by parts,
\[ P(z_1,\dots,z_k) \chi_1(z) = \int_{B_{\delta/8}}
    \left(P\left(\frac{\sqrt{-1}}{2\pi}\frac{\partial}{\partial \xi_1},\dots, \frac{\sqrt{-1}}{2\pi} \frac{\partial}{\partial \xi_k}\right) 
    \mathcal{F}(\chi_1)(\xi)\right) e^{2\pi \sqrt{-1}\xi\cdot z}\, d\xi_1\dots d\xi_k. \]
Hence there exists a positive constant $\mathrm{const}_P$ depending on $P$ and satisfying 
\[ \left|P(z_1,\dots,z_k) \chi_1(z)\right| \leq \mathrm{const}_P \cdot e^{(\pi \delta/4) |\mathrm{Im}(z)|},   \quad 
     \left(|\mathrm{Im}(z)| = \sqrt{\sum_{i=1}^k |\mathrm{Im}(z_i)|^2}\right).  \]
 In particular for each $l\geq 0$ there exists a constant $\mathrm{const}_l$ satisfying
 \begin{equation}  \label{eq: rapidly decreasing in complex domain}
     |\chi_1(z)| \leq \mathrm{const}_l \cdot (1+|z|)^{-l} \cdot e^{(\pi\delta/4)|\mathrm{Im}(z)|}.  
 \end{equation}      
It follows that the integral 
\[  \int_{\mathbb{R}^k} \chi_1(z-t)\, dt_1\dots dt_k \]
converges and becomes a holomorphic function in $z\in \mathbb{C}^k$.
It is constantly equal to one for $z\in \mathbb{R}^k$ by (\ref{eq: integral of chi_1}). 
Therefore indeed it is identically equal to one:
\begin{equation} \label{eq: integral of chi_1, complex domain}
   \int_{\mathbb{R}^k} \chi_1(z-t) \, dt_1\dots dt_k = 1, \quad (z\in \mathbb{C}^k).
\end{equation}

Let $\mathcal{W} = \{W_n\}_{n\in \mathbb{Z}^k}$ be a set of bounded convex sets of $\mathbb{R}^k$ (indexed by $\mathbb{Z}^k$) such that 
it \textbf{tiles the whole space} $\mathbb{R}^k$, namely 
\begin{itemize}
  \item $\mathbb{R}^k = \bigcup_{n\in \mathbb{Z}^k} W_n$.
  \item For any two distinct $m$ and $n$ in $\mathbb{Z}^k$ the sets $W_m$ and $W_n$ intersect at most on their boundaries
          $\partial W_m$ and $\partial W_n$:
          $W_m\cap W_n = \partial W_m \cap \partial W_n$.
\end{itemize}
Notice that some $W_n$ may be empty.
For $L>1$
we define a \textbf{tiling-like band-limited map} $\Phi_L\left(\mathcal{W}\right): \mathbb{C}^k\to \mathbb{C}^k$
by 
\[ \Phi_L\left(\mathcal{W}\right)(z) = \sum_{n\in \mathbb{Z}^k}  \Theta_L(z-n)
    \int_{W_n} \chi_1(z-t)\, dt_1\dots dt_k. \]
As was already explained before, intuitively we would like to consider 
\[  \sum_{n\in \mathbb{Z}^k} \Theta_L(t-n) \, 1_{W_n}(t), \quad (t\in \mathbb{R}^k). \]
Namely we would like to ``paint'' the map $\Theta_L(t-n)$ over each tile $W_n$.
But this is not even continuous in $t$. So instead we defined $\Phi_L\left(\mathcal{W}\right)$ as in the above.

It follows from $\norm{\Theta_L}_{L^\infty(\mathbb{R}^k)} = \sqrt{k}$ and Notation \ref{notation: K_1} that 
the map $\Phi_L\left(\mathcal{W}\right)(t)$ is bounded over $t\in \mathbb{R}^k$:
\begin{equation*}
  \begin{split}
         |\Phi_L\left(\mathcal{W}\right)(t)|  &\leq \sum_{n\in \mathbb{Z}^k} \norm{\Theta_L}_{L^\infty(\mathbb{R}^k)} 
                                                    \int_{W_n} |\chi_1(t-x)|\, dx_1\dots dx_k \\
        &= \norm{\Theta_L}_{L^\infty(\mathbb{R}^k)} \int_{\mathbb{R}^k} |\chi_1(t-x)|\, dx_1\dots dx_k \\
        &= \sqrt{k} K_1.
  \end{split}      
\end{equation*}

The Fourier transform of the $i$-th entry of $\Phi_L(\mathcal{W})|_{\mathbb{R}^k}$ is supported in 
\[ \left[-\frac{\delta}{8}, \frac{\delta}{8}  \right]^{i-1}\times 
    \left[\frac{b_i}{2}-\frac{1}{2L} -\frac{\delta}{8},\, \frac{b_i}{2} + \frac{1}{2L} + \frac{\delta}{8}\right]
    \times \left[-\frac{\delta}{8}, \frac{\delta}{8}  \right]^{k-i},   \quad (1\leq i\leq k). \]
Since $b_i=a_i+\delta/2$, if $L>4/\delta$ then this is contained in 
\begin{equation*}  
   \left(-\frac{\delta}{4},\frac{\delta}{4}\right)^{i-1} \times 
    \left(\frac{a_i}{2}, \frac{a_i}{2}+\frac{\delta}{2}\right)\times \left(-\frac{\delta}{4},\frac{\delta}{4}\right)^{k-i}
\end{equation*}
Since we assumed $\delta < \min(a_1,\dots,a_k)$ in (\ref{eq: choice of a_i and delta})
in Subsection \ref{subsection: fixing some notations}, these $k$ sets are disjoint with each other.

\begin{notation}  \label{notation: introduction of E}
    Recall that $\Omega \subset \mathbb{C}^k$ is the set of $z=(z_1,\dots,z_k)$ satisfying $|\mathrm{Im}(z_i)|\leq 1$ for all $1\leq i\leq k$.
    By (\ref{eq: rapidly decreasing in complex domain}) we can choose $E=E(L)>0$ such that for all $z\in \Omega$
            \begin{equation*}
             \norm{\Theta_L}_{L^\infty(\Omega)} 
             \int_{\mathbb{R}^k\setminus B_{E}\left(\mathrm{Re}(z)\right)} |\chi_1(z-t)|\, dt_1\dots dt_k < \frac{\theta_L}{2},
            \end{equation*}
            where $\theta_L$ is the constant introduced in Notation \ref{notation: theta_L}. 
\end{notation}

We recall the notations introduced in Subsection \ref{subsection: a lemma on convex sets}:
For $r>0$ and $A\subset \mathbb{R}^k$, we define $\partial_r A$ as the set of all $t\in \mathbb{R}^k$ satisfying 
$B_r(t)\cap A\neq \emptyset$ and $B_r(t)\cap \left(\mathbb{R}^k\setminus A\right) \neq \emptyset$. 
We set $\mathrm{Int}_r A= A\setminus \partial _r A$, namely 
$\mathrm{Int}_r A$ is the set of $t\in A$ satisfying $B_r(t)\subset A$.

\begin{lemma}  \label{lemma: Phi and Theta are close}
Let $n\in \mathbb{Z}^k$ and $z\in \Omega$.
If $\mathrm{Re}(z)\in \mathrm{Int}_{E} W_n$ then 
\[ \left|\Phi_L\left(\mathcal{W}\right)(z) - \Theta_L(z-n)\right| < \theta_L. \]
\end{lemma}

\begin{proof}
For simplicity of the notation we assume $n=0$. By (\ref{eq: integral of chi_1, complex domain})
\begin{equation*}
   \begin{split}
    \Phi_L\left(\mathcal{W}\right)(z) - \Theta_L(z)  = &\sum_{m \neq 0} \Theta_L(z-m) \int_{W_m} \chi_1(z-t) \, dt_1\dots dt_k \\
    & - \Theta_L(z) \int_{\mathbb{R}^k\setminus W_0}\chi_1(z-t) \, dt_1\dots dt_k. 
   \end{split}
\end{equation*}    
Hence 
\[  \left|\Phi_L\left(\mathcal{W}\right)(z) - \Theta_L(z)\right| 
     \leq 2\norm{\Theta_L}_{L^\infty(\Omega)} \int_{\mathbb{R}^k\setminus W_0} |\chi_1(z-t)|\, dt_1\dots dt_k. \]
This is smaller than $\theta_L$ by the definition of $E$.   
\end{proof}

The next lemma is the summary of this subsection.

\begin{lemma}[Main properties of $\Phi_L\left(\mathcal{W}\right)$]  \label{lemma: main properties of Phi_L}
  \begin{enumerate}
    \item The restriction of $\Phi_L\left(\mathcal{W}\right)$ to $\mathbb{R}^k$ is a bounded continuous map such that  
            $\norm{\Phi_L\left(\mathcal{W}\right)}_{L^\infty(\mathbb{R}^k)} \leq \sqrt{k}K_1$ and that
            if $L>4/\delta$ then the Fourier transform of the $i$-th entry of $\Phi_L\left(\mathcal{W}\right)|_{\mathbb{R}^k}$ is 
            supported in 
            \[    \left(-\frac{\delta}{4},\frac{\delta}{4}\right)^{i-1} \times 
            \left(\frac{a_i}{2}, \frac{a_i}{2}+\frac{\delta}{2}\right)\times \left(-\frac{\delta}{4},\frac{\delta}{4}\right)^{k-i}, \quad 
              (1\leq i\leq k), \]
            which are disjoint with each other.
    \item Let $n\in \mathbb{Z}^k$ and $z\in \Omega$. If $\mathrm{Re}(z)\in \mathrm{Int}_E W_n$ and $\Phi_L\left(\mathcal{W}\right)(z)=0$
            then there exists $m\in \mathbb{Z}^k$ satisfying $z\in D_{r_1}^k(n+Lm)$.
    \item Let $n, m\in \mathbb{Z}^k$ with $n+Lm \in \mathrm{Int}_{E+\sqrt{k}} W_n$.
            Then there exists $z\in D_{r_1}^k(n+Lm)$ satisfying  
           \[    \Phi_L\left(\mathcal{W}\right)(z)=0, \quad \nu\left(\left(d\Phi_L\left(\mathcal{W}\right)\right)_z\right) > \frac{2}{L}. \]
  \end{enumerate}
\end{lemma}

\begin{proof}
(1) was already proved. (2) follows from Lemmas \ref{lemma: persistence of zero points of Theta_L} (1) and 
\ref{lemma: Phi and Theta are close}.
For (3), notice that $n+Lm \in \mathrm{Int}_{E+\sqrt{k}} W_n$ implies 
 $B_{\sqrt{k}}(n+Lm)\subset \mathrm{Int}_E(W_n)$ and that 
every point $z\in D_1^k(n+Lm)$ satisfies $\mathrm{Re}(z)\in \mathrm{Int}_E(W_n)$.
Then (3) follows from  Lemmas \ref{lemma: persistence of zero points of Theta_L} (2) and 
\ref{lemma: Phi and Theta are close}.
\end{proof}

Conditions (2) and (3) in Lemma \ref{lemma: main properties of Phi_L}
are the rigorous formulation of ``good/bad decomposition'' explained in Subsection \ref{subsection: one dimension versus multi-dimension}.
The domain $\Omega$ is decomposed into two regions:
\begin{equation*}
   \begin{split}
   \text{Good region} &= \left\{z\in \Omega \middle| \, \mathrm{Re} z\in \bigcup_{n\in \mathbb{Z}^k} \mathrm{Int}_{E+\sqrt{k}} W_n \right\}, \\
   \text{Bad region}  &=  \left\{z\in \Omega \middle| \, \mathrm{Re} z\in \bigcup_{n\in \mathbb{Z}^k} \partial_{E+\sqrt{k}} W_n \right\}.
   \end{split}
\end{equation*}
The above (2) and (3) mean that we have a good control (for our purpose here) of zero points of $\Phi_L\left(\mathcal{W}\right)$
over the good region.
As we explained in Subsection \ref{subsection: one dimension versus multi-dimension}, the bad region should be ``tiny''.
This means that we need to introduce a tiling $\mathcal{W}$ such that the ``boundary region'' 
\[  \bigcup_{n\in \mathbb{Z}^k} \partial_{E+\sqrt{k}} W_n \]
is sufficiently small. (Indeed we will need a bit more involved condition later; see (\ref{eq: fixing M}) below).
The purpose of the next subsection is to construct such tilings.

\subsection{Dynamical Voronoi diagram}  \label{subsection: dynamical Voronoi diagram}

This subsection is largely a reproduction of \cite[Section 4]{Gutman--Lindenstrauss--Tsukamoto}.
Here we introduce a \textit{dynamically generated Voronoi diagram}.
Our use of Voronoi diagram is conceptually influenced by the works of 
Lightwood \cite{Lightwood 1, Lightwood 2}.

Let $(X,\mathbb{Z}^k,T)$ be a dynamical system having the marker property as we promised 
in Subsection \ref{subsection: fixing some notations}.
Let $M$ be a natural number.
It follows from the marker property that there exists an open set $U\subset X$ satisfying
\[  U\cap T^n U=\emptyset \quad (0<|n|< M), \quad  X = \bigcup_{n\in \mathbb{Z}^k} T^n U. \]
We can find a natural number $M_1\geq M$ and a compact set $F\subset U$ satisfying 
$X=\bigcup_{|n| <M_1} T^n F$.
Choose a continuous function $h:X\to [0,1]$ satisfying $\supp\, h\subset U$ and $h=1$ on $F$.
Then it satisfies 
\begin{equation}  \label{eq: conditions of M_1 and h}
  \left(\supp\, h\right) \cap T^n \left(\supp\, h\right) = \emptyset \quad (0<|n|<M), 
  \quad X= \bigcup_{|n|<M_1} T^n \left(\{x\in X|\, h(x)=1\}\right). 
\end{equation}

For $x\in X$ consider the following discrete set in $\mathbb{R}^{k+1}$:
\[ \left\{ \left(n, \frac{1}{h(T^n x)}\right) \middle| \, n \in \mathbb{Z}^k, h(T^n x) > 0 \right\}. \]
We consider the \textbf{Voronoi diagram} determined by this set.
Namely we define a convex set $V_0(x,n)$ for $n\in \mathbb{Z}^k$ as follows: 
If $h(T^n x)=0$ then we set $V_0(x,n)=\emptyset$.
If $h(T^n x) > 0$ then we define it as the set of all $u\in \mathbb{R}^{k+1}$ satisfying 
\[ \forall m\in \mathbb{Z}^k \text{ with $h(T^m x) > 0$}: \quad 
    \left|u-\left(n, \frac{1}{h(T^n x)}\right) \right| \leq \left|u-\left(m, \frac{1}{h(T^m x)}\right) \right|. \]
The sets $V_0(x,n)$ form a tiling of $\mathbb{R}^{k+1}$:
\[  \mathbb{R}^{k+1}  = \bigcup_{n\in \mathbb{Z}^k} V_0(x,n). \]
Let $\pi_{\mathbb{R}^k}:\mathbb{R}^{k+1}\to \mathbb{R}^k$ be the projection to the first $k$ coordinates 
(i.e. forgetting the last coordinate).
Set $H=\left(M_1+\sqrt{k}\right)^2$ and 
\[ W_0(x,n) = \pi_{\mathbb{R}^k}\left(V_0(x,n)\cap \left(\mathbb{R}^k\times \{-H\}\right)\right). \]
See Figure \ref{fig: Voronoi tiling}.  These form a tiling of $\mathbb{R}^k$:
\[  \mathbb{R}^k = \bigcup_{n\in \mathbb{Z}^k} W_0(x, n). \]
This tiling is $\mathbb{Z}^k$-equivariant in the sense that 
\[  W_0(T^n x, m) = -n + W_0(x, n+m)  = \{-n+u|\, u\in W_0(x,n+m)\}.  \]
The tiles $W_0(x,n)$ depend continuously on $x\in X$ in the Hausdorff topology.
More precisely, for $\varepsilon>0$ and for each $x\in X$ and $n\in \mathbb{Z}^k$ with $\mathrm{Int}\, W_0(x,n)\neq \emptyset$, if $y\in X$ is 
sufficiently close to $x$, then the Hausdorff distance between $W_0(x,n)$ and
$W_0(y,n)$ is smaller than $\varepsilon$.
Here $\mathrm{Int}\, W_0(x,n)$ denotes the interior of $W_0(x,n)$.
(When the tile $W_0(x,n)$ has no interior point, it may vanish after $x$ moves slightly.
But this does not cause any problem.)

\begin{figure}
    \centering
    \includegraphics[,bb= 0 0 300 170]{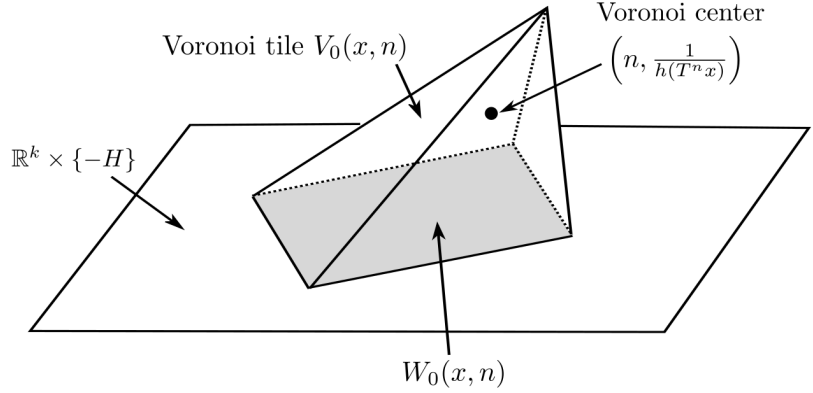}
    \caption{Voronoi tiling. The Voronoi center $\left(n, 1/h(T^n x)\right)$ is located inside $V_0(x,n)$.
    The shadowed region is $W_0(x,n)$.} 
    \label{fig: Voronoi tiling}
\end{figure}

The set $W_0(x,n)$ depends on the choice of the parameter $M$.
So we sometimes write $W_0^M(x,n) = W_0(x,n)$ for clarifying that $M$ is the parameter\footnote{It might look that $M_1$ and $h$
are also parameters.
But we can choose them to be \textit{functions of} $M$. 
Namely for each natural number $M$ we fix a natural number $M_1=M_1(M)$ and a continuous function 
$h= h_M:X\to [0,1]$ satisfying (\ref{eq: conditions of M_1 and h}).
Then only $M$ remains to be a parameter.}.
In the next subsection we consider the tiling-like band-limited map for this tiling.
The following property of $W^M_0(x,n)$ becomes crucial there, which means
that the ``bad region'' is negligible as we explained at the end of the last section.

\begin{lemma}  \label{lemma: basic properties of W_0}
       For any $r>0$
       \[   \lim_{M\to \infty} \left\{ \limsup_{R\to \infty} 
           \left(\sup_{x\in X} \frac{\left| B_{R}\cap \bigcup_{n\in \mathbb{Z}^k} \partial_r W_0^M(x,n)\right|}{\left|B_{R}\right|}\right) \right\}
              = 0. \]       
       Here $|\cdot|$ is the $k$-dimensional volume (Lebesgue measure).       
\end{lemma}

\begin{proof}
We need some auxiliary claims.

\begin{claim}  \label{claim: preliminaries of W_0(x,n)}
Let $x\in X$ and $n\in \mathbb{Z}^k$ with $h(T^n x)>0$.
  \begin{enumerate}
    \item $V_0(x,n)$ contains the ball $B_{M/2}\left(n, 1/h(T^n x)\right)$.
    \item $W_0(x,n) \subset B_{M_1+\sqrt{k}}(n)$.
    \item If $W_0(x,n)\neq \emptyset$ then $h(T^n x)>1/2$.
  \end{enumerate}
\end{claim}
\begin{proof}
It follows from the first property of $h$ in (\ref{eq: conditions of M_1 and h}) that
if $m\in \mathbb{Z}^k$ also satisfies $h(T^m x)>0$ then $|n-m|\geq M$.
Property (1) easily follows from this.
For the proof of (2) and (3), take $u\in W_0(x,n)$.
Let $v\in \mathbb{Z}^k$ be the integer point closest to $u$. 
By the second property of $h$ in (\ref{eq: conditions of M_1 and h}) there exists 
$m\in \mathbb{Z}^k$ satisfying $|m|<M_1$ and $h(T^{v+m}x)=1$.
It follows from the definition of the Voronoi tiling that 
\begin{equation} \label{eq: inequality by the Voronoi condition}
    \left|(u, -H) - \left(n, \frac{1}{h(T^n x)}\right) \right| 
     \leq \left|(u,-H) - (v+m, 1)\right|.  
\end{equation}     
Since $\left|H+1/h(T^n x)\right| \geq H+1$, 
\[ |u-n|\leq |u-v-m| < |u-v|+M_1 < \sqrt{k} + M_1. \]
This proves (2). It follows from (\ref{eq: inequality by the Voronoi condition}) that 
\[ \left(H+\frac{1}{h\left(T^n x\right)}\right)^2 \leq |u-v-m|^2 + (H+1)^2 < \left(\sqrt{k}+M_1\right)^2 + (H+1)^2. \]
Then Property (3) follows:
\[ \frac{1}{h(T^n x)} < \frac{\left(\sqrt{k}+M_1\right)^2}{2H} + 1 + \frac{1}{2H} < \frac{1}{2} + 1 + \frac{1}{2} =2, \]
where we used $H=\left(M_1+\sqrt{k}\right)^2 > 1$.
\end{proof}

Let $\varepsilon>0$. We take $c>1$ satisfying 
\[  c^{-k} > 1-\varepsilon.    \]
We choose $M$ so large that 
\begin{equation} \label{eq: choice of M in Dynamical Voronoi diagram subsection}
   \frac{(c-1)M}{2(c+1)}  > r, \quad  \left(c+\frac{2}{M} \right)^{-k} > 1-\varepsilon.
\end{equation}    

For $x\in X$ and $n\in \mathbb{Z}^k$ we set $W_1^M(x,n) = W_1(x,n) = \pi_{\mathbb{R}^k}\left(V(x,n)\cap \mathbb{R}^k\times \{-cH\}\right)$, 
where $\pi_{\mathbb{R}^k}: \mathbb{R}^{k+1}\to \mathbb{R}^k$ is the projection to the first $k$ coordinates.
It follows from the same proof of Claim \ref{claim: preliminaries of W_0(x,n)} (2) that 
\[  W_1(x,n) \subset B_{M_1+\sqrt{k}}(n).  \]
For $u\in W_1(x,n)$ we denote by $(f(u), -H)$ the intersection point of the hyperplane $\mathbb{R}^k\times \{-H\}$
and the line segment between $(u, -cH)$ and the Voronoi center $\left(n, 1/h(T^n x)\right)$.
It follows from the convexity of $V_0(x,n)$ that $f(u)\in W_0(x,n)$.
See Figure \ref{fig: configuration}.

\begin{figure}
    \centering
    \includegraphics[,bb= 0 0 400 220]{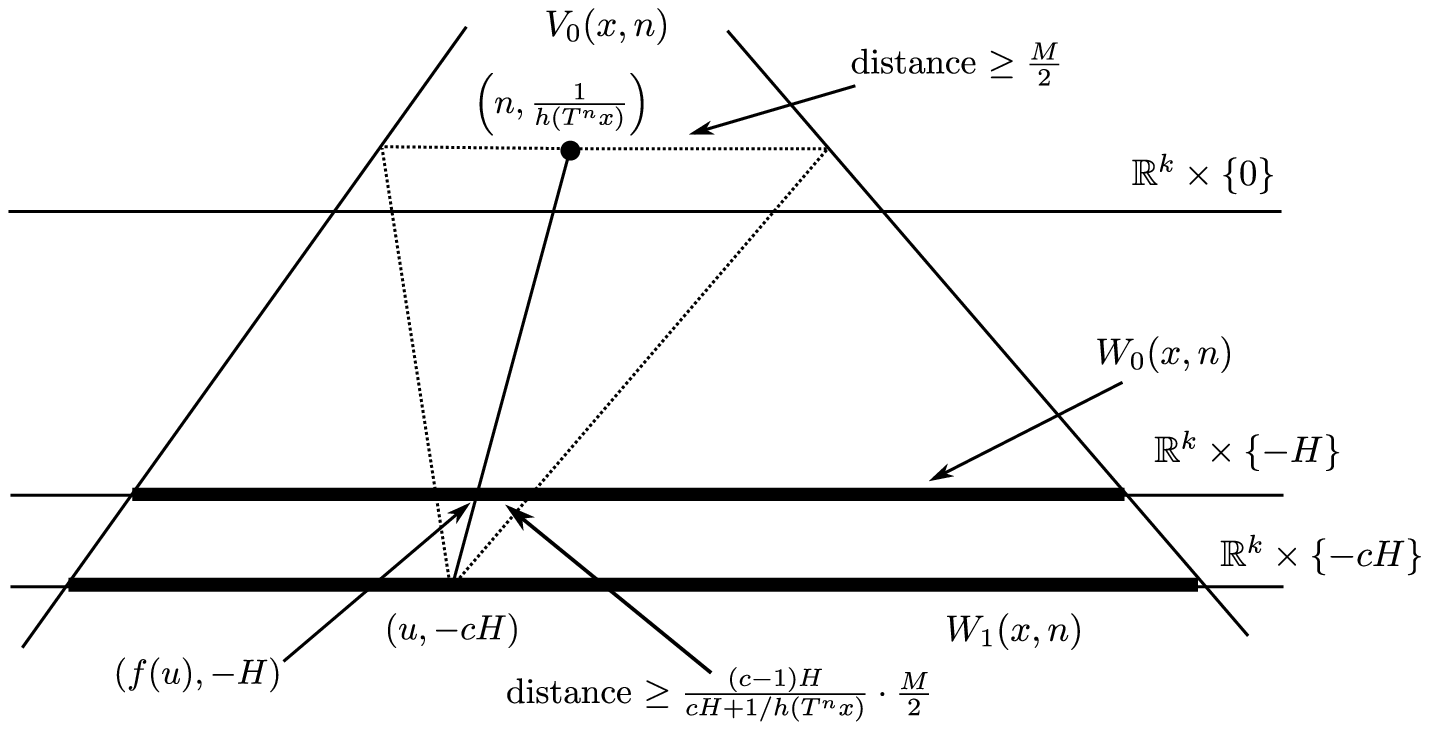}
    \caption{Configurations of $V_0(x,n)$, $W_0(x,n)$ and $W_1(x,n)$.} 
    \label{fig: configuration}
\end{figure}

\begin{claim}  \label{claim: W_0 and W_1}
   \begin{enumerate}
       \item   For any $u\in W_1(x,n)$ the point $f(u)$ belongs to $\mathrm{Int}_r W_0(x,n)$.
       \item   $\left|\mathrm{Int}_r W_0(x,n)\right|  \geq  (1-\varepsilon) \left|W_1(x,n)\right|$.
   \end{enumerate}    
\end{claim}

\begin{proof}
(1)  From $B_{M/2}\left(n, 1/h(T^n x)\right) \subset V_0(x,n)$, 
the set $W_0(x,n) \times \{-H\}$ contains the intersection point of the hyperplane 
$\mathbb{R}^k\times \{-H\}$ and the line segment between $(u, -cH)$ and $\left(v, 1/h(T^n x)\right)$ for $v\in B_{M/2}(n)$.
Then $W_0(x,n)$ contains the ball of radius 
\[   \frac{cH-H}{cH+\frac{1}{h(T^n x)}} \cdot   \frac{M}{2}  \]
around the point $f(u)$. See Figure \ref{fig: configuration} again.
From $h(T^n x) >1/2$ (Claim \ref{claim: preliminaries of W_0(x,n)} (3)), 
\[ \frac{cH-H}{cH+\frac{1}{h(T^n x)}} \cdot   \frac{M}{2}   > \frac{(c-1)H}{cH + 2} \cdot \frac{M}{2} = 
   \frac{(c-1)M}{2\left(c+\frac{2}{H}\right)}.  \]
By $H= \left(M_1+\sqrt{k}\right)^2 > 2$, this is larger than 
\[ \frac{(c-1)M}{2(c+1)} > r \quad (\text{by the first condition of (\ref{eq: choice of M in Dynamical Voronoi diagram subsection})}). \]
This shows $f(u)\in \mathrm{Int}_r W_0(x,n)$.

(2) We can assume $W_1(x,n) \neq \emptyset$. (Otherwise the statement is trivial.)
From (1) 
\begin{equation*}
  \left|\mathrm{Int}_r W_0(x,n)\right| \geq \left|\{f(u)|\, u\in W_1(x,n)\}\right|   =
  \left(\frac{H+ \frac{1}{h(T^n x)}}{cH+\frac{1}{h(T^n x)}}\right)^k \left|W_1(x,n)\right|. 
\end{equation*}
From $h(T^n x) >1/2$ (Claim \ref{claim: preliminaries of W_0(x,n)} (3)), 
\[   \left|\mathrm{Int}_r W_0(x,n)\right|   \geq \left(\frac{H}{cH+2}\right)^k \left|W_1(x,n)\right|
      = \left(\frac{1}{c + \frac{2}{H}}\right)^k \left|W_1(x,n)\right|.  \]
Since $H=\left(M_1+\sqrt{k}\right)^2 > M$
\[     \left|\mathrm{Int}_r W_0(x,n)\right|   \geq  \left(\frac{1}{c+ \frac{2}{M}}\right)^k \left| W_1(x,n)\right|. \]
By the second condition of (\ref{eq: choice of M in Dynamical Voronoi diagram subsection}) this proves (2).
\end{proof}

Let $x\in X$ and $R>2M_1+2\sqrt{k}$.
\begin{equation*}
  \begin{split}
   \left| B_{R}\cap \bigcup_{n\in \mathbb{Z}^k} \partial_r W_0(x,n)\right|  &\leq 
   |B_R| - \left|\bigcup_{|n|\leq R-M_1-\sqrt{k}} \mathrm{Int}_r W_0(x,n)\right|  \quad 
   (\text{by $W_0(x,n)\subset B_{M_1+\sqrt{k}}(n)$}) \\
   &\leq |B_R| - (1-\varepsilon) \left|\bigcup_{|n| \leq R-M_1-\sqrt{k}} W_1(x,n)\right|
   \quad (\text{by Claim \ref{claim: W_0 and W_1} (2)}).
  \end{split}
\end{equation*}  
Since $W_1(x,n) \subset B_{M_1+\sqrt{k}}(n)$ and $\mathbb{R}^k = \bigcup_{n\in \mathbb{Z}^k} W_1(x,n)$,
\[  \bigcup_{|n|\leq R-M_1-\sqrt{k}} W_1(x,n)  \supset B_{R-2M_1-2\sqrt{k}}. \]
Hence 
\[      \left| B_{R}\cap \bigcup_{n\in \mathbb{Z}^k} \partial_r W_0(x,n)\right|  
    \leq |B_R| - (1-\varepsilon) \left| B_{R-2M_1-2\sqrt{k}}\right|. \]
Thus 
\[     \limsup_{R\to \infty} 
           \left(\sup_{x\in X} \frac{\left| B_{R}\cap \bigcup_{n\in \mathbb{Z}^k} \partial_r W_0(x,n)\right|}{\left|B_{R}\right|}\right)  
        \leq \varepsilon.  \]   
This proves the statement.
\end{proof}

\subsection{Tiling and weight functions from tiling-like band-limited maps} 
\label{subsection: weight functions from tiling-like band-limited functions}

Let $A,L_0$ and $L$ be positive numbers. 
These three parameters control the construction of this subsection.
We assume:
\begin{condition}[$L$ is much larger than $AL_0$] \label{condition: L is much larger than L_0}
      $L$ is an \textit{integer} satisfying 
    \[  L > 1000^k (A+1) \left(L_0+1+\sqrt{k}\right).  \]
\end{condition}
Let $E=E(L)$ be the positive constant introduced in Notation \ref{notation: introduction of E}.
By Lemma \ref{lemma: basic properties of W_0} we can choose $M=M(A,L_0,L)>0$ so that the sets $W_0(x,n) = W_0^M(x,n)$ 
in Subsection \ref{subsection: dynamical Voronoi diagram} satisfy 
\begin{equation}  \label{eq: fixing M}
     \limsup_{R\to \infty} 
                 \left(\sup_{x\in X} \frac{\left| B_{3R}\cap \bigcup_{n\in \mathbb{Z}^k} \partial_{E+2(L+1)\sqrt{k} +L_0+1} 
                 W_0(x,n)\right|}{\left|B_{R/2}\right|}\right) 
                  < \frac{1}{6A+2}.
\end{equation}

Let $x\in X$. 
We define $\Phi_{A,L_0,L}(x) = \Phi_L\left(\{W_0(x,n)\}_{n\in \mathbb{Z}^k}\right):\mathbb{C}^k \to \mathbb{C}^k$ as 
the tiling-like band-limited map 
with respect to the tiling $\{W_0(x,n)\}_{n\in \mathbb{Z}^k}$, namely for 
$z\in \mathbb{C}^k$
\[  \Phi_{A,L_0,L}(x)(z) = \sum_{n\in \mathbb{Z}^k}  \Theta_L(z-n)
    \int_{W_0(x,n)} \chi_1(z-t)\, dt_1\dots dt_k. \]
    We often abbreviate $\Phi_{A,L_0, L}(x)$ to $\Phi(x)$.
    
\begin{lemma}[Basic properties of $\Phi(x)$] \label{lemma: basic properties of Phi(x)}
  $\Phi$ satisfies 
    \begin{description}
      \item[(1) Boundedness]  $\norm{\Phi(x)}_{L^\infty(\mathbb{R}^k)} \leq  \sqrt{k} K_1$.
      \item[(2) Frequencies] If $L>4/\delta$ then the Fourier transform of the $i$-th entry of $\Phi(x)|_{\mathbb{R}^k}$ 
                                   $(1\leq i\leq k)$ is supported in 
                                \[      \left(-\frac{\delta}{4},\frac{\delta}{4}\right)^{i-1} \times 
                                   \left(\frac{a_i}{2}, \frac{a_i}{2}+\frac{\delta}{2}\right)\times \left(-\frac{\delta}{4},\frac{\delta}{4}\right)^{k-i},  \]
                               which are disjoint with each other.
      \item[(3) Equivariance] $\Phi(T^n x)(z) = \Phi(x)(z+n)$ 
                                     for all $x\in X$, $n\in \mathbb{Z}^k$ and $z\in \mathbb{C}^k$.
      \item[(4) Continuity] If a sequence $x_n$ in $X$ converges to $x$ then $\Phi(x_n)$ converges to $\Phi(x)$
                                 uniformly over every compact subset of $\mathbb{C}^k$.                               
    \end{description}        
\end{lemma}       

\begin{proof}
(1) and (2) immediately follow from Lemma \ref{lemma: main properties of Phi_L} (1).
(3) and (4) follow from the equivariance and continuity of the tiles $W_0(x,n)$.
\end{proof}                                                                             
    
Let $0<r_1<1/4$ be the positive number introduced in Notation \ref{notation: choice of r_1}.    
Let $\alpha_1:\mathbb{C}^k \to [0,1]$ and $\alpha_2:\mathbb{R}\to [0,1]$ be continuous functions satisfying 
\begin{equation*}
   \begin{split}
       &\alpha_1(z) = 1 \quad \left(z\in D^k_{r_1}\right), \quad \alpha_1(z)=0 \quad \left(z\notin D^k_{2 r_1}\right), \\
       &\alpha_2(t) = 0 \quad \left(t\leq \frac{1}{L}\right), \quad \alpha_2(t) = 1 \quad \left(t\geq \frac{2}{L}\right).
   \end{split}  
\end{equation*}
For each $x\in X$ and $n\in \mathbb{Z}^k$ we define a non-negative number $\nu_{A,L_0,L}(x,n) = \nu(x,n)$ by 
\[  \nu(x,n) = \min\left\{1, \, \sum_{\Phi(x)(z)=0}  \alpha_1(z-n) \, \alpha_2\left(\nu\left(d\Phi(x)_z\right)\right)\right\}.  \]
Here the sum is taken over all $z\in \mathbb{C}^k$ satisfying $\Phi(x)(z)=0$, $z\in D_{2r_1}^k(n)$ and $\nu\left(d\Phi(x)_z\right)\geq 1/L$.
The set of such $z$ is a finite set because non-degenerate zero points are isolated. 
So this is a finite sum.
(Notice that the polydisk $D^k_{2r_1}(n)$ may contain infinitely many zeros of $\Phi(x)$.
The above definition of $\nu(x,n)$ addresses this problem.
The number $1/L$ is chosen so that the zero points found in Lemma \ref{lemma: main properties of Phi_L} (3)
contribute to the above sum.)
For each $n\in \mathbb{Z}^k$ the number $\nu(x,n)$ depends continuously on $x\in X$.

We again consider a Voronoi tiling.
Let $x\in X$. Consider the following discrete set in $\mathbb{R}^{k+1}$:
\[  \left\{\left(n, \frac{1}{\nu(x,n)}\right) \middle|\, n\in \mathbb{Z}^k, \> \nu(x,n) >0 \right\}. \]
This set is not empty. (See the proof of Lemma \ref{lemma: basic properties of W(x,n)} (1) below.)
So this determines a Voronoi diagram.
Namely for $n\in \mathbb{Z}^k$ we define $V(x,n)$ as the set of $u\in \mathbb{R}^{k+1}$ satisfying 
\[ \forall m\in \mathbb{Z}^k \text{ with $\nu(x,m)>0$}: \quad \left|u-\left(n,\frac{1}{\nu(x,n)}\right)\right| \leq 
    \left|u-\left(m, \frac{1}{\nu(x,m)}\right) \right|. \]
If $\nu(x,n)=0$ then $V(x,n)=\emptyset$.    
These form a tiling of $\mathbb{R}^{k+1}$: 
\[  \mathbb{R}^{k+1} = \bigcup_{n\in \mathbb{Z}^k} V(x,n). \]
We set $W(x,n) = \pi_{\mathbb{R}^k}\left(V(x,n)\cap \left(\mathbb{R}^k\times \{0\}\right)\right)$, where 
$\pi_{\mathbb{R}^k}:\mathbb{R}^{k+1}\to \mathbb{R}^k$ is the projection to the first $k$ coordinates as before.
The sets $W(x,n)$ form a tiling of $\mathbb{R}^k$:
\[  \mathbb{R}^k = \bigcup_{n\in \mathbb{Z}^k} W(x,n). \]
This is also $\mathbb{Z}^k$-equivariant, namely $W(T^n x, m) = -n + W(x,n+m)$.
Each $W(x,n)$ is a bounded convex set of $\mathbb{R}^k$ (see Lemma \ref{lemma: basic properties of W(x,n)} (1) below) and 
depends continuously on $x\in X$ in the Hausdorff topology as in the case of $W_0(x,n)$.
We sometimes denote $W(x,n)$ by $W^{A,L_0,L}(x,n)$ for clarifying the dependence on the parameters $A,L_0,L$.
For $x\in X$ we define $\partial(x) = \partial^{A,L_0,L}(x)$ as $\bigcup_{n\in \mathbb{Z}^k} \partial W(x,n)$.

\begin{lemma}[Basic properties of $W(x,n)$]  \label{lemma: basic properties of W(x,n)}
   \begin{enumerate} 
      \item   The diameter of $W(x,n)$ is bounded uniformly in $x\in X$ and $n\in \mathbb{Z}^k$.
               Moreover 
               \[  \sup_{x\in X} \sup_{n\in \mathbb{Z}^k} \sup_{u\in W(x,n)} |n-u|  < \infty. \]
      \item Let $\alpha_3:\mathbb{R}\to [0,3]$ be a continuous function satisfying 
              \[  \alpha_3(t)=3 \quad (t\leq L_0), \quad   \alpha_3(t) = 0 \quad (t\geq L_0+1). \]
              There exists $R_0 = R_0(A,L_0,L)>0$ such that for all $x\in X$ and 
              all $u\in \mathbb{Z}^k$
              \[  A \sum_{|n-u|\leq 2R_0} \alpha_3  \left(d\left(n, \partial(x)\right)\right) 
                   < \sum_{|n-u|\leq R_0}  \left|\mathrm{Int}_{L_0} W(x,n)\right|,   \]
             where $d\left(n, \partial(x)\right) = \inf_{t\in \partial(x)} |n-t|$ and 
             the sums are taken over all $n\in \mathbb{Z}^k$ satisfying $|n-u|\leq 2R_0$ and 
             $|n-u|\leq R_0$ in the left-hand side and right-hand side respectively.           
   \end{enumerate}
\end{lemma}

\begin{proof}
Let $x\in X$ and $n\in \mathbb{Z}^k$.
It follows from Lemma \ref{lemma: main properties of Phi_L} (2) and (3) that 
\begin{equation}  \label{eq: main properties of nu(x,n)}
   \begin{split}
   &  \forall \text{ integer point } m\in \mathrm{Int}_{E+2r_1\sqrt{k}} W_0(x,n) \setminus \left(n+L\mathbb{Z}^k\right): \>
    \nu(x,m)=0, \\
    &  \forall \, m\in \left(n+L\mathbb{Z}^k\right) \cap \mathrm{Int}_{E+\sqrt{k}}W_0(x,n):  \> \nu(x,m)=1.
   \end{split}
\end{equation}

(1) By (\ref{eq: fixing M}) if we choose $R$ sufficiently large then for all $x\in X$
\[  B_{R} \cap \bigcup_{n\in \mathbb{Z}^k} \partial_{E+(L+1)\sqrt{k}} W_0(x,n) \neq B_R. \]
Then there exists $n\in \mathbb{Z}^k$ satisfying 
$B_R\cap \mathrm{Int}_{E+(L+1)\sqrt{k}} W_0(x,n) \neq \emptyset$.
Then we can find a point $m\in  \left(n+L\mathbb{Z}^k\right)  \cap  B_{R+L\sqrt{k}}  \cap \mathrm{Int}_{E+\sqrt{k}} W_0(x,n)$.
It follows from the second line of (\ref{eq: main properties of nu(x,n)}) that $\nu(x,m)=1$.
By the $\mathbb{Z}^k$-equivariance we conclude that for any $p\in \mathbb{Z}^k$ there exists an integer point $q\in B_{R+L\sqrt{k}}(p)$
satisfying $\nu(x,q)=1$.

Take $u\in W(x,n)$. Let $p$ be the integer point closest to $u$.
Then we can find an integer point $q\in B_{R+L\sqrt{k}}(p)$ satisfying $\nu(x,q)=1$.
It follows form the definition of the Voronoi tiling that 
\[ \left|(u, 0) - \left(n, \frac{1}{\nu(x,n)}\right) \right|  \leq 
   \left|(u, 0) - (q, 1)\right|. \]
From $\nu(x,n)\leq 1$, 
\[ |u-n|\leq |u-q| \leq |u-p| + |p-q| \leq \sqrt{k}+R+L\sqrt{k}. \]
This proves 
\[ \sup_{x\in X} \sup_{n\in \mathbb{Z}^k} \sup_{u\in W(x,n)} |n-u|   \leq R + (L+1)\sqrt{k}. \]

(2) 
The above (\ref{eq: main properties of nu(x,n)}) implies
\begin{equation*}
     \forall m\in \left(n+L\mathbb{Z}^k\right) \cap \mathrm{Int}_{E+(L+1)\sqrt{k}} W_0(x,n): \> 
     W(x,m) = m + \left[-\frac{L}{2}, \frac{L}{2}\right]^k,
\end{equation*}
because for such $m$ 
\begin{equation*}
   \begin{split}
   & \forall \text{ integer point } l\in m+L \, (-1,1)^k \text{ with $l\neq m$}: \quad \nu(x,l)=0,  \\
   &  \forall \,  l \in m + L \, \{-1,0,1\}^k: \quad \nu(x,l)=1. 
   \end{split}
\end{equation*}

We introduce a dichotomy: For $x\in X$ we define a ``good set'' $\mathbb{G}_x$ as the set of $m\in \mathbb{Z}^k$
satisfying $W(x,m) = m + [-L/2,L/2]^k$. We define a ``bad set'' $\mathbb{B}_x$ as the complement of $\mathbb{G}_x$ in 
$\mathbb{Z}^k$.
Then it follows that 
\begin{equation}
    \begin{split}
    \mathrm{Int}_{E+(2L+1)\sqrt{k}} W_0(x,n) \subset \bigcup_{m\in \mathbb{G}_x} W(x,m), \\
    \bigcup_{m\in \mathbb{B}_x} W(x,m) \subset \bigcup_{n\in \mathbb{Z}^k} \partial_{E+ (2L+1)\sqrt{k}} W_0(x,n).    
    \end{split}
\end{equation}
 For simplicity of the notation, we set $E' = E+ (2L+1)\sqrt{k}$.
 We define $D$ as the supremum of $|n-u|$ over all $x\in X$, $n\in \mathbb{Z}^k$ and $u\in W(x,n)$.
 This is finite by (1).
 
Let $r$ and $R$ be positive numbers and $u\in \mathbb{Z}^k$.
We estimate $\sum_{n\in B_R(u)} \left|\partial_r W(x,n)\right|$.
Noting 
\[  B_{R-D}(u) \subset \bigcup_{n\in B_R(u)} W(x,n) \subset B_{R+D}(u), \]
we can estimate 
\begin{equation}  \label{eq: estimate from good/bad dichotomy}
    \begin{split}
      \sum_{n\in B_R(u)} \left|\partial_r W(x,n)\right| & =  \sum_{n\in B_R(u)\cap \mathbb{G}_x} \left|\partial_r W(x,n)\right|
      +  \sum_{n\in B_R(u)\cap \mathbb{B}_x} \left|\partial_r W(x,n)\right| \\
      & \leq 2^{k+1}r(L+2r)^{k-1} \frac{\left|B_{R+D}\right|}{L^k}  + 
         \left|B_{R+D+r}(u) \cap \bigcup_{n\in \mathbb{Z}^k} \partial_{E'+r} W_0(x,n)\right| \\
      & =     \frac{2^{k+1}r(L+2r)^{k-1}}{L^k} \left|B_{R+D}\right|  + 
         \left|B_{R+D+r} \cap \bigcup_{n\in \mathbb{Z}^k} \partial_{E'+r} W_0(T^u x,n)\right|.
    \end{split}
\end{equation}

We would like to prove that for sufficiently large $R$ (uniformly in $x\in X$ and $u\in \mathbb{Z}^k$)
\begin{equation}  \label{eq: tax is larger than cost}
   A \sum_{n\in B_{2R}(u)} \alpha_3  \left(d\left(n, \partial(x)\right)\right) 
                   < \sum_{n\in B_R(u)}  \left|\mathrm{Int}_{L_0} W(x,n)\right|. 
\end{equation}
It follows from (\ref{eq: estimate from good/bad dichotomy}) that 
\begin{equation*}
   \begin{split}
    \sum_{n\in B_R(u)} \left|\mathrm{Int}_{L_0} W(x,n)\right|   \geq &
    \sum_{n\in B_R(u)}\left( \left|W(x,n)\right| - \left|\partial_{L_0}W(x,n)\right|\right)  \\
     \geq &   \left|B_{R-D}\right| - \frac{2^{k+1}L_0(L+2L_0)^{k-1}}{L^k} \left|B_{R+D}\right|    \\
          &  - \left|B_{R+D+L_0} \cap \bigcup_{n\in \mathbb{Z}^k} \partial_{E'+L_0} W_0(T^u x,n)\right|.  
   \end{split}
\end{equation*}   
If $R$ is  sufficiently larger than $D+L_0$ then (noting $L>2L_0$ by Condition \ref{condition: L is much larger than L_0})
\begin{equation}  \label{eq: lower bound of the RHS, tax is larger than cost}
   \begin{split}
    \sum_{n\in B_R(u)} \left|\mathrm{Int}_{L_0} W(x,n)\right|  & \geq 
     \left|B_{R/2}\right| - \frac{4^k L_0}{L} \left|B_{2R}\right|  -
         \left|B_{2R} \cap \bigcup_{n\in \mathbb{Z}^k} \partial_{E'+L_0} W_0(T^u x,n)\right|  \\
          & = \left(1-\frac{16^k L_0}{L}\right) \left|B_{R/2}\right| 
              -  \left|B_{2R} \cap \bigcup_{n\in \mathbb{Z}^k} \partial_{E'+L_0} W_0(T^u x,n)\right|.
   \end{split}           
\end{equation}         
The left-hand side of (\ref{eq: tax is larger than cost}) is bounded by 
\begin{equation*}
   \begin{split}
      & 3A \#\left(\mathbb{Z}^k \cap B_{2R}(u)\cap \bigcup_{n\in \mathbb{Z}^k} \partial_{L_0+1}W(x,n)\right)  \\
      & \leq 3A \left|B_{2R+\sqrt{k}}(u)\cap \bigcup_{n\in \mathbb{Z}^k}\partial_{L_0+1+\sqrt{k}} W(x,n)\right| \\
      & \leq 3A \sum_{n\in B_{2R+2\sqrt{k}+D+L_0+1} (u)} \left|\partial_{L_0+1+\sqrt{k}} W(x,n) \right|.
    \end{split} 
\end{equation*}    
As in the case of (\ref{eq: lower bound of the RHS, tax is larger than cost}) if $R$ is sufficiently large then this is bounded by 
(noting $L>2L_0+2+2\sqrt{k}$ by Condition \ref{condition: L is much larger than L_0})
\begin{equation}  \label{eq: upper bound on the LHS, tax is larger than cost}
   \begin{split}
    &  3A\left(\frac{4^k\left(L_0+1+\sqrt{k}\right)}{L} \left|B_{3R}\right| 
    +  \left|B_{3R} \cap \bigcup_{n\in \mathbb{Z}^k} \partial_{E'+L_0+1+\sqrt{k}} W_0(T^u x,n)\right|\right) \\
   & = 3A\left(\frac{24^k \left(L_0+1+\sqrt{k}\right)}{L} \left|B_{R/2}\right| + 
         \left|B_{3R} \cap \bigcup_{n\in \mathbb{Z}^k} \partial_{E'+L_0+1+\sqrt{k}} W_0(T^u x,n)\right|\right).
    \end{split}
\end{equation}    
By combining (\ref{eq: lower bound of the RHS, tax is larger than cost}), (\ref{eq: upper bound on the LHS, tax is larger than cost}) with 
Condition \ref{condition: L is much larger than L_0} and (\ref{eq: fixing M}), we can conclude that 
(\ref{eq: tax is larger than cost}) holds for sufficiently large $R$.
\end{proof}

As we explained in Subsection \ref{subsection: strategy of the proof} we need a social welfare system among 
the tiles $W(x,n)$. The next proposition provides it.

\begin{proposition}[Construction of weight functions] \label{prop: construction of weight functions}
We can construct a continuous map 
\begin{equation} \label{eq: weight functions map}
   X\to \left([0,1]^{\mathbb{Z}^k}\right)^{\mathbb{Z}^k}, \>
   x\mapsto \left(w(x,n)\right)_{n\in \mathbb{Z}^k}, \>
   \text{where $w(x,n) = \left(w_m(x,n)\right)_{m \in \mathbb{Z}^k} \in [0,1]^{\mathbb{Z}^k}$} 
\end{equation}   
   satisfying the following.
   
   \begin{enumerate}
     \item The map is equivariant in the sense that for all $m,n\in \mathbb{Z}^k$ and $x\in X$
             \[ w(T^m x, n) = w(x, n+m). \]
     \item If $x, y\in X$ satisfy $\Phi(x)=\Phi(y)$ then $w(x,n)=w(y,n)$ for all $n\in \mathbb{Z}^k$.
     \item For all $x\in X$ and $n\in \mathbb{Z}^k$
             \[  \#\{m \in \mathbb{Z}^k|\, w_m(x,n)>0\} < 1 + \frac{1}{A} \left|\mathrm{Int}_{L_0} W(x,n)\right|. \]
             In particular if $\mathrm{Int}_{L_0}W(x,n) =\emptyset$ then $w_m(x,n)=0$ for all $m\in \mathbb{Z}^k$.
     \item Let $x\in X$ and $p\in \mathbb{Z}^k$ with $d\left(p, \partial(x)\right) \leq L_0$. 
             There exists 
             $n\in \mathbb{Z}^k\cap B_{R_0}(p)$ satisfying $w_{p-n}(x,n) = 1$.        
             Here $R_0 = R_0(A,L_0,L)$ is the positive number introduced in Lemma \ref{lemma: basic properties of W(x,n)} (2).
  \end{enumerate}
We call $w(x,n)$ \textbf{weight functions} and sometimes write 
$w^{A,L_0,L}(x,n) = w(x,n)$ for clarifying the dependence on parameters.  
\end{proposition}

Probably the intuitive meaning of the statement is not clear at all.
So we explain it before the proof. (The idea of the social welfare system was first introduced in \cite{Gutman--Tsukamoto minimal}.
The explanation below is more or less a reproduction of the argument around \cite[Lemma 6.4]{Gutman--Tsukamoto minimal}.)
In the proof of Proposition \ref{prop: main proposition} we construct a perturbation $g_1(x)$ 
of a given band-limited function $f(x)$ for each $x\in X$.
It is difficult to construct a perturbation over the whole space $\mathbb{R}^k$ at once.
So we will perturb $f(x)$ over each tile $W(x,n)$ separately.
The difficulty arises near the boundary $\partial W(x,n)$.
The parameter $L_0$ will be chosen so that we can construct a good perturbation in $\mathrm{Int}_{L_0} W(x,n)$.
But we cannot control the perturbation over $\partial_{L_0} W(x,n)$.
(Note that if the tile $W(x,n)$ is tiny, it may be contained in $\partial_{L_0} W(x,n)$.
So the problem is approximately equivalent to ``how to help small tiles''.)

The function $g_1(x)$ should encode the information of the orbit of $x\in X$.
The above argument means that it becomes difficult to encode the information of $T^p x$ for
$p\in \mathbb{Z}^k$ with $d\left(p,\partial (x)\right) \leq L_0$.
The weight functions $w_m(x,n)$ help with this situation.
Roughly (and very inaccurately) speaking, we will encode the $\left(100\times w_m(x,n)\right) \%$ of the information of $T^{n+m}x$
to $g_1(x)|_{W(x,n)}$.
So $w_m(x,n)$ is the ``amount of help'' that the tile $W(x,n)$ gives to the point $n+m$.
In particular, if $w_m(x,n)=1$ then the information of $T^{n+m} x$ is perfectly encoded into $g_1(x)|_{W(x,n)}$.

Now we can explain the intuitive meaning of the above (1)--(4).
Condition (1) is an obvious requirement. 
Condition (2)\footnote{The above explanation does not feature Condition (2).
But indeed this is a very important condition. It did not appear in the paper \cite{Gutman--Tsukamoto minimal}.
A substantial amount of the argument in this section has been developed for establishing this condition.}
 means that the weight function $w(x,n)$ is determined by 
the tiling-like band-limited map $\Phi(x)$.
Condition (3) means that the amount of ``additional encoding'' which $W(x,n)$ has to bear is controlled by the volume of 
$\mathrm{Int}_{L_0} W(x,n)$.
Condition (4) means that every point near the boundary $\partial (x)$ is ``successfully rescued'' by some large tile.
Namely we can solve all the difficulties coming from the boundary effect.

\begin{proof}[Proof of Proposition \ref{prop: construction of weight functions}]
We fix a bijection between $\mathbb{Z}^k$ and $\mathbb{N} = \{1,2,3,\dots\}$
\[  \mathbb{N}\ni l \mapsto m_l \in \mathbb{Z}^k \]
such that if $l_1<l_2$ then $|m_{l_1}|\leq |m_{l_2}|$.
(Then in particular $m_1=0$.)

Fix $x\in X$.
For $n\in \mathbb{Z}^k$ we set 
\[ u_0(n) = \frac{\left|\mathrm{Int}_{L_0} W(x,n)\right|}{A}, \quad v_0(n) = \alpha_3\left(d\left(n,\partial(x)\right)\right). \]
Here $\alpha_3:\mathbb{R}\to [0,3]$ is the continuous function introduced in Lemma \ref{lemma: basic properties of W(x,n)} (2).
We define $\omega_l(n), u_l(n), v_l(n)$ inductively with respect to $l\geq 1$ by 
\begin{equation*}
   \begin{split}
    &  \omega_l(n) = \min\left(u_{l-1}(n), v_{l-1}(n+m_l)\right), \\
    &  u_l(n) = u_{l-1}(n)- \omega_l(n), \\
    &  v_l(n) = v_{l-1}(n)-\omega_l(n-m_l). 
   \end{split}
\end{equation*}   
This process can be explained in ``social welfare'' terms:

\begin{itemize}
    \item  The tiles $W(x,n)$ should pay tax. Tax is used for helping integer points $p\in \mathbb{Z}^k$.
             The tile $W(x,n)$ can pay at most $u_0(n)$ and the point $p$ needs $v_0(p)$.
    \item  At the $l$-th step of the induction, the tile $W(x,n)$ gives the money $\omega_l(n)$ to the point $n+m_l$.
             After the step, $W(x,n)$ still can pay $u_l(n)$ and points $p$ still need $v_l(p)$.
    \item At each step, $W(x,n)$ pays as much as possible. Namely, after the $l$-th step, at least one of $u_l(n)$ and $v_l(n+m_l)$
            is zero. This is a key property of the process.
            Figure \ref{fig: tax} schematically explains the process.
\end{itemize}

\begin{figure}
    \centering
    \includegraphics[,bb= 0 0 350 110]{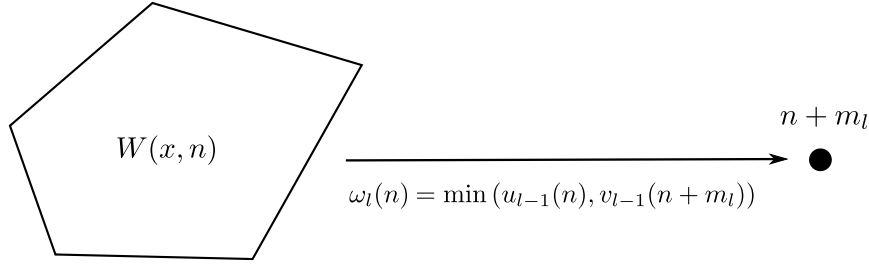}
    \caption{At the $l$-th step of the induction, the tile $W(x,n)$ gives the amount of money $\omega_l(n)$ to the point $n+m_l$.
    Before this step, $W(x,n)$ has the amount of money $u_{l-1}(n)$ and the point $n+m_l$ needs $v_{l-1}(n+m_l)$.
    After this step, either $W(x,n)$ loses all its money (i.e. $u_l(n)=0$) or $n+m_l$ becomes satisfied (i.e. $v_l(n+m_l)=0$).
    Namely, $\min\left(u_l(n), v_l(n+m_l)\right) = 0$.}
    \label{fig: tax}
\end{figure}

Set $l_0 = \#\left(\mathbb{Z}^k\cap B_{R_0}\right)$. 
Then $\mathbb{Z}^k\cap B_{R_0} = \{m_1,\dots, m_{l_0}\}$.
It follows from Lemma \ref{lemma: basic properties of W(x,n)} (2) that this process terminates: 
For all $l\geq l_0$ and $n\in \mathbb{Z}^k$
\[  v_l(n)=0, \quad \omega_{l+1}(n)=0,   \]
because if $v_{l_0}(n_0)>0$ for some $n_0$ then $u_{l_0}(n)=0$ for all $n\in B_{R_0}(n_0)$ and hence 
\[ \sum_{n\in B_{R_0}(n_0)} u_0(n)  = \sum_{n\in B_{R_0}(n_0)} \sum_{l=1}^{l_0} \omega_l(n)   
   <  \sum_{n\in B_{2R_0}(n_0)} v_0(n),   \]
which contradicts Lemma \ref{lemma: basic properties of W(x,n)} (2).

It follows from the construction that for $n, p \in \mathbb{Z}^k$
\begin{equation}  \label{eq: sum of tax is smaller than u_0}
    \sum_{l=1}^{l_0} \omega_l(n) \leq u_0(n) =  \frac{\left|\mathrm{Int}_{L_0} W(x,n)\right|}{A}, 
\end{equation}    
\begin{equation}  \label{eq: every point is satisfied}
   \sum_{l=1}^{l_0} \omega_l(p-m_l) = v_0(p) = \alpha_3\left(d\left(n,\partial(x)\right)\right).
\end{equation}

We define continuous functions $\alpha_4: \mathbb{R}\to [1, 2l_0]$ and $\alpha_5:\mathbb{R}\to [0,1]$ such that 
\begin{equation*}
   \begin{split}
    \alpha_4(0)=2l_0, \quad         \alpha_4(t) = 1 \quad (t\geq 1), \\
    \alpha_5(t)=0 \quad (t\leq 1),    \quad \alpha_5(t) = 1 \quad (t\geq 2).   
   \end{split}
\end{equation*}   
We define $F:\mathbb{R}^{l_0}\to \mathbb{R}^{l_0}$ by 
$F(x_1,\dots,x_{l_0}) = (y_1,\dots, y_{l_0})$ where 
\begin{equation*}
   \begin{split}
    & y_{l_0} = 2l_0 x_{l_0}, \quad  y_{l_0-1} = \alpha_4(y_{l_0}) x_{l_0-1}, \quad
    y_{l_0-2} = \alpha_4\left(\max(y_{l_0}, y_{l_0-1})\right) x_{l_0-2}, \quad \dots \\
    & y_1 = \alpha_4\left(\max(y_{l_0}, \dots, y_2)\right) x_1. 
   \end{split}
\end{equation*}   
This satisfies 
\begin{equation} \label{eq: property of F}
   \#\{1\leq l\leq l_0|\, y_l>1 \} \leq 1 + \#\{1\leq l\leq l_0|\, x_l >1\}.
\end{equation}
We define $G:\mathbb{R}^{l_0}\to [0,1]^{l_0}$ by $G(y_1,\dots,y_{l_0}) = \left(\alpha_5(y_1),\dots, \alpha_5(y_{l_0})\right)$.
Then it follows from the above (\ref{eq: property of F}) that if we set $GF(x_1,\dots, x_{l_0}) = (z_1,\dots,z_{l_0})$ then 
\begin{equation} \label{eq: property of GF}
   \# \{1\leq l\leq l_0|\, z_l>0\} \leq 1 + \#\{1\leq l\leq l_0|\, x_l>1\}.
\end{equation}

We define $w_m(x,n)$ for $|m|\leq R_0$ and $n\in \mathbb{Z}^k$ by 
\[ \left(w_{m_1}(x,n), w_{m_2}(x,n), \dots, w_{m_{l_0}}(x,n)\right) = GF\left(\omega_1(n), \dots, \omega_{l_0}(n)\right). \]
We set $w_m(x,n)=0$ for $|m|>R_0$.

Now we have defined the map (\ref{eq: weight functions map}) in the statement.
We need to check its properties.
It is continuous and equivariant (i.e. Property (1)) because the tiles $W(x,n)$ depend continuously on $x\in X$
and $W(T^m x, n) = -m + W(x,n+m)$.
Property (2) is obvious because the tiles $W(x,n)$ are constructed from the function $\Phi(x)$.

Fix $x\in X$ again.
From (\ref{eq: sum of tax is smaller than u_0})
\[ \#\{1\leq l\leq l_0|\, \omega_l(n)>1\} <   \frac{\left|\mathrm{Int}_{L_0} W(x,n)\right|}{A}.  \]
Then Property (3) follows from (\ref{eq: property of GF}):
\[ \#\{m|\, w_m(x,n) >0\} \leq 1 + \#\{1\leq l\leq l_0 |\, \omega_l(n)>1\}
   < 1 +  \frac{\left|\mathrm{Int}_{L_0} W(x,n)\right|}{A}.   \]

Finally we check Property (4).
Let $p\in \mathbb{Z}^k$ with $d\left(p,\partial(x)\right) \leq L_0$.
Then $v_0(p) = 3$.
From (\ref{eq: every point is satisfied}) 
\[  \sum_{l=1}^{l_0} \omega_l(p-m_l) = 3. \]
We define $l_1$ as the maximum of $1\leq l\leq l_0$ satisfying $\omega_l(p-m_l)>0$.
If $\omega_{l_1}(p-m_{l_1})\geq 2$ then $w_{m_{l_1}}(p-m_{l_1})=1$.
Otherwise 
\[  \sum_{l=1}^{l_1-1} \omega_l(p-m_l) > 1. \]
Then there exists $l_2 <l_1$ satisfying 
\[  \omega_{l_2}(p-m_{l_2}) > \frac{1}{l_1} \geq \frac{1}{l_0}. \]
The condition $l_2<l_1$ implies $u_{l_2}(p-m_{l_2})=0$ (otherwise $v_{l_2}(p)=0$ and $\omega_{l_1}(p-m_{l_1})=0$)
and hence
$\omega_l(p-m_{l_2})=0$ for all $l>l_2$.
It follows from the definition of the maps $F$ and $G$ that 
\[ w_{l_2}(p-m_{l_2}) =1. \]
This proves Property (4).
\end{proof}

\section{Proof of Main Proposition} \label{section: proof of Main Proposition}

We prove Proposition \ref{prop: main proposition} in this section.
We recommend readers to review the notations introduced in Subsection \ref{subsection: fixing some notations}.
We repeat the statement of the proposition (assuming the notations in Subsection \ref{subsection: fixing some notations}):

\begin{proposition}[$=$ Proposition \ref{prop: main proposition}]  
Let $d$ be a distance on $X$ and $f:X\to \mathcal{B}_1(a_1,\dots,a_k)$ a $\mathbb{Z}^k$-equivariant continuous map.
Then there exists a $\mathbb{Z}^k$-equivariant continuous map $g:X\to \mathcal{B}_1(a_1+\delta,\dots,a_k+\delta)$ such that 
\begin{itemize}
   \item $\norm{g(x)-f(x)}_{L^\infty(\mathbb{R}^k)} < \delta$ for all $x\in X$.
   \item $g$ is a $\delta$-embedding with respect to the distance $d$.
\end{itemize}
\end{proposition}

We can assume $\norm{f(x)}_{L^\infty(\mathbb{R}^k)} \leq 1-\delta$ for all $x\in X$ by replacing $f$ with $(1-\delta)f$ if necessary.
We fix $0<\delta'<\delta$ satisfying
\begin{equation} \label{eq: delta'}
     4K_0 \delta' < \delta
\end{equation}
where $K_0$ is the positive number introduced in Notation \ref{notation: choice of r_0} (1).
We choose $0<\varepsilon < \delta$ such that for any $x,y\in X$
\begin{equation} \label{eq: varepsilon}
  d(x,y) < \varepsilon \Longrightarrow \norm{f(x)-f(y)}_{L^\infty \left([0,1]^k\right)} < \delta'. 
\end{equation}

From $\mdim(X) < \rho_1\dots\rho_k/2$ we can find $c_0>1$ and a natural number $N$ such that 
\begin{itemize}
    \item $\rho_i N$ are integers for all $1\leq i\leq k$. (Recall $\rho_i\in \mathbb{Q}$.)
    \item \[ c_0\cdot  \frac{\widim_\varepsilon\left(X,d_{[N]}\right) + 1}{N^k} < \frac{\rho_1\dots \rho_k}{2}. \] 
\end{itemize}
From the second condition, we can find a simplicial complex $P$ and an $\varepsilon$-embedding 
$\Pi: \left(X,d_{[N]}\right) \to P$ satisfying 
\begin{equation}  \label{eq: dimension of P}
     c_0 \frac{\dim P + 1}{N^k}  < \frac{\rho_1\dots\rho_k}{2}. 
\end{equation}     
We take a simplicial complex $Q$ and an $\varepsilon$-embedding 
$\pi:(X,d) \to Q$.

We define the \textbf{cone} $CP$ as $[0,1]\times P/\sim$, where $(0,p)\sim (0,q)$ for all $p,q\in P$.
We denote by $tp$ the equivalence class of $(t,p)\in [0,1]\times P$.
We set $* = 0p$ and call it the \textbf{vertex} of the cone.
We define $CQ$ in the same way.
It follows from (\ref{eq: dimension of P}) that
\begin{equation} \label{eq: dimension of CP}
   c_0 \frac{\dim CP}{N^k}  < \frac{\rho_1\dots \rho_k}{2}.
\end{equation}
We set 
\begin{equation}  \label{eq: c_1}
   c_1 = \frac{\rho_1\dots\rho_k}{2} - c_0\frac{\dim CP}{N^k} >0.
\end{equation}
Recall that we introduced the positive number $r_0$ in Notation \ref{notation: choice of r_0} (2).
We set 
\begin{equation} \label{eq: definition of r_2}
   r_2 = r_0 + \sqrt{\sum_{i=1}^k (1/\rho_i)^2}. 
\end{equation}

\begin{notation}[Fixing the parameters $A,L_0,L$]  \label{notation: fixing A,L_0,L}
We take positive numbers $A, L_0$ and $L$ satisfying the following conditions.
  \begin{enumerate}
    \item $c_1 A > 2\dim CQ$.
    \item $L_0$ is sufficiently larger than $r_2 + N\sqrt{k}$
            such that if a bounded closed convex set $W\subset \mathbb{R}^k$ satisfies $\mathrm{Int}_{L_0} W\neq \emptyset$ then 
            \begin{equation} \label{eq: dim CQ}
               2\dim CQ < c_1 \left|\mathrm{Int}_{r_2+N\sqrt{k}} W \right|, 
            \end{equation}
            \begin{equation} \label{eq: condition on L_0}
                 \left|W\cup\partial_{N\sqrt{k}} W\right| < c_0 \left|\mathrm{Int}_{r_2+N\sqrt{k}} W\right|.    
            \end{equation}     
            The second condition is satisfied for sufficiently large $L_0$ by Lemma \ref{lemma: comparing boundary and interior of convex set}.
    \item $L>  4/\delta$ (see Lemma \ref{lemma: basic properties of Phi(x)} (2)) and Condition \ref{condition: L is much larger than L_0} holds.    
  \end{enumerate}
For these $A,L_0$ and $L$ we construct the tiles $W(x,n) = W^{A,L_0,L}(x,n)$ and the weight functions 
$w(x,n) = w^{A,L_0,L}(x,n)$ as in 
Subsection \ref{subsection: weight functions from tiling-like band-limited functions}.
\end{notation}

\begin{lemma}   \label{lemma: how many nonzero w_m}
Let $x\in X$ and $n\in \mathbb{Z}^k$. We define 
\[ \#(x,n) = \#\{m\in \mathbb{Z}^k|\, w_m(x,n) >0\}. \]
If $\mathrm{Int}_{L_0}W(x,n) \neq \emptyset$ then 
\begin{equation}  \label{eq: how many nonzero w_m}
   \begin{split}
    \frac{\left| W(x,n)\cup \partial_{N\sqrt{k}}W(x,n)\right|}{N^k}\dim CP   &+ \#(x,n) \cdot \dim CQ \\
    &< \frac{\rho_1\dots \rho_k}{2} \left|\mathrm{Int}_{r_2+N\sqrt{k}} W(x,n)\right|.  
   \end{split} 
\end{equation}    
\end{lemma}

\begin{proof}
By (\ref{eq: condition on L_0}) and Proposition \ref{prop: construction of weight functions} (3), the left-hand side of 
(\ref{eq: how many nonzero w_m}) is smaller than 
\begin{equation}  \label{eq: how many nonzero w_m, step 1}
   c_0 \frac{\left|\mathrm{Int}_{r_2+N\sqrt{k}} W(x,n)\right|}{N^k} \dim CP + 
   \left(1+\frac{\left|\mathrm{Int}_{L_0} W(x,n)\right|}{A}\right) \dim CQ. 
\end{equation}   
By Notation \ref{notation: fixing A,L_0,L} (1) and (2), the second term is smaller than 
\begin{equation*}
    \begin{split}
        \frac{1}{2}c_1 \left|\mathrm{Int}_{r_2+N\sqrt{k}} W(x,n)\right|  &+ \frac{1}{2}c_1 \left|\mathrm{Int}_{L_0} W(x,n)\right| \\
        & \leq  c_1 \left|\mathrm{Int}_{r_2+N\sqrt{k}} W(x,n)\right| \quad (\text{by $L_0>r_2+N\sqrt{k}$}). 
    \end{split}
\end{equation*}       
Then it follows from the definition of $c_1$ in (\ref{eq: c_1}) that the above (\ref{eq: how many nonzero w_m, step 1}) is smaller than 
\[  \left(c_0 \frac{\dim CP}{N^k} + c_1\right)  \left|\mathrm{Int}_{r_2+N\sqrt{k}} W(x,n)\right|
     = \frac{\rho_1\dots \rho_k}{2}  \left|\mathrm{Int}_{r_2+N\sqrt{k}} W(x,n)\right|. \]
\end{proof}

\begin{notation}[Choice of $R$]  \label{notation: choice of R}
  We take a positive number $R$ such that 
   \begin{itemize}
      \item $R\in N\mathbb{Z}$.
      \item $R>R_0(A,L_0,L)$ where $R_0(A,L_0,L)$ is the positive number introduced in Lemma \ref{lemma: basic properties of W(x,n)} (2).
      \item For all $x\in X$ and $n\in \mathbb{R}^k$, the tile $W(x,n)$ satisfies $-n+W(x,n) \subset [-R,R)^k$, where 
              $-n+W(x,n) = \{-n+t|\, t\in W(x,n)\}$. This condition is satisfied for sufficiently large $R$ by Lemma
              \ref{lemma: basic properties of W(x,n)} (1).
   \end{itemize}   
\end{notation}

  We set 
  \[  \boldsymbol{\mathcal{P}} = \left(CP\right)^{[-R,R)^k \cap N\mathbb{Z}^k}, \quad 
      \boldsymbol{\mathcal{Q}} = \left(CQ\right)^{[-R,R)^k \cap \mathbb{Z}^k}.    \]
  For $i\geq 0$ we define $\boldsymbol{\mathcal{Q}}(i) \subset \boldsymbol{\mathcal{Q}}$ as the set of 
  $(q_n)_{n\in [-R,R)^k\cap \mathbb{Z}^k}$ satisfying $q_n = *$ except for at most $i$ entries.
  In particular $\boldsymbol{\mathcal{Q}}(0) = \{(*,\dots,*)\}$ and 
  $\dim \boldsymbol{\mathcal{Q}}(i) \leq  i \dim CQ$.
   Let $W\subset [-R,R)^k$ be a convex set.
   We define $\boldsymbol{\mathcal{P}}(W) \subset \boldsymbol{\mathcal{P}}$ as the set of $(p_n)_{n\in [-R,R)^k\cap N\mathbb{Z}^k}$
   such that if $W\cap \left(n+[0,N)^k\right) = \emptyset$ then $p_n=*$.
   Its dimension is estimated by 
   \[  \dim \boldsymbol{\mathcal{P}}(W)  \leq \frac{\left|W\cup \partial_{N\sqrt{k}}W\right|}{N^k} \dim CP. \]
   We define a map $\Pi_W: X\to \boldsymbol{\mathcal{P}}(W)$ by 
   \[ \Pi_W(x) = \left(\frac{\left|W\cap \left(n+[0,N)^k\right)\right|}{N^k}\cdot \Pi(T^n x)\right)_{n\in [-R,R)^k\cap N\mathbb{Z}^k}. \]   
   $\Pi_W$ is an $\varepsilon$-embedding with respect to the distance $d_{\mathrm{Int} W \cap \mathbb{Z}^k}$, where 
   $\mathrm{Int} W$ is the interior of $W$.  
   It follows from Lemma \ref{lemma: how many nonzero w_m} that for any $x\in X$ and $n\in \mathbb{Z}^k$
   with $\mathrm{Int}_{L_0} W(x,n) \neq \emptyset$ 
\begin{equation}  \label{eq: dimension of P(W) Q(i)}
   \dim \left( \boldsymbol{\mathcal{P}}\left(-n+W(x,n)\right) \times \boldsymbol{\mathcal{Q}}\left(\#(x,n)\right) \right)
   < \frac{\rho_1\dots \rho_k}{2} \left|\mathrm{Int}_{r_2+N\sqrt{k}}W(x,n)\right|. 
\end{equation}

Recall that we introduced lattices $\Gamma$ and $\Gamma_1$ of $\mathbb{R}^k$ in Subsection \ref{subsection: fixing some notations}.
Applying Lemma \ref{lemma: approximation lemma} to the $\varepsilon$-embedding $\Pi:(X,d_{[N]})\to P$ and
a map 
\[ X\to \mathbb{R}^{\Gamma\cap [0,N)^k}, \quad x\mapsto \left(f(x)(\lambda)\right)_{\lambda\in \Gamma \cap [0,N)^k}, \]
we can find a simplicial map $F:P\to \mathbb{R}^{\Gamma\cap [0,N)^k}$ satisfying 
\[  \forall x\in X,\lambda\in \Gamma\cap [0,N)^k:  \quad 
     \left|F\left(\Pi(x)\right)(\lambda) - f(x)(\lambda)\right| < \delta'. \]
We extend $F$ over $CP$ by $F(tp) = tF(p)$.
We define a simplicial map $G:\boldsymbol{\mathcal{P}}\times \boldsymbol{\mathcal{Q}} \to \mathbb{R}^{\Gamma\cap [-R,R)^k}$
as follows: For $p=(p_n)_{n\in [-R,R)^k\cap N\mathbb{Z}^k}\in \boldsymbol{\mathcal{P}}$, $q\in \boldsymbol{\mathcal{Q}}$, 
$n\in [-R,R)^k\cap N\mathbb{Z}^k$ and $\lambda\in \Gamma\cap [0,N)^k$
\[   G(p, q)(n+\lambda) = F(p_n)(\lambda). \]  
Using Lemma \ref{lemma: embedding lemma} (2), we perturb $G$ and construct a simplicial map 
$\mathbb{G}: \boldsymbol{\mathcal{P}}\times \boldsymbol{\mathcal{Q}}\to \mathbb{R}^{\Gamma\cap [-R,R)^k}$ satisfying the 
following.

\begin{property}[Properties on $\mathbb{G}$] \label{condition: perturbation of G}
    \begin{enumerate}
    \item  $\mathbb{G}$ is sufficiently close to $G$ in the sense that for all 
             $(p,q)\in \boldsymbol{\mathcal{P}}\times \boldsymbol{\mathcal{Q}}$ and $\lambda\in \Gamma\cap [-R,R)^k$
            \[   \left|\mathbb{G}(p,q)(\lambda) - G(p,q)(\lambda)\right| 
                 < \delta' - \sup_{x\in X, \lambda\in \Gamma\cap [0,N)^k} \left|F\left(\Pi(x)\right)(\lambda) - f(x)(\lambda)\right|. \]
            This implies that for all convex sets $W\subset [-R,R)^k$, $x\in X$ and $q\in \boldsymbol{\mathcal{Q}}$
            \[ \forall \lambda\in \Gamma\cap \mathrm{Int}_{N\sqrt{k}} W: \quad 
                \left|\mathbb{G}\left(\Pi_W(x), q\right)(\lambda) - f(x)(\lambda) \right| < \delta'. \]
    \item If a convex set $W\subset [-R,R)^k$, $i\geq 0$ and a subset $\Lambda\subset \Gamma\cap [-R,R)^k$ satisfy
            \[  \dim\left(\boldsymbol{\mathcal{P}}(W)\times \boldsymbol{\mathcal{Q}}(i)\right) < \frac{|\Lambda|}{2}, \]
            then the map 
            \[ \boldsymbol{\mathcal{P}}(W)\times \boldsymbol{\mathcal{Q}}(i) \to \mathbb{R}^\Lambda, \quad 
                (p,q) \mapsto \mathbb{G}(p,q)|_{\Lambda} \]
            is an embedding.    
    \end{enumerate}
\end{property}

Now we are ready to prove the main proposition.

\begin{proof}[Proof of Proposition \ref{prop: main proposition}]
We choose continuous functions $\beta_1,\beta_2:\mathbb{R}\to [0,1]$ satisfying 
\[ \beta_1(t)=0 \quad (t\leq r_0), \quad \beta_1(t) = 1 \quad (t\geq r_0+1), \]
\[ \beta_2(t) = 0 \quad \left(t\leq N\sqrt{k}\right), \quad \beta_2(t) = 1 \quad \left(t\geq r_0+ N\sqrt{k}\right). \]
Let $x\in X$.
Recall that we introduced $\partial(x) = \bigcup_{n\in \mathbb{Z}^k} \partial W(x,n)$.
We define $p_x:\Gamma_1\to [0,1]$ and 
$u_x = \left(u_x(\lambda)\right)_{\lambda\in \Gamma_1}\in \ell^\infty(\Gamma_1)$ as follows: For $\lambda\in \Gamma_1$
\begin{itemize}
    \item If there exists $n\in \mathbb{Z}^k$ satisfying $\lambda\in (n+\Gamma)\cap W(x,n)$ then 
            \begin{equation*}
                \begin{split}
                 p_x(\lambda) = & \beta_1\left(d(\lambda, \partial(x))\right),     \\
                 u_x(\lambda) = &\beta_2\left(d(\lambda, \partial(x))\right) \{- f(x)(\lambda) \\
                & +\mathbb{G}\left(\Pi_{-n+W(x,n)}(T^n x), \left(w_m(x,n)\cdot \pi(T^{n+m}x)\right)_{m\in [-R,R)^k\cap \mathbb{Z}^k}\right)
                (\lambda-n) \}.
               \end{split}
            \end{equation*}    
    \item Otherwise we set $p_x(\lambda)=0$ and $u_x(\lambda)=0$.
\end{itemize}
$p_x$ is an \textit{admissible} function in the sense of Definition \ref{def: admissible set and function} (2).
It follows from Property \ref{condition: perturbation of G} (1) with $f(x)(\lambda) = f(T^n x)(\lambda-n)$ that
\begin{equation} \label{eq: norm of u_x}
    \norm{u_x}_\infty  < \delta'. 
\end{equation}    
Since the tiles $W(x,n)$ and the weight functions $w(x,n)$ depend continuously on $x\in X$, the numbers $p_x(\lambda)$ and 
$u_x(\lambda)$ depend continuously on $x$ for each fixed $\lambda\in \Gamma_1$.
We set 
\[   g_1(x) = f(x) + \Psi(p_x, u_x) \in \mathcal{B}(a_1,\dots,a_k) \]
where $\Psi$ is the interpolating function introduced in Section
\ref{section: Interpolating functions}.
By the continuity of $\Psi$ (Proposition \ref{prop: continuity of Psi}), $g_1(x)$ depends continuously on $x\in X$.
Since the tiles $W(x,n)$ and the weight functions $w(x,n)$ satisfy the natural $\mathbb{Z}^k$-equivariance, 
$p_x(\lambda)$ and $u_x(\lambda)$ also satisfy the $\mathbb{Z}^k$-equivariance 
(i.e. $p_{T^n x}(\lambda) = p_x(\lambda+n)$ and $u_{T^n x}(\lambda) = u_x(\lambda+n)$).
The interpolating function $\Psi$ also satisfies it (Lemma \ref{lemma: basic properties of Psi} (3)).
Therefore the map $X\ni x\mapsto g_1(x)\in \mathcal{B}_1(a_1,\dots,a_k)$ is $\mathbb{Z}^k$-equivariant.
By Lemma \ref{lemma: basic properties of Psi} (2) and (\ref{eq: norm of u_x})
\begin{equation} \label{eq: difference between g_1 and f}
  \begin{split}
       \norm{g_1(x)-f(x)}_{L^\infty(\mathbb{R}^k)} & = \norm{\Psi(p_x,u_x)}_{L^\infty(\mathbb{R}^k)} \\
       & < 2K_0 \delta' < \frac{\delta}{2} \quad 
       \left(\text{we assumed $4K_0\delta' < \delta$ in (\ref{eq: delta'})}\right).
  \end{split} 
\end{equation}

Let $\Phi_{A,L_0,L}(x) = (\Phi(x)_1,\dots, \Phi(x)_k):\mathbb{C}^k\to \mathbb{C}^k$ be the tiling-like band-limited map
introduced in Subsection \ref{subsection: weight functions from tiling-like band-limited functions} for the parameters $A, L_0, L$ in 
Notation \ref{notation: fixing A,L_0,L}.
We define $g_2(x):\mathbb{R}^k\to \mathbb{R}$ by 
\[  g_2(x) = \frac{\delta}{2kK_1}\mathrm{Re}\left(\sum_{i=1}^k \Phi(x)_i\right). \]
It follows from $\norm{\Phi(x)}_{L^\infty(\mathbb{R}^k)} \leq \sqrt{k}K_1$ (Lemma \ref{lemma: basic properties of Phi(x)} (1)) that
\begin{equation}  \label{eq: norm of g_2}
   \norm{g_2(x)}_{L^\infty(\mathbb{R}^k)} \leq \frac{\delta}{2}.
\end{equation}
By Lemma \ref{lemma: basic properties of Phi(x)} (2) with $L>4/\delta$,
the Fourier transform of each $\Phi(x)_i|_{\mathbb{R}^k}$ is supported in 
 \[   \Delta_i \overset{\mathrm{def}}{=}  \left(-\frac{\delta}{4},\frac{\delta}{4}\right)^{i-1} \times 
                                   \left(\frac{a_i}{2}, \frac{a_i}{2}+\frac{\delta}{2}\right)\times \left(-\frac{\delta}{4},\frac{\delta}{4}\right)^{k-i}, 
     \quad   (1\leq i\leq k).                               \]
Then the Fourier transform of $g_2(x)$ is supported in 
\begin{equation} \label{eq: spectrum of g_2}
     \bigcup_{i=1}^k \Delta_i \cup  (-\Delta_i), \quad (-\Delta_i = \{-\xi|\, \xi\in \Delta_i\}). 
\end{equation}     
This is contained in $\prod_{i=1}^k [-(a_i+\delta)/2, (a_i+\delta)/2]$.
So $g_2(x)\in \mathcal{B}(a_1+\delta, \dots, a_k+\delta)$.
The map $X\ni x\mapsto g_2(x)\in \mathcal{B}(a_1+\delta,\dots, a_k+\delta)$ is continuous and $\mathbb{Z}^k$-equivariant by 
the corresponding properties of $\Phi(x)$ in Lemma \ref{lemma: basic properties of Phi(x)} (3) and (4).

Finally we set $g(x) = g_1(x)+g_2(x)$.
By (\ref{eq: difference between g_1 and f}) and (\ref{eq: norm of g_2})
\[  \norm{f(x)-g(x)}_{L^\infty(\mathbb{R}^k)} < \delta. \]
In particular $\norm{g(x)}_{L^\infty(\mathbb{R}^k)} < 1$ by the assumption $\norm{f(x)}_{L^\infty(\mathbb{R}^k)} \leq 1-\delta$
in the beginning of the section.
The map 
\[  X\ni x\mapsto g(x) \in \mathcal{B}_1(a_1+\delta,\dots, a_k+\delta) \]
is continuous and $\mathbb{Z}^k$-equivariant.

We need to prove that $g:X\to \mathcal{B}_1(a_1+\delta,\dots,a_k+\delta)$ is a $\delta$-embedding with respect to $d$.
Suppose $g(x)=g(y)$ for some $x,y\in X$.
We would like to show $d(x,y)<\delta$.
It follows from $g(x)=g(y)$ that $g_1(x)=g_1(y)$ and $g_2(x)=g_2(y)$
because the Fourier transforms of $g_1(x)$ and $g_1(y)$ are supported in $\prod_{i=1}^k [-a_i/2,a_i/2]$, 
which is disjoint with (\ref{eq: spectrum of g_2}).
The equation $g_2(x)=g_2(y)$ implies $\Phi(x)_i = \Phi(y)_i$ for all $1\leq i\leq k$ because 
$2k$ sets $\Delta_i$ and $-\Delta_i$ are disjoint with each other.
So $\Phi(x)=\Phi(y)$.
Then it follows that
$W(x,n)=W(y,n)$ and $w(x,n)=w(y,n)$ for all $n\in \mathbb{Z}^k$
because the tiles $W(x,n)$ and the weight functions $w(x,n)$ are constructed from the tiling-like band-limited map $\Phi(x)$.

\textbf{Case 1}: Suppose $d\left(0,\partial(x)\right) > L_0$.
Then there exists $n\in \mathbb{Z}^k$ satisfying $0\in \mathrm{Int}_{L_0} W(x,n)$.
For all $\lambda\in (n+\Gamma) \cap \mathrm{Int}_{r_0+N\sqrt{k}} W(x,n)$
\begin{equation*}
   \begin{split}
   & p_x(\lambda) = 1, \\
   &  u_x(\lambda) = - f(x)(\lambda) 
                 +\mathbb{G}\left(\Pi_{-n+W(x,n)}(T^n x), \left(w_m(x,n)\cdot \pi(T^{n+m}x)\right)_{m\in [-R,R)^k\cap \mathbb{Z}^k}\right)
                (\lambda-n).
   \end{split}
\end{equation*}   
It follows from the interpolation property of $\Psi$ in Lemma \ref{lemma: basic properties of Psi} (1) that
for  all $\lambda\in (n+\Gamma) \cap \mathrm{Int}_{r_0+N\sqrt{k}} W(x,n)$
\begin{equation} \label{eq: formula of g_1 in a good region}
  \begin{split}
    &g_1(x)(\lambda) 
    =\mathbb{G}\left(\Pi_{-n+W(x,n)}(T^n x), \left(w_m(x,n)\cdot \pi(T^{n+m}x)\right)_{m\in [-R,R)^k\cap \mathbb{Z}^k}\right)
                (\lambda-n), \\
    &g_1(y)(\lambda) 
    =\mathbb{G}\left(\Pi_{-n+W(y,n)}(T^n y), \left(w_m(y,n)\cdot \pi(T^{n+m}y)\right)_{m\in [-R,R)^k\cap \mathbb{Z}^k}\right)
                (\lambda-n).     
   \end{split}
\end{equation}                
From $\mathrm{Int}_{L_0} W(x,n)\neq \emptyset$ (because $0\in \mathrm{Int}_{L_0}W(x,n)$) 
and (\ref{eq: dimension of P(W) Q(i)}) 
\begin{equation*}
     \begin{split}
       \dim \left( \boldsymbol{\mathcal{P}}\left(-n+W(x,n)\right) \times \boldsymbol{\mathcal{Q}}\left(\#(x,n)\right) \right)
        <&  \frac{\rho_1\dots \rho_k}{2} \left|\mathrm{Int}_{r_2+N\sqrt{k}}W(x,n)\right| \\
        \leq &     \frac{1}{2}\#\left(\Gamma \cap \left(-n+\mathrm{Int}_{r_0+N\sqrt{k}} W(x,n)\right)\right),  \\
        &  \left(\text{by $r_2 = r_0+\sqrt{\sum_{i} (1/\rho_i)^2}$ in (\ref{eq: definition of r_2})}\right).
     \end{split}
\end{equation*}     
Therefore by Property \ref{condition: perturbation of G} (2) the map 
\begin{equation} \label{eq: embedding map in case 1}
   \begin{split}
    \boldsymbol{\mathcal{P}}\left(-n+ W(x,n)\right)&  \times  \boldsymbol{\mathcal{Q}}\left(\#(x,n)\right)
    \to \mathbb{R}^{\Gamma \cap \left(-n+\mathrm{Int}_{r_0+N\sqrt{k}} W(x,n)\right)} \\
    &(p,q) \mapsto \mathbb{G}(p,q)|_{ \Gamma \cap \left(-n+\mathrm{Int}_{r_0+N\sqrt{k}} W(x,n)\right)}
   \end{split} 
\end{equation}    
is an embedding.
Thus $g_1(x)=g_1(y)$ implies that $\Pi_{-n+W(x,n)}(T^n x) = \Pi_{-n+W(y,n)}(T^n y)$.
Recall that the map $\Pi_{-n+W(x,n)}$ is an $\varepsilon$-embedding with respect to the distance 
$d_{-n+ \mathrm{Int} W(x,n)}$.
It follows from $0\in \mathrm{Int}_{L_0}W(x,n)$ that $-n \in -n + \mathrm{Int}W(x,n)$.
Thus we get 
\[    d(x,y) \leq d_{-n+ \mathrm{Int} W(x,n)}(T^n x, T^n y) <\varepsilon < \delta.   \]

\textbf{Case 2}: Suppose $d\left(0,\partial(x)\right) \leq L_0$.
It follows from Proposition \ref{prop: construction of weight functions} (3) and (4) that there exists $n\in B_{R_0}\cap \mathbb{Z}^k$
satisfying $\mathrm{Int}_{L_0} W(x,n)\neq \emptyset$ and $w_{-n}(x,n)= w_{-n}(y,n) = 1$.
As in Case 1, we have the same formula (\ref{eq: formula of g_1 in a good region}) for 
all $\lambda\in (n+\Gamma) \cap \mathrm{Int}_{r_0+N\sqrt{k}} W(x,n)$.
The map (\ref{eq: embedding map in case 1}) is also an embedding.
Then the equation $g_1(x)=g_1(y)$ implies that 
\[  \left(w_m(x,n)\cdot \pi(T^{n+m}x)\right)_{m\in [-R,R)^k\cap \mathbb{Z}^k} 
    = \left(w_m(y,n)\cdot \pi(T^{n+m}y)\right)_{m\in [-R,R)^k\cap \mathbb{Z}^k}. \]
Since $w_{-n}(x,n) = w_{-n}(y,n)=1$, we get $\pi(x)=\pi(y)$.
Thus $d(x,y)<\varepsilon <\delta$ because $\pi:X\to Q$ is an $\varepsilon$-embedding with respect to $d$.
\end{proof}

Now we have completed the proof of Main Theorem \ref{main theorem: continuous signal} 
(and hence the proof of Main Theorem \ref{main theorem: discrete signal}).

\section{Open problems}   \label{section: open problems}
Here we explain major open problems in the direction of the paper.

\begin{enumerate}
  \item From Main Theorem \ref{main theorem: discrete signal}, we have a good understanding of when 
  dynamical systems can be embedded in $\left([0,1]^D \right)^{\mathbb{Z}^k}$ under the assumption of the marker property.
  So the next step is how to remove the assumption.
  We would like to propose

  \begin{conjecture}  \label{conjecture: main embedding conjecture}
   Let $(X,\mathbb{Z}^k,T)$ be a dynamical system. For a subgroup $A \subset \mathbb{Z}^k$ we define $X_A\subset X$ as
   the space of $x\in X$ satisfying $T^a x=x$ for all $a\in A$. 
   The quotient group $\mathbb{Z}^k/A$ naturally acts on $X_A$.
   Then $X$ can be embedded in the shift on $\left([0,1]^D \right)^{\mathbb{Z}^k}$
  if for every subgroup $A\subset \mathbb{Z}^k$ the mean dimension of $X_A$ with respect to the $(\mathbb{Z}^k/A)$-action 
  is smaller than $D/2$.
  \end{conjecture}

  If $\mathbb{Z}^k/A$ is a finite group, then the mean dimension of $X_A$ is just 
  \[ \frac{\dim X_A}{\#(\mathbb{Z}^k/A)}. \]
  When $k=1$, Conjecture \ref{conjecture: main embedding conjecture} is equivalent to \cite[Conjecture 1.2]{Lindenstrauss--Tsukamoto}. 
  Notice that if we set $Y=\left([0,1]^D \right)^{\mathbb{Z}^k}$ then for any subgroup $A\subset \mathbb{Z}^k$ the 
  system $Y_A$ is naturally identified with $\left([0,1]^D \right)^{\mathbb{Z}^k/A}$, whose mean dimension is $D$.
  So Conjecture \ref{conjecture: main embedding conjecture} can be rephrased; if $(X,\mathbb{Z}^k,T)$ satisfies
  \[ \forall \text{ subgroup } A\subset \mathbb{Z}^k: \quad 
     \mdim(X_A) < \frac{\mdim(Y_A)}{2},  \]
then it can be embedded in $Y=\left([0,1]^D \right)^{\mathbb{Z}^k}$.

  \item   Let $\Omega\subset \mathbb{R}^k$ be a compact subset.
  We define $\mathcal{B}_1(\Omega)$ as the space of continuous functions $f:\mathbb{R}^k\to \mathbb{C}$ satisfying 
   $\supp \hat{f}\subset \Omega$ and $\norm{f}_{L^\infty(\mathbb{R}^k)}\leq 1$.
  The mean dimension of the natural $\mathbb{Z}^k$-action on $\mathcal{B}_1(\Omega)$ is equal to $2|\Omega|$.
  (Here notice that $f$ are complex-valued functions. The factor $2$ comes from $\dim \mathbb{C}=2$.)
  We would like to ask when a dynamical system $(X,\mathbb{Z}^k,T)$ can be embedded in the shift on $\mathcal{B}_1(\Omega)$.
  By almost the same argument as the proof of Main Theorem \ref{main theorem: continuous signal}, we can prove

  \begin{theorem}  \label{thm: embedding in continuous signals}
  Suppose $\Omega$ is a rectangle (i.e. congruent to a set $[a_1,b_1]\times \dots\times [a_k,b_k]$).
  If a dynamical system $(X,\mathbb{Z}^k,T)$ satisfies the marker property and 
  \[ \mdim(X) < |\Omega| = (b_1-a_1)\times \dots \times (b_k-a_k), \]
  then we can embed it in the shift on $\mathcal{B}_1(\Omega)$.
  \end{theorem}

  The proof of this theorem is notationally more messy 
  than the proof of Main Theorem \ref{main theorem: continuous signal}.
  So we concentrate on Main Theorem \ref{main theorem: continuous signal}
  in this paper.
  
  If $\Omega$ is not a rectangle, then the method of this paper does not work directly.
  Nevertheless, it seems reasonable to conjecture that if $\Omega$ is a ``nice'' set 
  (e.g. a semi-algebraic set) and if a dynamical system $(X,\mathbb{Z}^k,T)$ satisfies the marker property and  
  \[  \mdim(X) < |\Omega| \]
  then $X$ can be embedded in the shift on $\mathcal{B}_1(\Omega)$.
  We hope to return to this problem in a future.

  \item An ambitious question is how to generalize Main Theorems \ref{main theorem: discrete signal} and \ref{main theorem: continuous signal}
  to the actions of noncommutative groups.
  In particular it is very interesting to see what kind of signal analysis should be involved.
  We don't have an answer even for nilpotent groups.
\end{enumerate}

\vspace{0.5cm}

\address{Yonatan Gutman \endgraf
Institute of Mathematics, Polish Academy of Sciences,
ul. \'{S}niadeckich~8, 00-656 Warszawa, Poland}

\textit{E-mail address}: \texttt{y.gutman@impan.pl}

\vspace{0.5cm}

\address{Yixiao Qiao \endgraf
Institute of Mathematics, Polish Academy of Sciences, ul. \'{S}niadeckich 8, 00-656 Warszawa, Poland
-- \and --
Department of Mathematics, University of Science and Technology of China, Hefei, Anhui 230026, P.R. China}

\textit{E-mail address}: \texttt{yxqiao@mail.ustc.edu.cn}

\vspace{0.5cm}

\address{ Masaki Tsukamoto \endgraf
Department of Mathematics, Kyoto University, Kyoto 606-8502, Japan}

\textit{E-mail address}: \texttt{tukamoto@math.kyoto-u.ac.jp}


\begin{thebibliography}{99}




\bibitem[Aus88]{Auslander}
J. Auslander,
Minimal flows and their extensions,
North-Holland, Amsterdam, 1988.





\bibitem[Beu89]{Beurling}
Collected works of Arne Beurling, Vol. 2, Harmonic Analysis,
Edited by L. Carleson, P. Malliavan, J. Neuberger, and J. Wermer,  Birkh\"{a}user, Boston-Basel-Berlin, 1989.





\bibitem[CS72]{Cornalba--Shiffman}
M. Cornalba, B. Shiffman,
A couterexample to the ``Transcendental Bezout Problem''.
Ann. of Math. \textbf{96} (1972) 402-406.








\bibitem[CT06]{Cover--Thomas}
T.M. Cover, J.A. Thomas,
Elements of information theory, second edition, 
Wiley, New York, 2006.










\bibitem[DM72]{Dym--McKean}
H. Dym, H.P. McKean, 
Fourier series and integrals, 
Academic Press, New York-London, 1972.





\bibitem[Gro99]{Gromov}
M. Gromov, 
Topological invariants of dynamical systems and spaces of holomorphic maps: I,
Math. Phys. Anal. Geom. \textbf{2} (1999) 323-415.



\bibitem[Gut12]{Gutman 2}
Y. Gutman, Mean dimension and Jaworski-type theorems, 
Proceedings of the London Mathematical Society \textbf{111(4)} (2015) 831-850.



\bibitem[GLT16]{Gutman--Lindenstrauss--Tsukamoto}
Y. Gutman, E. Lindenstrauss, M. Tsukamoto, 
Mean dimension of $\mathbb{Z}^k$-actions, 
Geom. Funct. Anal.
\textbf{26} Issue 3 (2016) 778-817.




\bibitem[GT]{Gutman--Tsukamoto minimal}
Y. Gutman , M. Tsukamoto, 
Embedding minimal dynamical systems into Hilbert cubes,
preprint, arXiv:1511.01802.




\bibitem[Hay64]{Hayman}
W.K. Hayman, 
Meromorphic functions, 
Clarendon Press, Oxford, 1964.




\bibitem[HW41]{Hurewicz--Wallman}
W. Hurewicz, H. Wallman,
Dimension theory,
Princeton University Press, 1941.



\bibitem[Hat02]{Hatcher}
A. Hatcher,
Algebraic topology,
Cambridge University Press, 2002.




\bibitem[Jaw74]{Jaworski}
A. Jaworski, Ph.D. Thesis, University of Maryland (1974).




\bibitem[Lig03]{Lightwood 1}
S.J. Lightwood,  Morphisms from non-periodic $\mathbb{Z}^2$-subshifts. I. Constructing
 embeddings from homomorphisms,
 Ergodic Theory Dynam. Systems  \textbf{23} (2003) 587--609.
		


\bibitem[Lig04]{Lightwood 2}
S.J. Lightwood,  Morphisms from non-periodic $\mathbb{Z}^2$ subshifts. II. Constructing
 homomorphisms to square-filling mixing shifts of finite type.
 Ergodic Theory Dynam. Systems  \textbf{24}  (2004) 1227--1260.



\bibitem[Lin99]{Lindenstrauss}
E. Lindenstrauss,
Mean dimension, small entropy factors and an embedding theorem,
Inst. Hautes \'{E}tudes Sci. Publ. Math. \textbf{89} (1999) 227-262.





\bibitem[LT14]{Lindenstrauss--Tsukamoto}
E. Lindenstrauss, M. Tsukamoto,
Mean dimension and an embedding problem: an example,
Israel J. Math. \textbf{199} (2014) 573-584.






\bibitem[LW00]{Lindenstrauss--Weiss}
E. Lindenstrauss, B. Weiss,
Mean topological dimension,
Israel J. Math. \textbf{115} (2000) 1-24.





\bibitem[NW14]{Noguchi--Winkelmann}
J. Noguchi, J. Winkelmann, 
Nevanlinna theory in several complex variables and Diophantine approximation, 
Springer, Tokyo, 2014.



\bibitem[OU12]{Olevskii--Ulanovskii}
A. Olevskii, A. Ulanovskii, 
On multi-dimensional sampling and interpolation,
Anal. Math. Phys. \textbf{2} (2012) 149-180



\bibitem[Sch66]{Schwartz}
L. Schwartz, 
Th\'{e}orie des distributions, 
Hermann, Paris, 1966.




\bibitem[Sha48]{Shannon}
C.E. Shannon,
A mathematical theory of communication, 
Bell Syst. Tech. J. \textbf{27} (1948) 379-423, 623-656.



\bibitem[Spa66]{Spanier}
E.H. Spanier,
Algebraic topology,
Springer-Verlag, New York, 1966.




\end{thebibliography}
\end{document}